\documentclass{article}
\usepackage{delarray}
\usepackage{bigdelim}
\usepackage{amsmath,amssymb,graphicx,enumerate,latexsym,tabularx,xcolor,theorem}
\usepackage{algorithm}
\usepackage{algpseudocode}
\usepackage{arxiv}
\usepackage{comment}
\usepackage[utf8]{inputenc} 
\usepackage[T1]{fontenc}    
\usepackage{hyperref}       
\usepackage{url}            
\usepackage{booktabs}       
\usepackage{amsmath}
\usepackage{amssymb}
\usepackage{amsfonts}       
\usepackage{nicefrac}       
\usepackage{microtype}      
\usepackage{xcolor}
\usepackage{graphicx}
\usepackage{stmaryrd}
\SetSymbolFont{stmry}{bold}{U}{stmry}{m}{n} 
\usepackage{bm}
\usepackage[normalem]{ulem}
\usepackage{cancel}
\usepackage[tickmarkheight=0.1cm,textsize=tiny]{todonotes}
\usepackage{algorithmicx}
\usepackage{algorithm}
\usepackage{algpseudocode}
\usepackage{caption}
\usepackage{subcaption}
\usepackage{lineno}
\usepackage{orcidlink}
\usepackage[inline,final]{showlabels}
\usepackage[tickmarkheight=0.1cm,textsize=tiny]{todonotes}

\usepackage{xpatch}
\makeatletter
\xpatchcmd{\algorithmic}
{\ALG@tlm\z@}{\leftmargin\z@\ALG@tlm\z@}
{}{}
\makeatother

\hypersetup{
colorlinks=true, 
linkcolor=blue, 
urlcolor=red, 
citecolor=magenta,
linktoc=all 
}
\newcommand{\assign}{:=}
\newcommand{\cdummy}{\cdot}
\newcommand{\mathLaplace}{\Delta}
\newcommand{\mathd}{\mathrm{d}}
\newcommand{\nocomma}{}

\newcommand{\tmmathbf}[1]{\ensuremath{\boldsymbol{#1}}}
\newcommand{\tmop}[1]{\ensuremath{\operatorname{#1}}}

\newcommand{\tmrsup}[1]{\textsuperscript{#1}}
\newcommand{\tmsamp}[1]{\textsf{#1}}
\newcommand{\tmtextbf}[1]{\text{{\bfseries{#1}}}}
\newcommand{\tmtextit}[1]{\text{{\itshape{#1}}}}
\newcommand{\tmtextrm}[1]{\text{{\rmfamily{#1}}}}
\newcommand{\tmtexttt}[1]{\text{{\ttfamily{#1}}}}
\newcommand{\tmverbatim}[1]{\text{{\ttfamily{#1}}}}
\newtheorem{definition}{Definition}
\usepackage[T3,T1]{fontenc}
\DeclareSymbolFont{tipa}{T3}{cmr}{m}{n}
\DeclareMathAccent{\invbreve}{\mathalpha}{tipa}{16}
{\theorembodyfont{\rmfamily}\newtheorem{remark}{Remark}}

\newcommand{\avg}[1]{\overline{#1}}

\newcommand{\myvector}[1]{\mathsf{#1}}

\newcommand{\vu}{\myvector{u}}
\newcommand{\vA}{\myvector{A}}
\newcommand{\vAone}{\myvector{A}^{(1)}}
\newcommand{\vD}{\myvector{D}}
\newcommand{\vH}{\myvector{H}}
\newcommand{\vT}{\myvector{T}}
\newcommand{\vTone}{\myvector{T}^{(1)}}
\newcommand{\vI}{\myvector{I}}
\newcommand{\vU}{\myvector{U}}
\newcommand{\vUone}{\myvector{\uUone}}

\newcommand{\vf}{\myvector{f}}

\newcommand{\vs}{{\myvector{s}}}
\newcommand{\vF}{\myvector{F}}
\newcommand{\vFone}{\myvector{\Fone}}
\newcommand{\vFtwo}{\myvector{\Ftwo}}
\newcommand{\vSone}{\myvector{\Sone}}
\newcommand{\vStwo}{\myvector{\Stwo}}

\newcommand{\vb}{\myvector{b}}
\newcommand{\vR}{\myvector{R}}

\newcommand{\vV}{\myvector{V}}

\newcommand{\extrapolate}{\text{\tmtextbf{AE}}}
\newcommand{\evaluate}{\tmtextbf{EA}}
\newcommand{\pdv}[2]{\frac{\partial #1}{\partial #2}}

\newcommand{\dv}[2]{\frac{\mathd #1}{\mathd #2}}

\newcommand{\bzero}{\tmmathbf{0}}

\newcommand{\bu}{\tmmathbf{u}}

\newcommand{\bx}{\ensuremath{\tmmathbf{x}}}

\newcommand{\poly}{\mathbb{P}}

\newcommand{\emh}{{e - \frac{1}{2}}}

\newcommand{\eph}{{e + \frac{1}{2}}}

\newcommand{\nph}{n + \frac{1}{2}}
\newcommand{\Nmh}{N - \frac{1}{2}}
\newcommand{\Nph}{N + \frac{1}{2}}
\newcommand{\half}{\frac{1}{2}}
\newcommand{\ud}{\text{\tmtextrm{d}}}
\newcommand{\pd}[2]{\frac{\partial #1}{\partial #2}}
\newcommand{\od}[2]{\frac{\ud #1}{\ud #2}}

\newcommand{\Uad}{\ensuremath{\mathcal{U}}_{\textrm{\tmop{ad}}}}

\newcommand{\uu}{\tmmathbf{u}}

\newcommand{\bw}{\tmmathbf{u}}
\newcommand{\re}{\mathbb{R}}

\newcommand{\utilow}{\invbreve{\boldsymbol{u}}^{\text{low},n+1}}
\newcommand{\au}{\avg{\uu}}
\newcommand{\uU}{\tmmathbf{U}}

\newcommand{\pf}{\tmmathbf{f}}

\newcommand{\ff}{\pf}

\newcommand{\F}{\tmmathbf{F}}

\newcommand{\bss}{\boldsymbol{s}}

\newcommand{\uep}{\tmmathbf{u}_{e, p}}

\newcommand{\uez}{\tmmathbf{u}_{e, 0}}

\newcommand{\uepoz}{\tmmathbf{u}_{e + 1, 0}}
\newcommand{\ueN}{\tmmathbf{u}_{e, N}}
\newcommand{\xep}{x^e_p}

\newcommand{\pph}{p + \frac{1}{2}}
\newcommand{\pmh}{p - \frac{1}{2}}

\newcommand{\paragraphtoc}[1]{}
\newcommand{\ad}{p}

\newcommand{\uus}{\tmmathbf{u}^{\ast}}
\newcommand{\Fone}{\tmmathbf{F}}
\newcommand{\uUone}{\uU}

\newcommand{\uUtwom}{\uU^{\ast-}}
\newcommand{\uUtwop}{\uU^{\ast+}}
\newcommand{\Ftwo}{\tmmathbf{F}^{\ast}}
\newcommand{\Ftwop}{\tmmathbf{F}^{\ast+}}
\newcommand{\Ftwom}{\tmmathbf{F}^{\ast-}}

\newcommand{\Ftwod}{\tmmathbf{F}^{\ast\delta}}
\newcommand{\FtwoHO}{\tmmathbf{F}^{\text{HO}\ast}}
\newcommand{\Sone}{\mathbf{S}}
\newcommand{\Stwo}{\mathbf{S}^*}

\newcommand{\bL}{\tmmathbf{L}}

\newcommand{\vfs}{\myvector{f}^*}
\newcommand{\vfsone}{\myvector{f}^{*(1)}}
\newcommand{\vus}{\myvector{u}^*}
\newcommand{\vusone}{\myvector{u}^{*(1)}}

\newcommand{\cH}{\tmverbatim{H}}
\newcommand{\cV}{\tmverbatim{V}}

\newtheorem{theorem}{Theorem}
\newenvironment{proof}{\noindent\textbf{Proof\ }}{\hspace*{\fill}$\Box$\medskip}

\newcommand{\correction}[1]{\textcolor{black}{#1}}

\begin{document}

\title{Multi-Derivative Runge-Kutta Flux Reconstruction for hyperbolic
conservation laws}

\author{
Arpit~Babbar \orcidlink{0000-0002-9453-370X} \\
Centre for Applicable Mathematics\\
Tata Institute of Fundamental Research\\
Bangalore -- 560065\\
\texttt{arpit@tifrbng.res.in} \\
\And
Praveen~Chandrashekar \orcidlink{0000-0003-1903-4107}\thanks{Corresponding author}\\
Centre for Applicable Mathematics\\
Tata Institute of Fundamental Research\\
Bangalore -- 560065\\
\texttt{praveen@math.tifrbng.res.in}
}

\date{March 4, 2024}

\maketitle

\begin{abstract}
We extend the fourth order, two stage Multi-Derivative Runge Kutta (MDRK) scheme to the Flux Reconstruction (FR) framework by writing both stages in terms of a time averaged flux and then using the approximate Lax-Wendroff procedure to compute the time averaged flux. Numerical flux is carefully constructed to enhance Fourier CFL stability and accuracy. A subcell based blending limiter is developed for the MDRK scheme which ensures that the limited scheme is provably admissibility preserving. Along with being admissibility preserving, the blending scheme is constructed to minimize dissipation errors by using Gauss-Legendre solution points and performing MUSCL-Hancock reconstruction on subcells. The accuracy enhancement of the blending scheme is numerically verified on compressible Euler's equations, with test cases involving shocks and small-scale structures.
\end{abstract}

\section{Introduction}

Higher order methods incorporate higher order terms and perform more
computations per degree of freedom, increasing accuracy and arithmetic
intensity, and making these methods relevant to the current state of
memory-bound HPC hardware~{\cite{attig2011,subcommittee2014}}. Spectral
element methods like Discontinuous Galerkin (DG) and Flux Reconstruction (FR)
are two high order methods that are suitable for modern requirements due to
their parallelization, capability of handling curved meshes and shock
capturing. Discontinuous Galerkin (DG) is a Spectral Element Method first introduced by Reed and Hill~{\cite{reed1973}} for neutron transport equations and developed for fluid dynamics equations by Cockburn and Shu and others, see~{\cite{cockburn2000}} and the references therein. The DG method uses an approximate solution which is a polynomial within each element and is allowed to be discontinuous across element interfaces. The neighbouring DG elements are coupled only through the numerical flux and thus bulk of computations are local to the element, minimizing data transfers.

Flux Reconstruction (FR) is also a class of discontinuous Spectral Element
Methods introduced by Huynh~{\cite{Huynh2007}} and uses the same solution
approximation as DG. FR method is obtained by using the numerical flux and
correction functions to construct a continuous flux approximation and
collocating the differential form of the equation. Thus, FR is quadrature free
and all local operations can be vectorized. The choice of the correction
function affects the accuracy and stability of the
method~{\cite{Huynh2007,Vincent2011a,Vincent2015,Trojak2021}}; by properly
choosing the correction function and solution points, FR method can be shown
to be equivalent to some discontinuous Galerkin and spectral difference
schemes~{\cite{Huynh2007,Trojak2021}}. A provably non-linearly stable FR in
split form was obtained in~{\cite{Cicchino2022a}}.

FR and DG are procedures for discretizing the equation in space and can be
used to obtain a system of ODEs in time, i.e., a semi-discretization of the
PDE. The standard approach to solve the semi-discretization is to use a high
order multi-stage Runge-Kutta method. In this approach, the spatial
discretization has to be performed in every RK stage and thus the expensive
operations of MPI communication and limiting have to be performed multiple
times per time step.

An alternative approach is to use high order single-stage solvers, which
include ADER schemes~{\cite{Titarev2002,Dumbser2008}} and Lax-Wendroff
schemes~{\cite{Qiu2003,Qiu2005b,Zorio2017,Burger2017,Carrillo2021}};
see~{\cite{babbar2022,babbar2023admissibility}} for a more in-depth review.

Lax-Wendroff schemes perform a Taylor's expansion of the solution in time to the order of the desired accuracy and compute the temporal derivatives locally. Multiderivative Runge-Kutta (MDRK) schemes also use temporal derivatives but combine them with multiple stages to obtain the desired order of accuracy. As MDRK schemes use both temporal derivatives and multiple stages, they are a generalization of LW and standard multistage RK methods~{\cite{Seal2013}}. MDRK methods typically require fewer temporal derivatives in contrast to the Lax-Wendroff schemes and fewer stages in contrast to the standard RK methods, which is what makes them promising.

For a system of ODE\correction{s} of the form
\[
\dv{\uu}{t} = \bL \left( \uu \right)
\]
which is obtained after spatial discretization of a PDE, Runge-Kutta methods make use of only the right hand side function $\bL$.  Multiderivative Runge-Kutta (MDRK)~{\cite{obrechkoff1940neue}} methods were initially developed to solve systems of ODE\correction{s} that also make use temporal derivatives of $\bL$, see~{\cite{Seal2013}} for a review of MDRK methods for solving ODE\correction{s}. They were first used for temporal discretization of hyperbolic conservation laws in~{\cite{Seal2013}} by using Weighted Essentially Non-Oscillatary (WENO)~{\cite{Shu1989}} and Discontinuous Galerkin (DG)~{\cite{cockburn2000}} methods for spatial discretization.

In~{\cite{babbar2022}}, a Lax-Wendroff Flux Reconstruction (LWFR) scheme was
proposed which used the approximate Lax-Wendroff procedure
of~{\cite{Zorio2017}} to obtain an element local high order approximation of
the time averaged flux which was made globally continuous using FR; the
globally continuous time averaged flux approximation was used to perform
evolution to the next time level. The numerical flux was carefully constructed
in~{\cite{babbar2022}} to obtain enhanced accuracy and stability.
In~{\cite{babbar2023admissibility}}, a subcell based shock capturing blending
scheme was introduced for LWFR based on the subcell based scheme
of~{\cite{hennemann2021}}. To enhance
accuracy,~{\cite{babbar2023admissibility}} used Gauss-Legendre solution points
and performed MUSCL-Hancock reconstruction on the subcells. Since the subcells
used in~{\cite{babbar2023admissibility}} were inherently non-cell centred, the
MUSCL-Hancock scheme was extended to non-cell centred grids along with the
proof of~{\cite{Berthon2006}} for admissibility preservation. The subcell
structure was exploited to obtain a provably admissibility preserving LWFR
scheme by careful construction of the \tmtextit{blended numerical flux} at the
element interfaces.

In~{\cite{li2016}}, a two stage fourth order MDRK scheme was introduced for
solving hyperbolic conservation laws by solving a Generalized Riemann Problem
(GRP). In this work, we show the first combination of MDRK with a Flux
Reconstruction scheme by using the scheme of~{\cite{li2016}}. The idea to apply Multiderivative Runge-Kutta Flux Reconstruction on the time
dependent conservation law
\[
\uu_t + \pf \left( \uu \right)_x = \bzero
\]
is to cast the fourth order multi-derivative Runge-Kutta scheme
of~{\cite{li2016}} in the form of
\begin{equation*}
\begin{split}
\uus & = \uu^n - \frac{\mathLaplace t}{2} \partial_x  \Fone
\\
\uu^{n + 1} & = \uu^n - \mathLaplace t \partial_x  \Ftwo
\end{split}
\end{equation*}
where
\[
\Fone = \Fone \left( \uu^n \right) \approx
\frac{1}{\Delta t/2}  \int_{t^n}^{t^{n + 1 / 2}} \pf  \ud t, \qquad  \Ftwo
= \Ftwo \left( \uu^n, \uus \right) \approx \frac{1}{\Delta t}
\int_{t^n}^{t^{n + 1}} \pf  \ud t
\]
The method is two-stage; in the first stage, $\Fone$ is locally approximated and then flux reconstruction~{\cite{Huynh2007}} is used to construct a globally continuous approximation of $\Fone$ which is used to perform evolution to $\uus$; and the same procedure is then performed using $\Ftwo$ for evolution to $\uu^{n + 1}$. We also use the construction of the numerical flux from~{\cite{babbar2022}}; in particular, we use the D2 dissipation introduced in~{\cite{babbar2022}} and show that it leads to enhanced Fourier CFL stability limit. We also use the {\evaluate} scheme from~{\cite{babbar2022}} which leads to enhanced accuracy for non-linear problems when using Gauss-Legendre solution points. We also develop admissibility preserving subcell based blending scheme based on~{\cite{babbar2023admissibility}} and show how it superior to other schemes like a TVB limiter.

The rest of this paper is organized as follows. The discretization of the domain and function approximation by polynomials is presented in Section~\ref{sec:mdrk.scl}. Section~\ref{sec:mdrk.rk} reviews the one dimensional Runge-Kutta Flux Reconstruction (RKFR) method and Section~\ref{sec:mdrk.mdrk} introduces the MDRK method in an FR framework. In particular, Section~\ref{sec:mdrk.alw} discusses approximate Lax-Wendroff procedure applied to MDRK, Section~\ref{sec:mdrk.numflux} discusses the D2 dissipation for computing the dissipative part of the numerical flux to enhance Fourier CFL stability limit and Sections~\ref{sec:mdrk.ae},~\ref{sec:mdrk.ea} discuss the {\evaluate} scheme for computing the central part of numerical flux to enhance stability. The Fourier stability analysis is performed in Section~\ref{sec:mdrk.fourier} to demonstrate the improved stability of D2 dissipation. In Section~\ref{sec:mdrk.blending}, we explain the admissibility preserving blending limiter for MDRK scheme. The numerical results validating the order of accuracy and capability of the blending scheme are shown in Section~\ref{sec:mdrk.num} and a summary of the new MDRK scheme is presented in Section~\ref{sec:mdrk.conclusion}.

\section{Conservation law and solution space}\label{sec:mdrk.scl}

Consider a conservation law of the form
\begin{equation}
\uu_t + \ff (\bw)_x = \bzero \label{eq:con.law}
\end{equation}
where $\uu \in \re^p$ is the vector of conserved quantities, $\ff (\bw)$ is
the corresponding flux, together with some initial and boundary conditions.
The solution that is physically correct is assumed to belong to an admissible
set, denoted by $\Uad$. For example, in case of compressible flows, the density
and pressure (or internal energy) must remain positive. In case of shallow
water equations, the water depth must remain positive. In most of the models
that are of interest, the admissible set is a convex subset of $\re^p$, and
can be written as
\begin{equation}
\label{eq:mdrk.uad.form} \Uad = \{ \uu \in \re^p : p_k (\bw) > 0, 1 \le k \le K\}
\end{equation}
where each admissibility constraint $p_k$ is a concave function if $p_j > 0$ for all $j <
k$. For Euler's equations, $K = 2$ and $p_1, p_2$ are density, pressure
functions respectively; if the density is positive then pressure is a concave
function of the conserved variables.

We will divide the computational domain
$\Omega$ into disjoint elements $\Omega_e$, with
\[
\Omega_e = [x_{\emh}, x_{\eph}] \qquad \textrm{and} \qquad \Delta x_e =
x_{\eph} - x_{\emh}
\]
Let us map each element to a reference element, $\Omega_e \to [0, 1]$, by
\[
x \to \xi = \frac{x - x_{\emh}}{\Delta x_e}
\]
Inside each element, we approximate the solution by $\poly_N$ functions which
are degree $N \geq 0$ polynomials.

\begin{remark}
In this work, we will only be using degree $N = 3$ polynomials as the
Multiderivative Runge-Kutta discretization that we use~{\cite{li2016}} is
fourth order accurate in time. However, the description is given for general $N$ for ease of
notation.
\end{remark}

To construct $\poly_N$ polynomials, choose $N + 1$ distinct nodes
\[ 0 \le \xi_0 < \xi_1 < \cdots < \xi_N \le 1 \]
which will be taken to be Gauss-Legendre (GL) or \correction{Gauss-Legendre-Lobatto} (GLL)
nodes, and will also be referred to as \tmtextit{solution points}. There are
associated quadrature weights $w_j$ such that the quadrature rule is exact for
polynomials of degree up to $2 N + 1$ for GL points and \correction{up to} degree $2 N - 1$
for GLL points. Note that the nodes and weights we use are with respect to the
interval $[0, 1]$ whereas they are usually defined for the interval $[- 1, +
1]$. The solution inside an element is given by
\[ x \in \Omega_e : \qquad u_h (\xi, t) = \sum_{p = 0}^N \uep (t) \ell_p (\xi)
\]
where each $\ell_p$ is a Lagrange polynomial of degree $N$ given by
\[ \ell_q (\xi) = \prod_{p = 0, p \ne q}^N \frac{\xi - \xi_p}{\xi_q - \xi_p}
\in \poly_N, \qquad \ell_q (\xi_p) = \delta_{p \nocomma q} \]
Figure~(\ref{fig:mdrk.solflux1}a) illustrates a piecewise polynomial solution at
some time $t_n$ with discontinuities at the element boundaries. Note that the
coefficients $\uep$ which are the basic unknowns or \tmtextit{degrees of
freedom} (dof), are the solution values at the solution points.

\begin{figure}
\centering
\begin{tabular}{cc}
\resizebox{0.40\columnwidth}{!}{\includegraphics{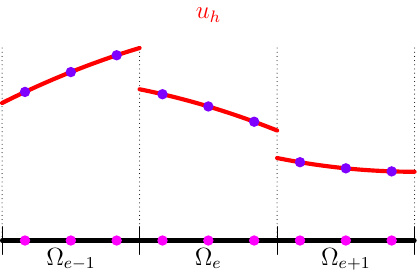}} &
\resizebox{0.40\columnwidth}{!}{\includegraphics{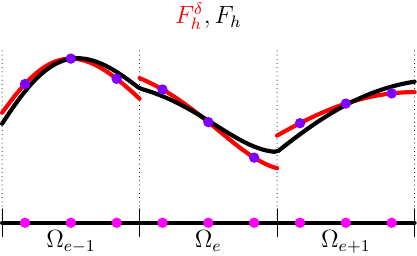}}\\
(a) & (b)
\end{tabular}
\caption{(a) Piecewise polynomial solution at time $t_n$, and (b)
discontinuous and continuous flux. \label{fig:mdrk.solflux1}}
\end{figure}

The numerical method will require spatial derivatives of certain quantities.
We can compute the spatial derivatives on the reference interval using a
differentiation matrix $\vD = [D_{p \nocomma q}]$ whose entries are given by
\[ D_{p \nocomma q} = \ell_q' (\xi_p), \qquad 0 \le p, q \le N. \]
For example, we can obtain the spatial derivatives of the solution $\uu_h$ at
all the solution points by a matrix-vector product as follows
\[ \left[ \begin{array}{c}
\partial_x  \uu_h (\xi_0, t)\\
\vdots\\
\partial_x  \uu_h (\xi_N, t)
\end{array} \right] = \frac{1}{\Delta x_e}  \vD \vu (t), \qquad \vu =
\left[ \begin{array}{c}
\uez\\
\vdots\\
\ueN
\end{array} \right] \]
We will use symbols in sans serif font like $\vD, \vu$, etc. to denote
matrices or vectors defined with respect to the solution poi{\tmsamp{}}nts.
Define the Vandermonde matrices corresponding to the left and right boundaries
of a cell by
\begin{equation}
\vV_L = [\ell_0 (0), \ell_1 (0), \ldots, \ell_N (0)]^{\top}, \qquad \vV_R =
[\ell_0 (1), \ell_1 (1), \ldots, \ell_N (1)]^{\top}\label{eq:mdrk.VlVr}
\end{equation}
which is used to extrapolate the solution and/or flux to the cell faces for
the computation of inter-cell fluxes.

\section{Runge-Kutta FR}\label{sec:mdrk.rk}
We first give a brief resume of RKFR scheme since our method has some similarities and borrows ideas from FR method. The RKFR scheme is based on an FR spatial discretization leading to a system of ODE\correction{s} followed by application of an RK scheme to march forward in time. The key idea is to construct a continuous polynomial approximation of the flux which is then used in a collocation scheme to update the nodal solution values. At some time $t$, we have the piecewise polynomial solution defined inside each cell; the FR scheme can be described by the following steps.

\paragraph{Step 1.}In each element, we construct the flux approximation by
interpolating the flux at the solution points leading to a polynomial of
degree $N$, given by
\[ \pf_h^{\delta} (\xi, t) = \sum_{p = 0}^N \pf (\uep (t)) \ell_p (\xi) \]
The above flux is in general discontinuous across the elements similar to the
red curve in Figure~(\ref{fig:mdrk.solflux1}b).

\paragraph{Step 2.}We build a continuous flux approximation by adding some
correction terms at the element boundaries
\begin{equation}
\pf_h (\xi, t) = \left[ \pf_{\emh} (t) - \pf_h^{\delta} (0, t) \right] g_L
(\xi) + \pf_h^{\delta} (\xi, t) + \left[ \pf_{\eph} (t) - \pf_h^{\delta}
(1, t) \right] g_R (\xi) \label{eq:cts.flux}
\end{equation}
where
\[ \pf_{\eph} (t) = \pf (\uu_h (x_{\eph}^-, t), \uu_h (x_{\eph}^+, t)) \]
is a numerical flux function that makes the flux unique across the cells. The
continuous flux approximation is illustrated by the black curve in
Figure~(\ref{fig:mdrk.solflux1}b). The functions $g_L, g_R$ are the correction
functions which must be chosen to obtain a stable scheme.

\paragraph{Step 3.}We obtain the system of ODE\correction{s} by collocating the PDE at the
solution points
\[ \od{\uu_p^e}{t} (t) = - \frac{1}{\Delta x_e}  \correction{(\pf_h)_\xi} (\xi_p, t),
\qquad 0 \le p \le N \]
which is solved in time by a Runge-Kutta scheme.

\paragraph{Correction functions.}The correction functions $g_L, g_R$ should
satisfy the end point conditions
\begin{align*}
g_L (0) = g_R(1) = 1, \qquad & g_R (0) = g_L(1) = 0
\end{align*}
which ensures the continuity of the flux, i.e., $\pf_h (x_{\eph}^-, t) = \pf_h
(x_{\eph}^+, t) = \pf_{\eph} (t)$. The reader is referred
to~{\cite{babbar2022,Huynh2007,Vincent2011a}} for detailed discussion and choices of correction functions.
\correction{
\begin{remark}
The global flux approximation in~\eqref{eq:cts.flux} is globally continuous in 1-D. However, in higher dimensions, the flux is only made continuous along the normal direction to the element interfaces.
\end{remark}
}
\section{Multi-derivative Runge-Kutta FR scheme}\label{sec:mdrk.mdrk}

Multiderivative Runge-Kutta~{\cite{obrechkoff1940neue}} methods were initially
developed to solve systems of ODE\correction{s} like
\begin{equation}
\label{eq:ode.L} \dv{\tmmathbf{u}}{t} = \tmmathbf{L} (\tmmathbf{u})
\end{equation}
that use temporal derivatives of $\tmmathbf{L}$. In this work, we use the two stage fourth order multiderivative Runge-Kutta method from~\cite{li2016}. For the system of ODEs~\eqref{eq:ode.L}, the MDRK
scheme of~\cite{li2016} to evolve from $t^n$ to $t^{n + 1}$ is given by
\begin{equation*}
\begin{split}
\tmmathbf{u}^{\ast} & = \tmmathbf{u}^n + \frac{1}{2} \mathLaplace t
\tmmathbf{L} (\tmmathbf{u}^n) + \frac{\mathLaplace t^2}{8}
\dv{\tmmathbf{L}}{t} (\tmmathbf{u}^n)\\
\tmmathbf{u}^{n + 1} & = \tmmathbf{u}^n + \mathLaplace t \tmmathbf{L}
(\tmmathbf{u}^n) + \frac{\mathLaplace t^2}{6}  \left(
\dv{\tmmathbf{L}}{t} (\tmmathbf{u}^n) + 2 \dv{\tmmathbf{L}}{t}
(\tmmathbf{u}^{\ast}) \right)
\end{split}
\end{equation*}
In order to solve the 1-D conservation law~\eqref{eq:con.law} using the above
scheme, we formally set $\tmmathbf{L} = - \pf (\tmmathbf{u})_x$ to get the
following two stage procedure
\begin{align}
\tmmathbf{u}^{\ast} & = \tmmathbf{u}^n - \frac{\mathLaplace t}{2}
\partial_x \tmmathbf{F} \label{eq:mdrk.mdrk.first.stage} \\
\tmmathbf{u}^{n + 1} & = \tmmathbf{u}^n - \mathLaplace t \partial_x
\tmmathbf{F}^{\ast} \label{eq:mdrk.mdrk.second.stage}
\end{align}
where
\begin{equation}
\begin{split}
\tmmathbf{F} & \assign \pf \left( \bu^n \right) + \frac{1}{4}
\mathLaplace t \correction{\pf_t} \left( \bu^n \right)\\
\tmmathbf{F}^{\ast} & \assign \pf (\tmmathbf{u}^n) + \frac{1}{6}
\mathLaplace t \left( \correction{\pf_t} (\tmmathbf{u}^n) + 2\correction{\pf_t}
(\tmmathbf{u}^{\ast}) \right)
\end{split} \label{eq:mdrk.f2.defn}
\end{equation}
The formal order of accuracy of the scheme
(Appendix~\ref{sec:mdrk.formal.accuracy}) is obtained from
\[ \partial_x \tmmathbf{F}^{\ast} = \frac{1}{\mathLaplace t} \partial_x
\int_{t^n}^{t^{n + 1}} \pf + O (\mathLaplace t^4) \]
The idea is to use~(\ref{eq:mdrk.mdrk.first.stage},~\ref{eq:mdrk.mdrk.second.stage}) to obtain solution update at the nodes written as a collocation scheme
\begin{equation}
\begin{split}
\uus_{e,p} &= \uep^n - \frac{\Delta t}{2 \Delta x_e}  \od{\Fone_h}{\xi}
(\xi_p)\\
\uep^{n + 1} & = \uep^n - \frac{\Delta t}{\Delta x_e}  \od{\Ftwo_h}{\xi}
(\xi_p)
\end{split}, \qquad 0 \le p \le N \label{eq:mdrk.uplwfr}
\end{equation}
where we take $N = 3$ to get fourth order accuracy. The major work is in the construction of the time average flux approximations $\Fone_h, \Ftwo_h$ which is explained in subsequent sections.

\subsection{Conservation property}

The computation of correct weak solutions for non-linear conservation laws in
the presence of discontinuous solutions requires the use of conservative
numerical schemes. The Lax-Wendroff theorem shows that if a consistent,
conservative method converges, then the limit is a weak solution. The
method~\eqref{eq:mdrk.uplwfr} is also conservative though it is not directly
apparent; to see this multiply~\eqref{eq:mdrk.uplwfr} by the quadrature weights
associated with the solution points and sum over all the points in the
$e^{\text{th}}$ element,
\begin{align}
\correction{\au_e^{*} := } \sum_{p = 0}^N w_p  \uep^{\ast} & = \sum_{p = 0}^N w_p  \uep^n -
\frac{\Delta t}{2 \Delta x_e}  \sum_{p = 0}^N w_p  \correction{\od{\Fone_h}{\xi}}
(\xi_p) \nonumber \\
\correction{\au_e^{n+1} := }\sum_{p = 0}^N w_p  \uep^{n + 1} & = \sum_{p = 0}^N w_p  \uep^n -
\frac{\Delta t}{\Delta x_e}  \sum_{p = 0}^N w_p  \correction{\od{\Ftwo_h}{\xi}} (\xi_p) \label{eq:au.defn}
\end{align}
The correction functions are of degree $N + 1$ and thus the fluxes $\Fone_h,
\Ftwo_h$ are polynomials of degree $\le N + 1$. If the quadrature is exact for
polynomials of degree at least $N$, which is true for both GLL and GL points,
then the quadrature is exact for the flux derivative term and we can write it
as an integral, which leads to
\begin{equation}
\begin{split}
\int_{\Omega_e} \uus_h  \ud x & = \int_{\Omega_e} \uu_h^n  \ud x - \frac{\Delta
t}{2} [\Fone_{\eph} - \Fone_{\emh}] \\
\int_{\Omega_e} \uu^{n + 1}_h  \ud x & = \int_{\Omega_e} \uu_h^n  \ud x -
\Delta t [\Ftwo_{\eph} - \Ftwo_{\emh}]
\end{split} \label{eq:mdrk.upmean}
\end{equation}
This shows that the total mass inside the cell changes only due to the
boundary fluxes and the scheme is hence conservative. The conservation
property is crucial in the proof of admissibility preservation studied
in~Section~\ref{sec:mdrk.flux.limiter}.

Recall that the solutions of conservation law~\eqref{eq:con.law} belong to
an admissible set $\Uad$~\eqref{eq:mdrk.uad.form}. The admissibility preserving
property, also known as convex set preservation property (since $\Uad$ is
convex) of the conservation law can be written as
\begin{equation}
\label{eq:mdrk.conv.pres.con.law} \bw (\cdummy, t_0) \in \Uad \qquad
\Longrightarrow \qquad \bw (\cdummy, t) \in \Uad, \qquad t > t_0
\end{equation}
and thus we define an admissibility preserving scheme to be

\begin{definition}
\label{defn:mdrk.adm.pres}The flux reconstruction scheme is said to be
admissibility preserving if
\[ \uep^n \in \Uad \quad \forall e, p \qquad \Longrightarrow \qquad \uep^{n
+ 1} \in \Uad \quad \forall e, p \]
where $\Uad$ is the admissibility set of the conservation law.
\end{definition}

To obtain an admissibility preserving scheme, we exploit the weaker
admissibility preservation in means property defined as

\begin{definition}
\label{defn:mdrk.mean.pres}The flux reconstruction scheme is said to be
admissibility preserving in the means if
\[ \uep^n \in \Uad \quad \forall e, p \qquad \Longrightarrow \qquad \au_e^{n
+ 1} \in \Uad \quad \forall e \]
where $\Uad$ is the admissibility set of the conservation law\correction{, and $\au_e^{n+1}$ is element mean as given by~\eqref{eq:au.defn}}.
\end{definition}

\subsection{Reconstruction of the time average flux}\label{sec:mdrk.reconstruction}

To complete the description of the MDRK method~\eqref{eq:mdrk.uplwfr}, we must
explain the method for the computation of the time average fluxes $\Fone_h,
\Ftwo_h$ when evolving from $t^n$ to $t^{n + 1}$. In the first
stage~\eqref{eq:mdrk.mdrk.first.stage}, we compute $\Fone_h$ which is then used to
evolve to $\uus$. In the second stage~\eqref{eq:mdrk.mdrk.second.stage}, $\uu^n,
\uus$ are used to compute $\Ftwo_h$ which is used from evolution to $\uu^{n +
1}$. Since the reconstruction procedure is same for both of the time average
fluxes, we explain the procedure of Flux Reconstruction of $\Ftwo_h (\xi)$, which
is a time averaged flux over the interval $[t^n, t^{n + 1}]$, performed in
three steps. Note that all results in this work use degree $N = 3$, although
the following steps are written for general $N$.

\paragraph{Step 1.}Use the approximate Lax-Wendroff procedure
of~{\cite{Zorio2017}} to compute the time average flux $\Ftwo$ at all the
solution points
\[ \Ftwo_{e,p} \approx \Ftwo (\xi_p), \qquad 0 \le p \le N \]
The approximate LW procedure explained in Section~\ref{sec:mdrk.alw}.

\paragraph{Step 2.}Build a local approximation of the time average flux inside
each element by interpolating at the solution points
\[ \Ftwod_h (\xi) = \sum_{p = 0}^N \Ftwo_{e,p} \ell_p (\xi) \]
which however may not be continuous across the elements. This is illustrated
in Figure~\ref{fig:mdrk.solflux1}b.

\paragraph{Step 3.} Modify the flux approximation $\Ftwod_h (\xi)$ so
that it becomes continuous across the elements. Let $\Ftwo_{\eph}$ be some
numerical flux function that approximates the flux $\Ftwo$ at $x = x_{\eph}$.
Then the continuous flux approximation is given by
\[ \Ftwo_h (\xi) = \left[ \Ftwo_{\emh} - \Ftwod_h (0) \right] g_L (\xi) +
\Ftwod_h (\xi) + \left[ \Ftwo_{\eph} - \Ftwod_h (1) \right] g_R
(\xi) \]
which is illustrated in Figure~\ref{fig:mdrk.solflux1}b. The correction functions
$g_L, g_R \in \poly_{N + 1}$ are chosen from the FR
literature~{\cite{Huynh2007,Vincent2011a}}.

\paragraph{Step 4.}The derivatives of the continuous flux approximation at the
solution points can be obtained as
\[
\myvector{\partial_\xi \Ftwo_h} = \left[\Ftwo_\emh - \vV_L^\top \vFtwo \right] \vb_L +
\vD \vFtwo +
\left[\Ftwo_\eph - \vV_R^\top \vFtwo \right] \vb_R, \qquad \vb_L = \begin{bmatrix}
g_L'(\xi_0) \\
\vdots \\
g_L'(\xi_N) \end{bmatrix}, \qquad \vb_R = \begin{bmatrix}
g_R'(\xi_0) \\
\vdots \\
g_R'(\xi_N) \end{bmatrix}
\]
which can also be written as
\begin{equation}
\myvector{\partial_{\xi} \Ftwo_h} = \Ftwo_{\emh}  \vb_L + \vD_1
\vFtwo + \Ftwo_{\eph}  \vb_R, \qquad \vD_1 = \vD - \vb_L  \vV_L^{\top} - \vb_R
\vV_R^{\top} \label{eq:mdrk.fder}
\end{equation}
where $\vV_L, \vV_R$ are Vandermonde matrices which are defined
in~\eqref{eq:mdrk.VlVr}. The quantities $\vb_L, \vb_R, \vV_L, \vV_R, \vD, \vD_1$
can be computed once and re-used in all subsequent computations. They do not
depend on the element and are computed on the reference element.
Equation~\eqref{eq:mdrk.fder} contains terms that can be computed inside a single
cell (middle term) and those computed at the faces (first and third terms)
where it is required to use the data from two adjacent cells. The computation
of the flux derivatives can thus be performed by looping over cells and then
the faces.

\subsection{Approximate Lax-Wendroff procedure}\label{sec:mdrk.alw}

The time average fluxes $\Fone_p, \Ftwo_p$ must be computed
using~\eqref{eq:mdrk.f2.defn}, which involves temporal derivatives of the flux. A
natural approach is to use the PDE and replace time derivatives with spatial
derivatives, but this leads to large amount of algebraic computations since we
need to evaluate the flux Jacobian and its higher tensor versions. To avoid
this process, we follow the ideas in~{\cite{Zorio2017,Burger2017}} and adopt
an approximate Lax-Wendroff procedure. To present this idea in a concise and
efficient form, we introduce the notation
\[ \uu^{(1)} = \Delta t \partial_t  \uu, \qquad \pf^{(1)} = \Delta t
\partial_t  \pf \]
The time derivatives of the solution are computed using the PDE
\[ \uu^{(1)} = - \Delta t \partial_x  \pf \]
The approximate Lax-Wendroff procedure is applied in each element and so for
simplicity of notation, we do not show the element index in the following.

\paragraph{First stage.}

\begin{equation}
\Fone \assign \pf \left( \bu^n \right) + \frac{1}{4} \mathLaplace t
\correction{\pf_t} \left( \bu^n \right) \approx \frac{1}{\mathLaplace t / 2}
\int_{t^n}^{t^{n + 1 / 2}} \pf \left( \bu \right)  \ud t
\label{eq:Fone}
\end{equation}
To obtain fourth order accuracy, the approximation for $\pdv{}{t}  \pf \left(
\bu^n \right)$ needs to be at least third order
accurate~(Appendix~\ref{sec:mdrk.formal.accuracy}) which we obtain as

\begin{equation*}
\begin{split}
\pf_t (\xi, t) & \approx \frac{- \pf (\uu (\xi, t + 2 \Delta t)) + 8
\pf (\uu (\xi, t + \Delta t)) - 8 \pf (\uu (\xi, t - \Delta t)) + \pf
(\uu (\xi, t - 2 \Delta t))}{12 \Delta t} \\
& \approx \left. \frac{- \pf (\uu + 2 \mathLaplace t \uu_t) + 8 \pf
(\uu + \mathLaplace t \uu_t) - 8 \pf (\uu - \mathLaplace t \uu_t) + \pf
(\uu - 2 \mathLaplace t \uu_t)}{12 \Delta t} \right|_{(\xi, t)}
\end{split}
\end{equation*}

We denote $\vf$ as the vector of flux values at the solution points. Thus, the time averaged flux is computed as
\[ \vFone = \vf + \frac{1}{4}  \vf^{(1)} \]
where
\begin{align*}
\vu^{(1)} & = - \frac{\Delta t}{\Delta x_e}  \vD \vf\\
\vf^{(1)} & = \frac{1}{12}  \left[ - \pf ( \vu + 2 \vu^{(1)} )
+ 8 \pf ( \vu + \vu^{(1)} ) - 8 \pf ( \vu - \vu^{(1)}
) + \pf ( \vu - 2 \vu^{(1)} ) \right]
\end{align*}

\paragraph{Second stage.}

\begin{equation}
\tmmathbf{F}^{\ast} \assign \pf (\tmmathbf{u}^n) + \frac{1}{6}
\mathLaplace t \left( \pdv{}{t}  \pf (\tmmathbf{u}^n) + 2 \pdv{}{t}  \pf
(\tmmathbf{u}^{\ast}) \right)
\label{eq:Ftwo}
\end{equation}
The time averaged flux $\tmmathbf{F}^{\ast}$ is computed as
\[ \vFtwo = \vf + \frac{1}{6}  ( \vf^{(1)} + 2 \vfsone ) \]
where
\begin{align*}
\vu^{*(1)} & = - \frac{\Delta t}{\Delta x_e}  \vD \vfs \nonumber\\
\vf^{*(1)} & = \frac{1}{12}  \left[ - \pf ( \vus + 2 \vusone
) + 8 \pf ( \vus + \vusone ) - 8 \pf ( \vus -
\vusone ) + \pf ( \vus - 2 \vusone ) \right]
\end{align*}
\begin{remark}
The $\vf, \vf^{(1)}$ computed in the first stage are reused in the second.
\end{remark}

\subsection{Numerical flux}\label{sec:mdrk.numflux}

The numerical flux couples the solution between two neighbouring cells in a
discontinuous Galerkin type method. In RK methods, the numerical flux is a
function of the trace values of the solution at the faces. In the MDRK scheme,
we have constructed time average fluxes at all the solution points inside the
element and we want to use this information to compute the time averaged
numerical flux at the element faces. The simplest numerical flux is based on
Lax-Friedrich type approximation and is similar to the one used for
LW~{\cite{Qiu2005b}} and is similar to the D1 dissipation
of~{\cite{babbar2022}}
\begin{equation}
\begin{split}
\Fone_{\eph} & = \half  [\Fone_{\eph}^{-} + \Fone_{\eph}^{+}] - \half
\lambda_{\eph}  [\uu_h (x_{\eph}^+, t_n) - \uu_h (x_{\eph}^-, t_n)]\\
\Ftwo_{\eph}  & = \half  [\Ftwom_{\eph} + \Ftwop_{\eph}] - \half \lambda_{\eph}
[\uu_h (x_{\eph}^+, t_n) - \uu_h (x_{\eph}^-, t_n)]
\end{split} \label{eq:mdrk.nfdiss1}
\end{equation}
which consists of a central flux and a dissipative part. For linear advection
equation $\uu_t + a \uu_x = \bzero$, the coefficient in the dissipative part
of the flux is taken as $\lambda_{\eph} = |a|$, while for a non-linear PDE
like Burger's equation, we take it to be
\[ \lambda_{\eph} = \max \{| \pf' (\au_e^n) |, | \pf' (\au_{e + 1}^n) |\} \]
where $\au_e^n$ is the cell average solution in element $\Omega_e$ at time
$t_n$, and will be referred to as Rusanov or local
Lax-Friedrich~{\cite{Rusanov1962}} approximation. Note that the dissipation
term in the above numerical flux is evaluated at time $t_n$ whereas the
central part of the flux uses the time average flux. Since the dissipation
term contains solution difference at faces, we still expect to obtain optimal
convergence rates, which is verified in numerical experiments. This numerical
flux depends on the following quantities: $\{ \au_e^n, \au_{e + 1}^n, \uu_h
(x_{\eph}^-, t_n), \uu_h (x_{\eph}^+, t_n), \F_{\eph}^-, \F_{\eph}^+ \}$.

The numerical flux of the form~\eqref{eq:mdrk.nfdiss1} leads to somewhat reduced
CFL numbers as is experimentally verified and discussed in
Section~\ref{sec:mdrk.fourier}, and also does not have upwind property even for
linear advection equation. An alternate form of the numerical flux is obtained
by evaluating the dissipation term using the time average solution, leading to
the formula similar to D2 dissipation of~{\cite{babbar2022}}
\begin{equation}
\begin{split}
\Fone_{\eph} & = \half  [\Fone_{\eph}^- + \Fone_{\eph}^+] -
\half \lambda_{\eph}  [\uUone_{\eph}^+ - \uUone_{\eph}^-]\\
\Ftwo_{\eph} & = \half  [\Ftwom_{\eph} + \Ftwop_{\eph}] - \half \lambda_{\eph}
[\uUtwop_{\eph} - \uUtwom_{\eph}]
\end{split} \label{eq:mdrk.nfdiss2}
\end{equation}
where
\begin{equation}
\begin{split}
\uU^{\ast} & = \vu + \frac{1}{4}  \vu^{(1)}\\
\uU & = \vu + \frac{1}{6}  \left( \vu^{(1)} + 2 \vusone \right)
\end{split} \label{eq:mdrk.tavgsol}
\end{equation}
are the time average solutions. In this case, the numerical flux depends on
the following quantities: $\{ \au_e^n, \au_{e + 1}^n, \uU_{\eph}^-,
\uU_{\eph}^+, \F_{\eph}^-, \F_{\eph}^+ \}$. Following~{\cite{babbar2022}}, we
will refer to the above two forms of dissipation as D1 and D2, respectively.
The dissipation model D2 is not computationally expensive compared to the D1
model since since all the quantities required to compute the time average solutions
$\tmmathbf{U}, \tmmathbf{U}^{\ast}$ are available during the Lax-Wendroff
procedure. It remains to explain how to compute $\tmmathbf{F}^{\pm}_{e +
\frac{1}{2}}, \tmmathbf{F}^{\ast \pm}_{e + \frac{1}{2}}$ appearing in the
central part of the numerical flux. There were two ways introduced for Lax-Wendroff
in~{\cite{babbar2022}} to compute the central flux, termed {\extrapolate} and
{\evaluate}. We explain how the two apply to MDRK in the next two subsections.

\subsection{Numerical flux -- average and extrapolate to face
(AE)}\label{sec:mdrk.ae}

In each element, the time average fluxes $\tmmathbf{F}_h^{\delta} {,
\tmmathbf{F}^{\ast}_h}^{\delta}$ corresponding to each stage have been
constructed using the Lax-Wendroff procedure. The simplest approximation that
can be used for $\tmmathbf{F}^{\pm}_{\eph}, \tmmathbf{F}^{\ast \pm}_{\eph}$ in
the central part of the numerical flux is to extrapolate the fluxes
$\tmmathbf{F}_h^{\delta} {, \tmmathbf{F}^{\ast}_h}^{\delta}$to the faces
\[ \tmmathbf{F}^{\pm}_{\eph}, \tmmathbf{F}^{\ast \pm}_{\eph}
=\tmmathbf{F}_h^{\delta} (x_{\eph}^{\pm}) {,
\tmmathbf{F}^{\ast}_h}^{\delta} (x_{\eph}^{\pm}) \]
As in~{\cite{babbar2022}}, we will refer to this approach with the
abbreviation {\extrapolate}. However, as shown in the numerical results, this
approximation can lead to sub-optimal convergence rates for some non-linear
problems. Hence we propose another method for the computation of the
inter-cell flux which overcomes this problem as explained next.

\subsection{Numerical flux -- extrapolate to face and average
(EA)}\label{sec:mdrk.ea}

Instead of extrapolating the time average flux from the solution points to the
faces, we can instead build the time average flux at the faces directly using
the approximate Lax-Wendroff procedure that is used at the solution points.
The flux at the faces is constructed after the solution is evolved at all the
solution points. In the following equations, $\alpha$ denotes either the left
face ($L$) or the right face ($R$) of a cell. For $\alpha \in \{L, R\}$, we
compute the time average flux at the faces of the $e$\tmrsup{th} element by
the following steps, where we suppress the element index since all the
operations are performed inside one element.

\paragraph{Stage 1.}

\begin{align*}
\uu_{\alpha}^{\pm} & = \vV_{\alpha}^{\top} ( \vu \pm \vu^{(1)}
)\\
\pf^{(1)}_{\alpha} & = \frac{1}{12}  [ - \pf ( \uu^{+
2}_{\alpha} ) + 8 \pf \left( \uu^+_{\alpha} \right) - 8 \pf \left(
\uu^-_{\alpha} \right) + \pf \left( \uu^{- 2}_{\alpha} \right) ]\\
\F_{\alpha}^{\ast} & = \pf (\uu_{\alpha}) + \frac{1}{4}
\pf^{(1)}_{\alpha}
\end{align*}

\paragraph{Stage 2.}

\begin{align*}
\uu_{\alpha}^{\ast \pm} & = \vV_{\alpha}^{\top} ( \vus \pm \vusone
)\\
\uu_{\alpha}^{\ast \pm 2} & = \vV_{\alpha}^{\top} ( \vus \pm 2
\vusone )\\
\pf^{\ast (1)}_{\alpha} & = \frac{1}{12}  [ - \pf ( \uu^{\ast +
2}_{\alpha} ) + 8 \pf ( \uu^{\ast +}_{\alpha} ) - 8 \pf
( \uu^{\ast -}_{\alpha} ) + \pf ( \uu^{\ast - 2}_{\alpha}
) ]\\
\F_{\alpha} & = \pf (\uu_{\alpha}) + \frac{1}{6}  (
\pf^{(1)}_{\alpha} + 2 \pf^{\ast (1)} )
\end{align*}
\begin{remark}
The $\pf (\uu_{\alpha})$, $\pf^{(1)}_{\alpha}$ computed in the first stage,
are reused in the second stage.
\end{remark}

We see that the solution is first extrapolated to the cell faces and the same
finite difference formulae for the time derivatives of the flux which are used
at the solution points, are also used at the faces. The numerical flux is
computed using the time average flux built as above at the faces; the central
part of the flux $\F_{\eph}^{\pm}$ in
equations~\eqref{eq:mdrk.nfdiss1},~\eqref{eq:mdrk.nfdiss2} are computed as
\[
\begin{array}{c}
\tmmathbf{F}^-_{e + \frac{1}{2}} = (\tmmathbf{F}_R)_e, \qquad
\tmmathbf{F}^+_{e + \frac{1}{2}} = (\tmmathbf{F}_L)_{e + 1}\\
\tmmathbf{F}^{\ast -}_{e + \frac{1}{2}} = (\tmmathbf{F}^{\ast}_R)_e,
\qquad \tmmathbf{F}^{\ast +}_{e + \frac{1}{2}} =
(\tmmathbf{F}^{\ast}_R)_{e + 1}
\end{array}
\]
We will refer to this method with the abbreviation {\evaluate}.

\begin{remark}
The two methods {\extrapolate} and {\evaluate} are different only when there
are no solution points at the faces. E.g., if we use GLL solution points,
then the two methods yield the same result since there is no interpolation
error. For the constant coefficient advection equation, the above two
schemes for the numerical flux lead to the same approximation but they
differ in case of variable coefficient advection problems and when the flux
is non-linear with respect to $u$. The effect of this non-linearity and the
performance of the two methods are shown later using some numerical
experiments.
\end{remark}

\section{Fourier stability analysis}\label{sec:mdrk.fourier}

We now perform Fourier stability analysis of the MDRK scheme applied to the
linear advection equation, $u_t + au_x = 0$, where $a$ is the constant advection speed. We will assume that the advection speed $a$ is positive and denote the CFL number by
\[ \sigma = \frac{a \Delta t}{\Delta x} \]
We will perform the CFL stability analysis for the MDRK scheme with D2
dissipation flux~\eqref{eq:mdrk.nfdiss2} and thus will like to write the two stage
scheme as
\begin{equation}
\vu_e^{n + 1} = - \sigma^2  \vA_{- 2}  \vu_{e - 2} - \sigma \vA_{- 1}
\vu_{e - 1}^n + \left( 1 - \sigma \vA_0 \right)  \vu_e^n - \sigma \vA_{+ 1}
\vu_{e + 1}^n - \sigma^2  \vA_{+ 2}  \vu_{e + 2}^n
\label{eq:mdrk.2stage.update.eqn}
\end{equation}
where the matrices $\vA_{- 2}, \vA_{- 1}, \vA_0, \vA_{+ 1}, \vA_{+ 2}$ depend
on the choice of the dissipation model in the numerical flux. We will perform
the final evolution by studying both the stages.

\subsection{Stage 1}

We will try to write the first stage as
\begin{equation}
\vu_e^{\ast} = - \sigma \vAone_{- 1}  \vu_{e - 1}^n + ( 1 - \sigma
\vAone_0 )  \vu_e^n - \sigma \vAone_{+ 1}  \vu_{e + 1}^n
\label{eq:mdrk.A0A1.matrices}
\end{equation}
Since $f_t = au_t$, the time average flux for first stage at all solution
points is given by
\[ \vFone_e = a \vUone_e \quad \text{where} \quad \vUone_e =
\vTone  \vu_e \quad \text{and} \quad \vTone = \vI -
\frac{\sigma}{4}  \vD \]
The extrapolation to the cell boundaries are given by
\[ F_h^\delta (x_{\emh}^+) = \cV_L^T  \vFone_e, \qquad F_h^{\delta} (x_{\eph}^+) = \vV_R^T \vF_e \]
The D2 dissipation numerical flux is given by
\[ \vFone_{\eph} = \vV_R^T  \vFone_e= a \vV_R^T  \vTone  \vu_e \]
and the flux differences at the face as
\[ F_{\emh} - F_h^{\delta} (x_{\emh}^+) = a \vV_R^T  \vu_{e - 1}
- a \vV_L^T  \vTone  \vu_e, \qquad F_{\eph} - F_h^{\delta}(x_{\eph}^-) = 0 \]
so that the flux derivative at the solution points is given by
\[
\partial_{\xi}  \vFone_h = \vb_L  ( a \vV_R^T  \vTone  \vu_{e
- 1} - a \vV_L^T  \vTone  \vu_e ) + a \vD  \vTone  \vu_e = a
\vb_L  \vV_R^T  \vTone  \vu_{e - 1} + a ( \vD  \vTone - \vb_L
\vV_L^T  \vTone )  \vu_e
\]
Since the evolution to $\vu^{\ast}$ is given by
\begin{equation}
\vu^{\ast} = \vu^n - \frac{\mathLaplace t / 2}{\mathLaplace x_e}
\partial_{\xi}  \vFone_h \label{eq:mdrk.us.fourier}
\end{equation}
the matrices in~\eqref{eq:mdrk.A0A1.matrices} are given by
\begin{equation}
\vAone_{- 1} = \frac{1}{2}  \vb_L  \vV_R^T  \vTone, \qquad
\vAone_0 = \frac{1}{2}  ( \vD  \vTone - \vb_L  \vV_L^T
\vTone ), \qquad \vAone_{+ 1} = 0 \label{eq:mdrk.A.star.defn}
\end{equation}
The upwind character of the D2 dissipation flux leads to $\vAone_{+ 1} =
0$ and the right cell does not appear in the update equation.

\subsection{Stage 2}

After stage 1, we have $\vu^{\ast}, \vu^n$ and both are used to obtain $\vu^{n
+ 1}$. In this case,
\[ \vF^*_e = a \vU^*_e, \qquad \vU^*_e = \vu_e^n - \frac{1}{6} \sigma \vD
\vu_e^n - \frac{1}{3} \sigma \vD  \vu_e^{\ast} = \vT^{(2)}  \vu_e^n + \vT^{(2,
\ast)}  \vu_e^{\ast} \]
where
\[ \vT^{(2)} = \vI - \frac{1}{6} \sigma \vD, \qquad \vT^{(2, \ast)} = - \frac{1}{3}
\sigma \vD \]
The numerical fluxes are given by
\[
\vF^*_{\eph} = \frac{1}{2}  \left[ \vV_R^T  \vF_e^* + \vV_L^T  \vF_{e +
1}^* \right] - \frac{1}{2} a \left( \vV_L^T  \vU^*_{e + 1} - \vV_R^T
\vU^*_e \right)   = a \vV_R^T  \vU_e^* \qquad \textrm{and} \qquad
\vF^*_{\emh} = a \vV_R^T  \vU_{e - 1}^*
\]
and the face extrapolations are
\[
F_h^{\delta} (x_{\eph}^-)  = \vV_R^T  \vF_e^* = a \vV_R^T  \vU_e^*, \qquad
F_h^{\delta} (x_{\emh}^+)  = \cV_L^T  \vF_e^* = a \vV_L^T  \vU_e^*
\]
Thus, the flux difference at the faces is
\begin{equation}
\begin{split}
F_{\eph} - F_h^{\delta} (x_{\eph}^-) & = 0\\
F_{\emh} - F_h^{\delta} (x_{\emh}^+) & = a \left( \vV_R^T  \vU_{e - 1}^*
- \vV_L^T  \vU_e^* \right)
\end{split}
\end{equation}
and the flux derivative at the solution points is given by
\begin{equation}
\begin{split}
\partial_{\xi}  \vF_h^* & = a \vD  \vU^*_e + a \vb_L  \left( \vV_R^T
\vU_{e - 1}^* - \vV_L^T  \vU_e^* \right)\\
& = a \vb_L  \vV_R^T  \vU_{e - 1}^* + a \left( \vD - \vb_L  \vV_L^T
\right)  \vU^*_e
\end{split}
\end{equation}
We now expand $\vU_e^*$ in terms of $\vu_e^n$ as follows
\[
\vU_e^* = \vT^{(2)}  \vu_e^n + \vT^{(2, \ast)}  \vu_e^{\ast}
\qquad \textrm{where} \qquad
\vT^{(2)} = \vI - \frac{1}{6} \sigma \vD, \qquad \vT^{(2, \ast)} = - \frac{1}{3}
\sigma \vD
\]
Thus, using the 1-D update formula
\[ \vu^{n + 1} = \vu^n - \frac{\mathLaplace t}{\mathLaplace x_e}
\partial_{\xi}  \vF_h^* \]
and also expanding $\vu^{\ast}$ from~\eqref{eq:mdrk.us.fourier}, the matrices
in~\eqref{eq:mdrk.2stage.update.eqn} are given by
\begin{equation}
\begin{split}
\vA_{- 2} & = - \vb_L  \vV_R^T  \vT^{(2, \ast)}  \vAone_{- 1}\\
\vA_{- 1} & = \vb_L  \vV_R^T  ( \vT^{(2)} + \vT^{(2, \ast)}  ( 1 -
\sigma \vAone_0 ) ) - \sigma (  \vD - \vb_L  \vV_L^T
)  \vT^{(2, \ast)}  \vAone_{- 1} \\
\vA_0 & = - (  \vD - \vb_L  \vV_L^T )  ( \vT^{(2)} +
\vT^{(2, \ast)}  ( \vI - \sigma \vAone_0 ) )\\
\vA_{+ 1} & = \vA_{+ 2} = 0
\end{split}
\end{equation}
where $\vAone_i$ are defined in~\eqref{eq:mdrk.A.star.defn}. The upwind
character of D2 flux is reason why we have $\vA_{+ 1} = \vA_{+ 2} = 0$.

\paragraph{Stability analysis.}We assume a solution of the form $\vu_e^n =
\widehat{\vu}_k^n \exp (ikx_e)$ where $i = \sqrt{- 1}$, $k$ is the wave number
which is an integer and $\widehat{\vu}_k^n \in \mathbb{R}^{N + 1}$ are the
Fourier amplitudes; substituting this ansatz in the update
equation~\eqref{eq:mdrk.2stage.update.eqn}, we get
\begin{equation}
\begin{gathered}
\widehat{\vu}_k^{n + 1} = H (\sigma, k)  \widehat{\vu}_k^n\\
\cH = - \sigma^2  \vA_{- 2} \exp (- 2 i \kappa) - \sigma \vA_{- 1} \exp
(- i \kappa) + \vI - \sigma \vA_0 - \sigma \vA_{+ 1} \exp (i \kappa) -
\sigma^2  \vA_{+ 2} \exp (2 i \kappa)
\end{gathered}
\end{equation}
where $\kappa = k \mathLaplace x$ is the non-dimensional wave number. The
eigenvalues of $\vH$ depend on the CFL number $\sigma$ and wave number
$\kappa$, i.e., $\lambda = \lambda (\sigma, \kappa)$; for stability, all the
eigenvalues of $\vH$ must have magnitude less than or equal to one for all
$\kappa \in [0, 2 \pi]$, i.e.,
\[ \lambda (\sigma) = \max_{\kappa} | \lambda (\sigma, \kappa) | \le 1 \]
The CFL number is the maximum value of $\sigma$ for which above stability
condition is satisfied. This CFL number is determined approximately by
sampling in the wave number space; we partition $\kappa \in [0, 2 \pi]$ into a
large number of uniformly spaced points $\kappa_p$ and determine
\[ \lambda (\sigma) = \max_p | \lambda (\sigma, \kappa_p) | \]
The values of $\sigma$ are also sampled in some interval $[\sigma_{\min},
\sigma_{\max}]$ and the largest value of $\sigma_l \in [\sigma_{\min},
\sigma_{\max}]$ for which $\lambda (\sigma_l) \le 1$ is determined in a Python
code. We start with a large interval $[\sigma_{\min}, \sigma_{\max}]$ and then
progressively reduce the size of this interval so that the CFL number is determined to about three decimal places. The results of this numerical investigation of stability are shown in Table~\ref{tab:mdrk.cfl} for two correction functions with polynomial degree $N = 3$. The comparison is made with CFL numbers of MDRK-D1~\eqref{eq:mdrk.nfdiss1} which are experimentally obtained from the linear advection test case~(Section~\ref{sec:mdrk.cla}), i.e., using a time step size that is slightly larger than these numbers causes the solution to blow up.

\begin{table}
\centering {\begin{tabular}{|c|c|c|c|c|c|}
\hline
Correction & \begin{tabular}{c}
D1\\
$\left(\begin{tabular}{c}
\text{Experimentally}\\
\text{obtained}
\end{tabular}\right)$
\end{tabular} & D2 & $\frac{\text{D2}}{\text{D1}}$ &
\begin{tabular}{c}
LW-D2\\
\begin{tabular}{c}
\text{$(N=3)$}
\end{tabular}
\end{tabular}
& $\frac{\text{MDRK-D2}}{\text{LW-D2}}$ \\
\hline
Radau & {$\approx$}0.09 & 0.107 & 1.19 & 0.103 & 1.04 \\
\hline
$g_2$ & {$\approx$}0.16 & 0.224 & 1.4 & 0.170 & 1.31 \\
\hline
\end{tabular}}
\caption{CFL numbers for MDRK scheme with Radau and $g_2$ correction
functions.\label{tab:mdrk.cfl}}
\end{table}

We see that dissipation model D2 has a higher CFL number compared to dissipation model D1. The CFL numbers for the $g_2$ correction function are also significantly higher than those for the Radau correction function. The MDRK scheme also has higher CFL numbers than the single stage LW method~\cite{babbar2022} for degree $N = 3$, which especially with the $g_2$ correction function. The optimality of these CFL numbers has been verified by experiments on the linear advection test case~(Section~\ref{sec:mdrk.cla}), i.e., the solution eventually blows up if the time step is slightly higher than what is allowed by the CFL condition.

\section{Blending scheme}\label{sec:mdrk.blending}

The MDRK scheme~\eqref{eq:mdrk.uplwfr} gives a high (fourth) order method for smooth problems, \correction{but an important class of practical problems} for hyperbolic conservation laws consist of non-smooth solutions with shocks \correction{and other discontinuities}. In such situations, using a higher order method is bound to produce Gibbs oscillations, as stated in
Godunov's order barrier theorem~{\cite{godunov1959}}. The cure is to
non-linearly add dissipation in regions where the solution is non-smooth, with
methods like artificial viscosity, limiters and switching to a robust lower
order scheme; the resultant scheme will be non-linear even for linear
equations. In this work, we develop a subcell based blending scheme for MDRK
similar to the one in~{\cite{babbar2023admissibility}} for the high order
single stage Lax-Wendroff methods. The limiter is applied to each MDRK stage.

Let us write the MDRK update equation~\eqref{eq:mdrk.uplwfr}
\begin{equation}
\vu^{H, \ast}_e = \vu^n_e - \frac{\Delta t / 2}{\Delta x_e}  \vR^H_e, \qquad
\vu^{H, n + 1}_e = \vu^n_e - \frac{\Delta t}{\Delta x_e}  \vR^{\ast, H}_e
\label{eq:mdrk.ho.method}
\end{equation}
where $\vu_e$ is the vector of nodal values in the element. We use the lower
order schemes as
\begin{equation}
\label{eq:mdrk.lo.method} \vu^{L, \ast}_e = \vu^n_e - \frac{\Delta t /
2}{\Delta x_e}  \vR^L_e, \qquad \vu^{L, n + 1}_e = \vu^n_e - \frac{\Delta
t}{\Delta x_e}  \vR^{\ast, L}_e
\end{equation}
Then the two-stage blended scheme is given by
\begin{equation}
\begin{split}
\vu^{\ast}_e & = (1 - \alpha_e)  \vu^{H, \ast}_e + \alpha_e  \vu^{L,
\ast}_e = \vu^n_e - \frac{\Delta t / 2}{\Delta x_e}  [(1 - \alpha_e)
\vR^H_e + \alpha_e  \vR^L_e]\\
\vu^{n + 1}_e & = (1 - \alpha_e)  \vu^{H, n + 1}_e + \alpha_e  \vu^{L, n
+ 1}_e = \vu^n_e - \frac{\Delta t}{\Delta x_e}  [(1 - \alpha_e) \vR^{\ast,
H}_e + \alpha_e  \vR^{\ast, L}_e]
\end{split} \label{eq:mdrk.blended.scheme}
\end{equation}
where $\alpha_e \in [0, 1]$ must be chosen based on the local smoothness
indicator. If $\alpha_e = 0$ then we obtain the high order MDRK scheme, while
if $\alpha_e = 1$ then the scheme becomes the low order scheme that is less
oscillatory. In subsequent sections, we explain the details of the lower order
scheme; \correction{for a description of the smoothness indicator, the reader is referred to~\cite{hennemann2021,babbar2023admissibility}}. The lower order scheme will
either be a first order finite volume scheme or a high resolution scheme based
on MUSCL-Hancock idea. In either case, the common structure of the low order
scheme can be explained as follows. In the first stage, the lower order
residual $\vR^{\ast, L}_e$ performs evolution from time $t^n$ to $t^{\nph}$
while, in the second stage, $\vR^L_e$ performs evolution from $t^n$ to $t^{n +
1}$. It may seem more natural to evolve from $t^{\nph}$ to $t^{n + 1}$ in the
next stage, but that will violate the conservation property because of the
expression of second stage of
MDRK~(\ref{eq:mdrk.mdrk.second.stage},~\ref{eq:mdrk.ho.method}).

\begin{figure}
\centering
\resizebox{0.7\columnwidth}{!}{\includegraphics{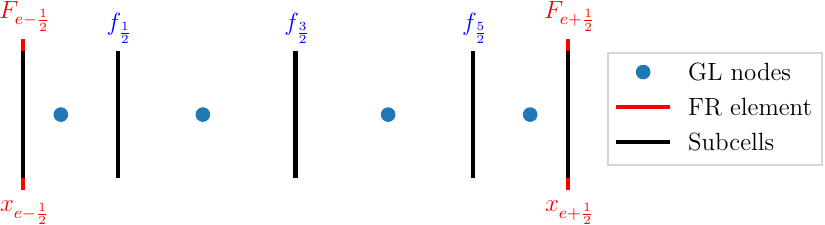}}
\caption{Subcells used by lower order scheme for degree $N = 3$. \label{fig:mdrk.subcells}}
\end{figure}

The MDRK scheme used in this work is fourth order accurate so we use
polynomial degree $N = 3$ in what follows, but write the procedure for general
$N$. Let us subdivide each element $\Omega_e$ into $N + 1$ subcells associated
to the solution points $\{ \xep, p = 0, 1, \ldots, N\}$ of the MDRK scheme.
Thus, we will have $N + 2$ subfaces denoted by $\{x^e_{\pph}, p = - 1, 0,
\ldots, N\}$ with $x^e_{- \half} = x_{\emh}$ and $x^e_{\Nph} = x_{\eph}$. For
maintaining a conservative scheme, the $p^{\text{th}}$ subcell is chosen so
that
\begin{equation}
\label{eq:mdrk.subcell.defn} x_{\pph}^e - x_{\pmh}^e = w_p \Delta x_e, \qquad 0
\le p \le N
\end{equation}
where $w_p$ is the $p^{\text{th}}$ quadrature weight associated with the
solution points. Figure~\ref{fig:mdrk.subcells} gives an illustration of the
subcells for degree $N = 3$ case. We explain the blending scheme for the
second stage, where the lower order scheme performs evolution from $t^n$ to
$t^{n + 1}$. The low order scheme is obtained by updating the solution in each
of the subcells by a finite volume scheme,
\begin{equation}
\begin{split}
\uez^{L, n + 1} & = \uez^n - \frac{\Delta t}{w_0 \Delta x_e}
[\pf_{\half}^e - \Ftwo_{\emh}]\\
\uep^{L, n + 1} & = \uep^n - \frac{\Delta t}{w_p \Delta x_e}
[\pf_{\pph}^e - \pf_{\pmh}^e], \qquad 1 \le p \le N - 1\\
\ueN^{L, n + 1} & = \ueN^n - \frac{\Delta t}{w_N \Delta x_e}  [\Ftwo_{\eph} -
\pf_{\Nmh}^e]
\end{split} \label{eq:mdrk.low.order.update}
\end{equation}
The inter-element fluxes $\Ftwo_{\eph}$ used in the low order scheme are same as
those used in the high order MDRK scheme in equation~\eqref{eq:mdrk.f2.defn}.
Usually, Rusanov's flux~{\cite{Rusanov1962}} will be used for the
inter-element fluxes and in the lower order scheme. The element mean value
obtained by the low order scheme satisfies
\begin{equation}
\label{eq:mdrk.low.order.cell.avg.update} \overline{\uu}_e^{L, n + 1} = \sum_{p =
0}^N \uep^{L, n + 1} w_p = \overline{\uu}_e^n - \frac{\Delta t}{\Delta x_e}
(\Ftwo_{\eph} - \Ftwo_{\emh})
\end{equation}
which is identical to the update equation by the MDRK scheme given in equation~\eqref{eq:mdrk.upmean} in the absence of source terms. The element mean in the blended scheme evolves according to
\begin{align*}
\overline{\uu}_e^{n + 1} & = (1 - \alpha_e)  (\overline{\uu}_e)^{H, n + 1}
+ \alpha_e  (\overline{\uu}_e)^{L, n + 1}\\
& = (1 - \alpha_e)  \left[ \overline{\uu}_e^n - \frac{\Delta t}{\Delta
x_e} (\Ftwo_{\eph} - \Ftwo_{\emh}) \right] + \alpha_e  \left[ \overline{\uu}_e^n -
\frac{\Delta t}{\Delta x_e} (\Ftwo_{\eph} - \Ftwo_{\emh}) \right]\\
& = \overline{\uu}_e^n - \frac{\Delta t}{\Delta x_e}  (\Ftwo_{\eph} -
\Ftwo_{\emh})
\end{align*}
and hence the blended scheme is also conservative. The similar arguments will
apply to the first stage and we will have by~\eqref{eq:mdrk.upmean}
\begin{equation}
\begin{split}
\avg{\uu}^{\ast}_e & = \overline{\uu}_e^n - \frac{\Delta t / 2}{\Delta x_e}
(\Fone_{\eph} - \Fone_{\emh})\\
\overline{\uu}_e^{n + 1} & = \overline{\uu}_e^n - \frac{\Delta
t}{\Delta x_e}  (\Ftwo_{\eph} - \Ftwo_{\emh})
\end{split}
\label{eq:blended.numfluxes}
\end{equation}
Overall, all three schemes, i.e., lower order, MDRK and the blended scheme,
predict the same mean value.

The inter-element fluxes $\Fone_{\eph}, \Ftwo_{\eph}$ are used both in the low
and high order schemes. To achieve high order accuracy in smooth regions, this
flux needs to be high order accurate, however it may produce spurious
oscillations near discontinuities when used in the low order scheme. A natural
choice to balance accuracy and oscillations is to take
\begin{equation}
\begin{split}
\Fone_{\eph} & = (1 - \alpha_{\eph})  \Fone_{\eph}^\text{HO} +
\alpha_{\eph}  \pf_{\eph}\\
\Ftwo_{\eph} & = (1 - \alpha_{\eph})  \FtwoHO_{\eph} +
\alpha_{\eph}  \pf_{\eph}
\end{split}
\label{eq:mdrk.Fblend}
\end{equation}
where $\Fone^\text{HO}_{\eph}, \FtwoHO_{\eph}$ are the high
order inter-element time-averaged numerical fluxes used in the MDRK
scheme~\eqref{eq:mdrk.nfdiss2} and $\pf_{\eph}$ is a lower order flux at the face
$x_{\eph}$ shared between FR elements and
subcells~(\ref{eq:mdrk.fo.at.face},~\ref{eq:mdrk.mh.at.face}). The blending coefficient
$\alpha_{\eph}$ will be based on a local smoothness indicator which will bias
the flux towards the lower order flux $\pf_{\eph}$ near regions of lower
solution smoothness. However, to enforce admissibility in means
(Definition~\ref{defn:mdrk.mean.pres}), the flux has to be further corrected, as
explained in Section~\ref{sec:mdrk.flux.limiter}.

\subsection{First order blending}\label{sec:mdrk.fo}

The lower order scheme is taken to be a first order finite volume scheme, for
which the subcell fluxes in~\eqref{eq:mdrk.low.order.update} are given by
\[ \pf_{\pph}^e = \pf (\uep, \uu_{e, p + 1}) \]
At the interfaces that are shared with FR elements, we define the lower order
flux used in computing inter-element flux as
\begin{equation}
\pf_{\eph} = \pf (\ueN, \uepoz) \label{eq:mdrk.fo.at.face}
\end{equation}
In this work, the numerical flux $\pf (\cdummy, \cdummy)$ is taken to be
Rusanov's flux~{\cite{Rusanov1962}}, which is the same flux used by the higher
order scheme at the element interfaces.

\subsection{Higher order blending}\label{sec:mdrk.mh}

The MUSCL-Hancock scheme uses a flux whose stencil is given by
\begin{equation}
\pf_{\eph} = \pf (\uu_{e, N - 1}, \ueN, \uepoz, \uu_{e + 1, 1})
\label{eq:mdrk.mh.at.face}
\end{equation}
Since the evolution in first stage is from $t^n$ to $t^{\nph}$ and in the
second from $t^n$ to $t^{n + 1}$, the procedure remains the same for both
stages. Thus, the details of the construction of an admissibility preserving
MUSCL-Hancock flux on the non-uniform non-cell centred subcells involved in
the blending scheme can be found in~{\cite{babbar2023admissibility}}.

\section{Admissibility preservation}

\correction{The Runge-Kutta FR schemes described in Section~\ref{sec:mdrk.rk} using a Strong Stability Preserving (SSP) time discretization satisfy the admissibility in means property (Definition~\ref{defn:mdrk.mean.pres})~\cite{Zhang2010b}. For any admissibility preserving in means FR scheme, the scaling limiter of~\cite{Zhang2010b} described in Section~\ref{sec:scaling} can be used to obtain an admissibility preserving scheme (Definition~\ref{defn:mdrk.adm.pres}). Since every stage of the multiderivative Runge-Kutta (MDRK) method uses a high order time-averaged flux~(\ref{eq:Fone},~\ref{eq:Ftwo}), the base MDRK scheme does not satisfy the admissibility preserving in means property. Thus, we propose to use the flux limiting procedure of~{\cite{babbar2023admissibility}} at every stage to choose the blended numerical fluxes~\eqref{eq:mdrk.Fblend} and ensure admissibility preservation in means for MDRK with blending scheme described in Section~\ref{sec:mdrk.blending}. The procedure for the MDRK scheme is briefly described in this section; for more details, the reader is referred to~{\cite{babbar2023admissibility}}.}

\subsection{Flux limiter}\label{sec:mdrk.flux.limiter}

The flux limiter to enforce admissibility will be applied in each stage so
that the evolution of each stage is admissibility preserving \correction{in means}. Thus, the
admissibility enforcing flux limiting procedure
of~{\cite{babbar2023admissibility}} applied to each stage of MDRK ensures
admissibility \correction{in means} of that stage and thus of the overall scheme.

The theoretical basis for flux limiting can be summarised in the following Theorem~\ref{thm:lwfr.admissibility}.
\begin{theorem}\label{thm:lwfr.admissibility}
Consider the MDRK scheme with blending~\eqref{eq:mdrk.blended.scheme} where low and high order schemes use the same numerical fluxes $\Fone_\eph, \Ftwo_\eph$~\eqref{eq:blended.numfluxes} at every element interface. Then the following can be said about admissibility preserving in means property (Definition~\ref{defn:mdrk.mean.pres}) of each stage of the MDRK scheme:
\begin{enumerate}
\item element means of both low and high order schemes are same and thus the blended scheme~\eqref{eq:mdrk.blended.scheme} is admissibility preserving in means if and only if the lower order scheme is admissibility preserving in means;
\item if the finite volume method using the lower order fluxes $\pf_\eph$~(\ref{eq:mdrk.fo.at.face},~\ref{eq:mdrk.mh.at.face}) as the interface flux is admissibility preserving, such as the first-order finite volume method or the MUSCL-Hancock scheme with CFL restrictions and slope correction from Theorem 3 of~\cite{babbar2023admissibility}, and the blended numerical fluxes $\Fone_\eph, \Ftwo_\eph$ are chosen to preserve the admissibility of lower-order updates at solution points adjacent to the interfaces, then the blending scheme~\eqref{eq:mdrk.blended.scheme} will preserve admissibility in means.
\end{enumerate}
\end{theorem}
\begin{proof}
By~(\ref{eq:au.defn}, \ref{eq:mdrk.low.order.cell.avg.update}), element means are the same for both low and high order schemes. Thus, admissibility in means of one implies the same for other, proving the first claim. For the second claim, note that our assumptions imply $(u_j^e)^{L,n+1}$ given by~\eqref{eq:mdrk.low.order.update} is in $\Uad$ for $0\le j \le N$, implying admissibility in means property of the lower order scheme by~\eqref{eq:mdrk.low.order.cell.avg.update} and thus admissibility in means for the blended scheme.
\end{proof}

We now explain the procedure of ensuring that the update obtained by the lower order scheme in the second stage will be admissible. The application to the first stage is analogous.
The lower order scheme is computed with a first order finite volume method or MUSCL-Hancock with slope correction from Theorem 3 of~\cite{babbar2023admissibility} so that admissibility is already ensured for inner solution points; i.e., we already have
\[
(\bw^e_{j})^{L, n + 1} \in \Uad, \quad  1 \le j \le N-1
\]
The remaining admissibility constraints for the first ($j=0$) and last solution points ($j=N$) will be satisfied by appropriately choosing the inter-element flux $F_{\eph}$. The first step is to choose a candidate for $\Ftwo_\eph$ which is heuristically expected to give reasonable control on spurious oscillations, i.e.,
\[
\Ftwo_{\eph} = (1 - \alpha_{\eph}) \FtwoHO_{\eph} + \alpha_{\eph} \pf_{\eph}, \quad \alpha_{\eph} = \frac{\alpha_e + \alpha_{e+1}}{2}
\]
where $\pf_{\eph}$ is the lower order flux at the face $\eph$ shared between FR elements and subcells~(\ref{eq:mdrk.fo.at.face}, \ref{eq:mdrk.mh.at.face}), and $\alpha_e$ is the blending coefficient~\eqref{eq:mdrk.blended.scheme} based on element-wise smoothness indicator from~\cite{hennemann2021} which is also described in Section 3.3 of~\cite{babbar2023admissibility}.

The next step is to correct $\Ftwo_{\eph}$ to enforce the admissibility constraints. The guiding principle of our approach is to perform the correction within the face loops, minimizing storage requirements and additional memory reads. The lower order updates in subcells neighbouring the $\eph$ face with the candidate flux are
\begin{equation}
\label{eq:low.order.tilde.update}
\begin{split}
\tilde{\bw}_{0}^{n+1} & = (\bw^{e+1}_0)^n - \frac{\Delta t}{w_0 \Delta x_{e+1}} ( \pf^{e+1}_{\half} - \Ftwo_{\eph} ) \\
\tilde{\bw}_{N}^{n+1} &= (\bw^e_N)^n - \frac{\Delta t}{w_N \Delta x_{e}} (\Ftwo_{\eph} - \pf^e_{\Nmh})
\end{split}
\end{equation}
In order to correct the interface flux, we will again use the fact that first order finite volume method and MUSCL-Hancock with slope corrections from~\cite{babbar2023admissibility} preserve admissibility, i.e.,
\begin{align*}
\utilow_0 & = (\bw^{e+1}_{0})^n - \frac{\Delta t}{w_0 \Delta x_{e+1}} ( \pf^{e+1}_{\half} - \pf_{\eph} ) \in \Uad \\
\utilow_N &= (\bw^e_{N})^n - \frac{\Delta t}{w_N \Delta x_{e}} (\pf_{\eph} - \pf^e_{\Nmh}) \in \Uad
\end{align*}
Let $\{p_k, 1 \le 1 \le K\}$ be the admissibility constraints~\eqref{eq:mdrk.uad.form} of the conservation law. The numerical flux is corrected by iterating over the admissibility constraints as follows

\begin{minipage}{\textwidth}  
\vspace{8pt}
\hrule
\vspace{1.5pt}
\begin{algorithmic}
\State $\Ftwo_{\eph} \gets (1 - \alpha_{\eph}) \FtwoHO_{\eph} + \alpha_{\eph} \pf_{\eph}$
\For{$k$=1:$K$}
\State $\epsilon_0, \epsilon_N \gets \frac{1}{10}p_k(\utilow_0), \frac{1}{10}p_k(\utilow_N)$
\State $\theta \gets \min \left(\min_{j=0,N} \left| \frac{\epsilon_j - p_k(\tilde{\bw}_j^{\text{low},n+1})}{p_k(\tilde{\bw}_j^{n+1}) - p_k(\tilde{\bw}_j^{\text{low},n+1})} \right |, 1 \right)$
\State $\Ftwo_{\eph} \gets \theta \Ftwo_{\eph} + (1-\theta) \pf_{\eph}$
\State $\tilde{\bw}_{0}^{n+1}  \gets (\bw^{e+1}_0)^n - \frac{\Delta t}{w_0 \Delta x_{e+1}} ( \pf^{e+1}_{\half} - \Ftwo_\eph )$
\State $\tilde{\bw}_{N}^{n+1} \gets (\bw^e_N)^n - \frac{\Delta t}{w_N \Delta x_{e}} (\Ftwo_\eph - \pf^e_{\Nmh})$
\EndFor
\end{algorithmic}
\hrule
\vspace{5pt}
\end{minipage}
By concavity of $p_k$, after the $k^\text{th}$ iteration, the updates computed using flux $\Ftwo_\eph$ will satisfy
\begin{equation}
p_k(\tilde{\bw}_j^{n+1}) =  p_k(\theta (\tilde{\bw}_j^{n+1})^\text{prev} + (1-\theta)\utilow_j) \ge \theta p_k( (\tilde{\bw}_j^{n+1})^\text{prev}) + (1 - \theta)p_k(\utilow_j) \ge \epsilon_j, \qquad j=0,N \label{eq:positive.pk}
\end{equation}
satisfying the $k^\text{th}$ admissibility constraint;  here $(\tilde{\bw}_j^{n+1})^\text{prev}$ denotes $\tilde{\bw}_j^{n+1}$ before the $k^\text{th}$ correction and the choice of $\epsilon_j = \frac{1}{10}p_k(\utilow_j)$ is made following~\cite{ramirez2021}. After the $K$ iterations, all admissibility constraints will be satisfied and the resulting flux $\Ftwo_{\eph}$ will be used as the interface flux keeping the lower order updates and thus the element means admissible. Thus, by Theorem~\ref{thm:lwfr.admissibility}, the choice of blended numerical flux gives us admissibility preservation in means. The above procedure is for 1-D conservation laws; the extension to 2-D is performed by breaking the update into convex combinations of 1-D updates and adding additional time step restrictions; the process is similar to Appendix B of~\cite{babbar2023admissibility}.

\subsection{Scaling limiter} \label{sec:scaling}

The flux limiter described in the previous section gives us an admissibility preserving in means MDRK scheme. In this section, we review the scaling limiter of~{\cite{Zhang2010b}} to use the
admissibility in means property (Definition~\ref{defn:mdrk.mean.pres}) to obtain an admissibility preserving scheme (Definition~\ref{defn:mdrk.adm.pres}). Consider the solution $\tmmathbf{u}_h^n$ at the current time time level $n$. Within each element, $\tmmathbf{u}_h^n \in \poly_N$ and since the scheme is admissibility preserving in means, we assume $\au^n_e \in \Uad$ for each
element $e$. We will iteratively correct all admissibility constraints
$\left\{ \ad_k \right\}_{k = 1}^K$~\eqref{eq:mdrk.uad.form}. For each constraint
$\ad_k$, we find $\theta_k \in [0, 1]$ such that $\ad_k \left( (1 - \theta_k)
\au^n_e + \theta_k  \tmmathbf{u}_h^n \right) > 0$ at the $N + 1$ solution
points and replace the polynomial $\tmmathbf{u}_h^n$ with $(1 - \theta)
\au^n_e + \theta \tmmathbf{u}_h^n$ where $\theta = \min_k \theta_k$. This correction procedure is performed for all admissibility constraints as follows.
\begin{minipage}{\textwidth}  
\vspace{8pt}
\hrule
\vspace{1.5pt}
\begin{algorithmic}
\State $\theta = 1$
\For{$k$=1:$K$}
\State $\epsilon_k \gets \frac{1}{10}p_k(\au_e^n)$
\State $\theta_k = \min \left( \min_{0 \leq j \leq N} \left| \frac{\epsilon_p - \ad_k
(\au^n_e)}{\ad_k (\tmmathbf{u}_{e, j}^n) - \ad_k (\au^n_e)} \right|, 1
\right)$
\State $\theta \gets \min (\theta_k, \theta)$
\EndFor
\end{algorithmic}
\hrule
\vspace{5pt}
\end{minipage}
By the same arguments as~\eqref{eq:positive.pk}, the procedure will ensure $\uu_{e,j}^n \in \Uad$ for all $e,j$ by the concavity of the admissibility constraints $\{p_k\}$~\eqref{eq:mdrk.uad.form}. Thus, by Definition~\ref{defn:mdrk.mean.pres}, the element means at the next time level will also be admissible. Overall, by performing this procedure at every time level, we have an admissibility preserving scheme in the sense of Definition~\ref{defn:mdrk.adm.pres}.

\section{Numerical results}\label{sec:mdrk.num}

In this section, we test the MDRK scheme with numerical experiments using
polynomial degree $N = 3$ in all results. Most of the test cases
from~{\cite{babbar2022,babbar2023admissibility}} were tried and were seen to
validate our claims; but we only show the important results here. We also
tested all the benchmark problems for higher order methods
in~{\cite{Pan2016}}, and show some of the results from there. \correction{The developments have been contributed to the Julia package \texttt{Tenkai.jl}~\cite{tenkai} which has been used for all the numerical experiments.}

\subsection{Scalar equations}

We perform convergence tests with scalar equations. The MDRK scheme with D1
and D2 dissipation are tested using the optimal CFL numbers from
Table~\ref{tab:mdrk.cfl}. We make comparison with RKFR scheme with polynomial
degree $N = 3$ described in Section~\ref{sec:mdrk.rk} using the SSPRK scheme
from~{\cite{Spiteri2002}}. The CFL number for the fourth order RK scheme is
taken from~{\cite{Gassner2011a}}. In many problems, we compare with
Gauss-Legendre (GL) solution points and Radau correction functions, and
Gauss-Legendre-Lobatto (GLL) solution points with $g_2$ correction functions.
These combinations are important because they are both variants of
Discontinuous Galerkin methods~{\cite{Huynh2007,Grazia2014}}.

\subsubsection{Linear advection equation}\label{sec:mdrk.cla}

The initial condition $u (x, 0) = \sin (2 \pi x)$ is taken with periodic
boundaries on $[0, 1]$. The error norms are computed at time $t = 2$ units and
shown in Figure~\ref{fig:mdrk.cla1}. The MDRK scheme with D2 dissipation (MDRK-D2)
scheme shows optimal order of convergence and has errors close to that MDRK-D1
and the RK scheme for all the combinations of solution points and correction
functions.

\begin{figure}
\centering
\begin{tabular}{cc}
\resizebox{0.46\columnwidth}{!}{\includegraphics{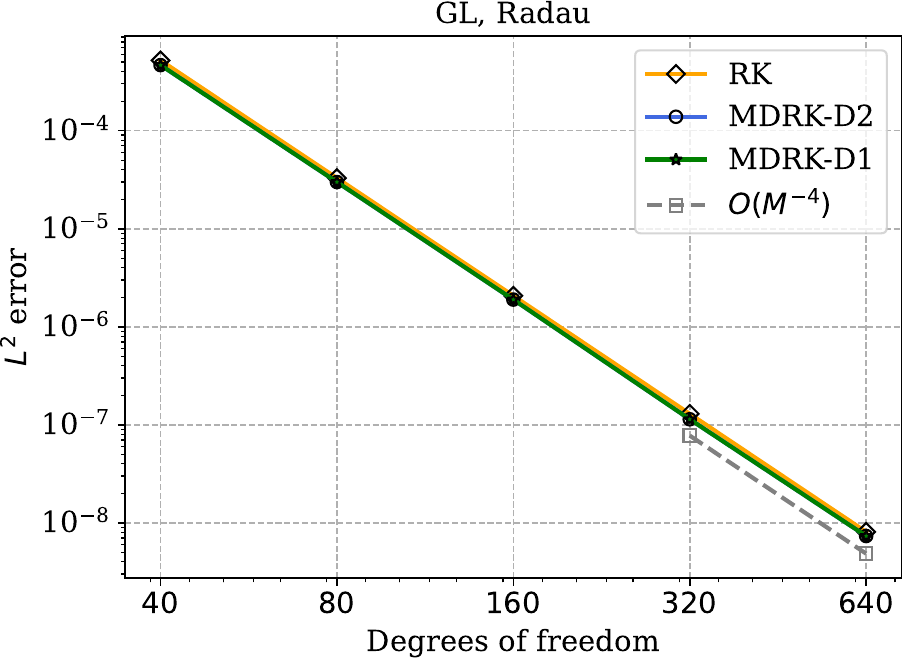}}
&
\resizebox{0.46\columnwidth}{!}{\includegraphics{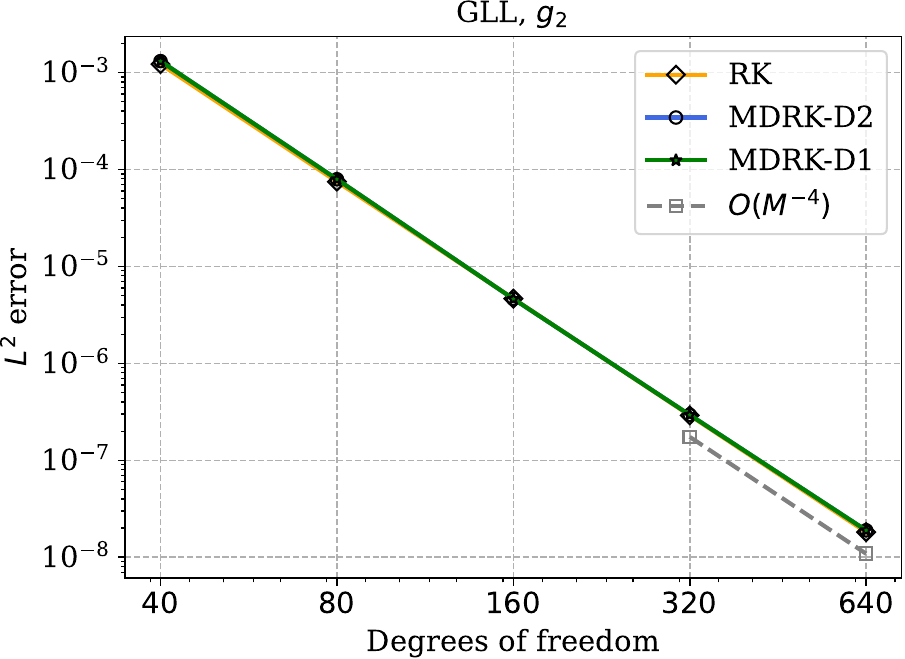}}\\
(a) & (b)
\end{tabular}
\caption{Error convergence for constant linear advection equation comparing
MDRK and RK - (a) GL points with Radau correction, (b) GLL points with $g_2$
correction\label{fig:mdrk.cla1}}
\end{figure}

\subsubsection{Variable advection equation}

\[ u_t + f (x, u)_x = 0, \qquad f (x, u) = a (x) u \]
The velocity is $a (x) = x^2$, $u_0 (x) = \cos (\pi x / 2)$, $x \in [0.1, 1]$
and the exact solution is $u (x, t) = u_0 (x / (1 + tx)) / (1 + tx)^2$.
Dirichlet boundary conditions are imposed on the left boundary and outflow
boundary conditions on the right. This problem is non-linear in the spatial
variable, i.e., if $I_h$ is the interpolation operator, $I_h (au_h) \neq I_h
(a) I_h (u_h)$. Thus, the $\extrapolate$ and {\evaluate} schemes are expected
to show different behavior, unlike the previous test.

The grid convergence analysis is shown in Figure~\ref{fig:mdrk.vla2}. In
Figure~\ref{fig:mdrk.vla2}a, the scheme with {\extrapolate} shows larger errors
compared to the RK scheme though the convergence rate is optimal. The MDRK
scheme with {\evaluate} shown in the middle figure, is as accurate as the RK
scheme. The last figure compares {\extrapolate} and {\evaluate} schemes using
GL solution points, Radau correction function and D2 dissipation; we clearly
see that {\evaluate} scheme has smaller errors than {\extrapolate} scheme.

\begin{figure}
{\noindent}\begin{tabular}{ccc}
\resizebox{0.3\columnwidth}{!}{\includegraphics{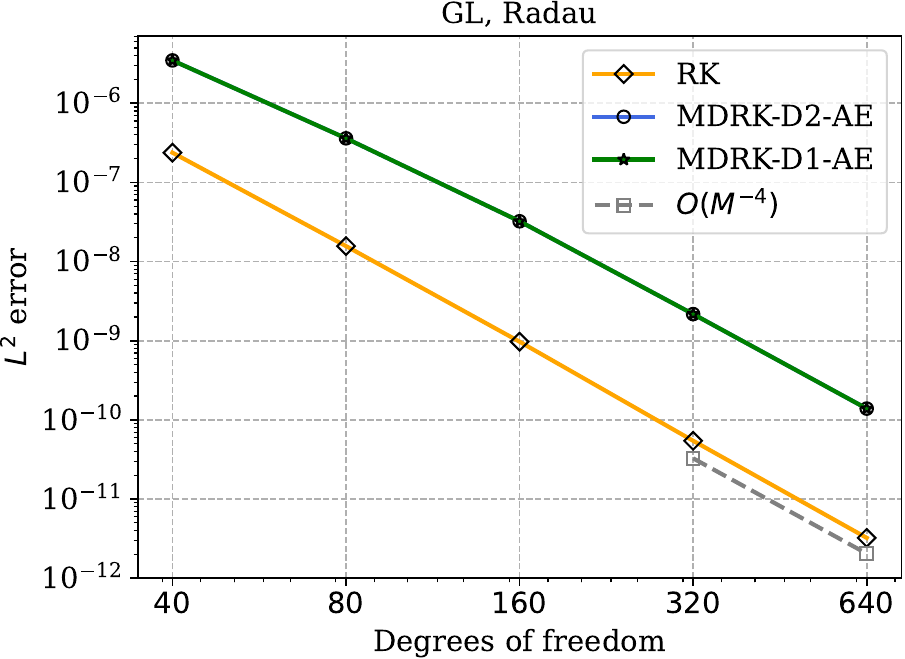}}
&
\resizebox{0.3\columnwidth}{!}{\includegraphics{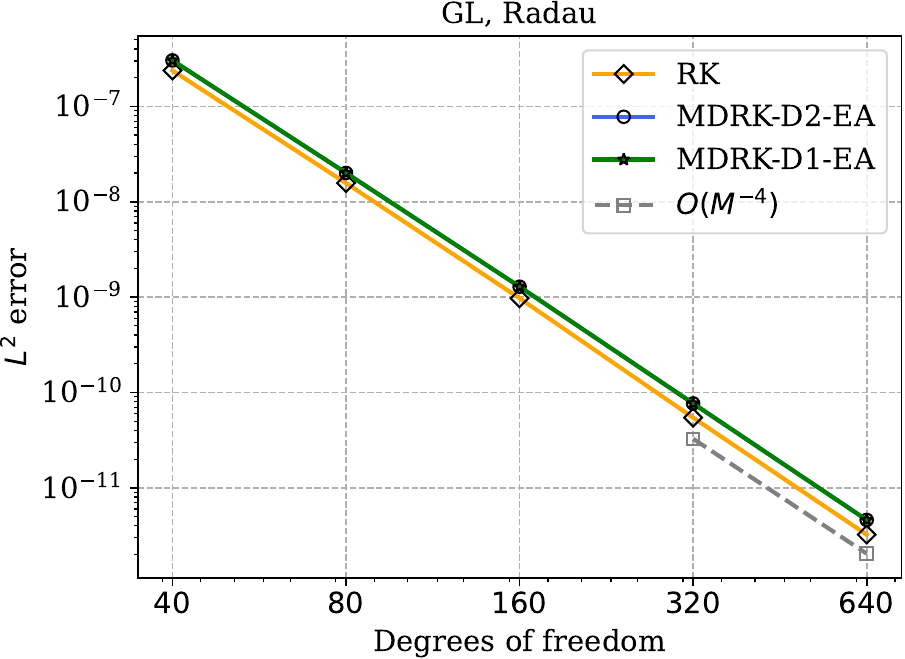}}
&
\resizebox{0.3\columnwidth}{!}{\includegraphics{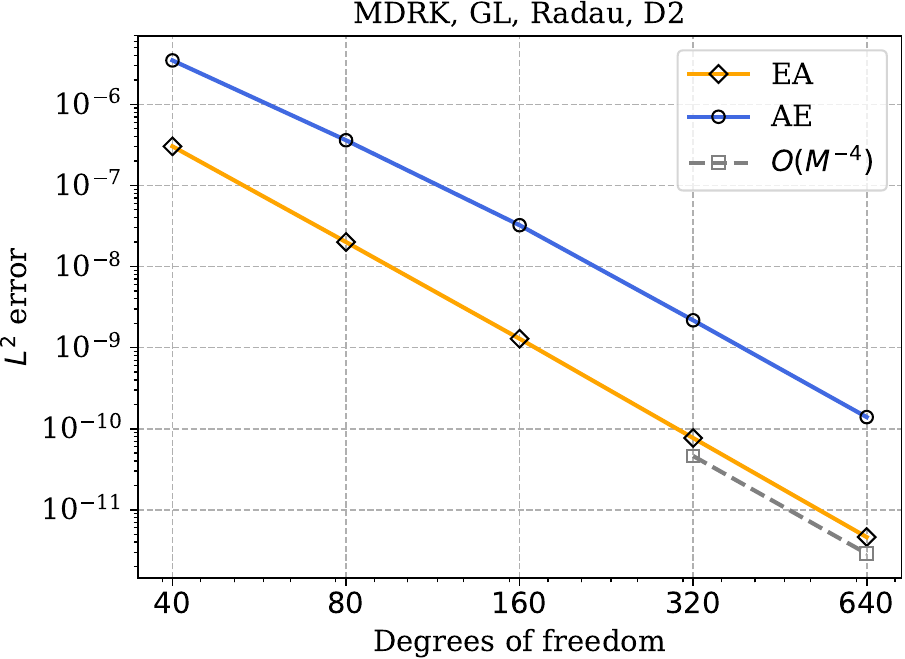}}\\
(a) & (b) & (c)
\end{tabular}
\caption{Error convergence for variable linear advection equation with $a
(x) = x^2$; (a) {\extrapolate} scheme, (b) {\evaluate} scheme, (c)
{\extrapolate} versus {\evaluate}\label{fig:mdrk.vla2}}
\end{figure}

\subsubsection{Burgers' equations}

The one dimensional Burger's equation is a conservation law of the form $u_t +
f (u)_x = 0$ with the quadratic flux $f (u) = u^2 / 2$. For the smooth initial
condition $u (x, 0) = 0.2 \sin (x)$ with periodic boundary condition in the
domain $[0, 2 \pi]$, we compute the numerical solution at time $t = 2.0$ when
the solution is still smooth. Figure~\ref{fig:mdrk.burg1}a compares the error norms
for the {\extrapolate} and {\evaluate} methods for the Rusanov numerical flux,
and using GL solution points, Radau correction and D2 dissipation. The
convergence rate of {\extrapolate} is less than optimal and close to $O (h^{3
+ 1 / 2})$. In Figure~\ref{fig:mdrk.burg1}b, we see that no scheme shows optimal
convergence rates when $g_2$ correction + GLL points is used. The comparison
between D1, D2 dissipation is made in Figure~\ref{fig:mdrk.burg2} and their
performances are found to be similar.

\begin{figure}
\centering
{\noindent}\begin{tabular}{cc}
\resizebox{0.47\columnwidth}{!}{\includegraphics{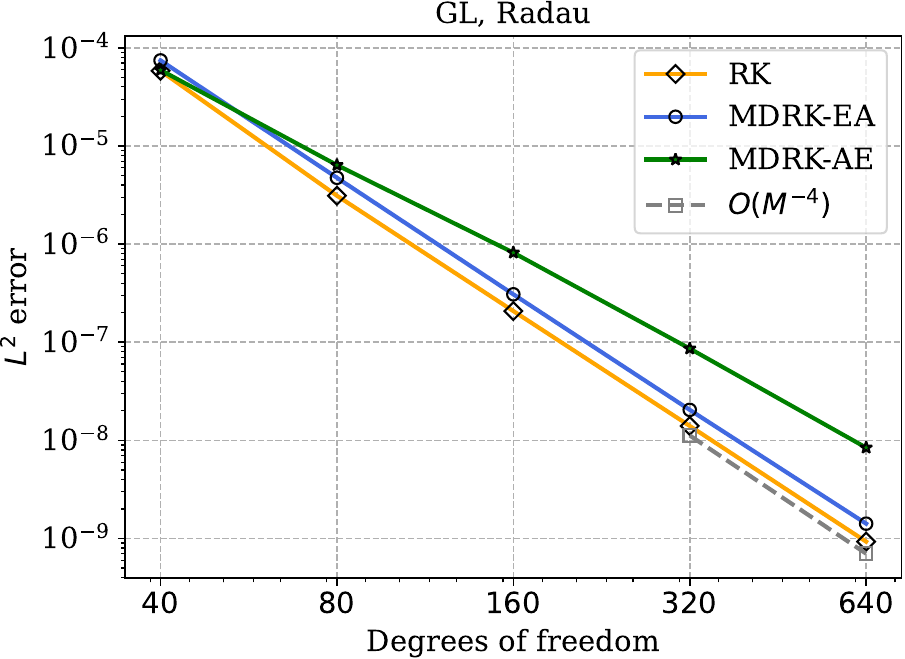}}
&
\resizebox{0.47\columnwidth}{!}{\includegraphics{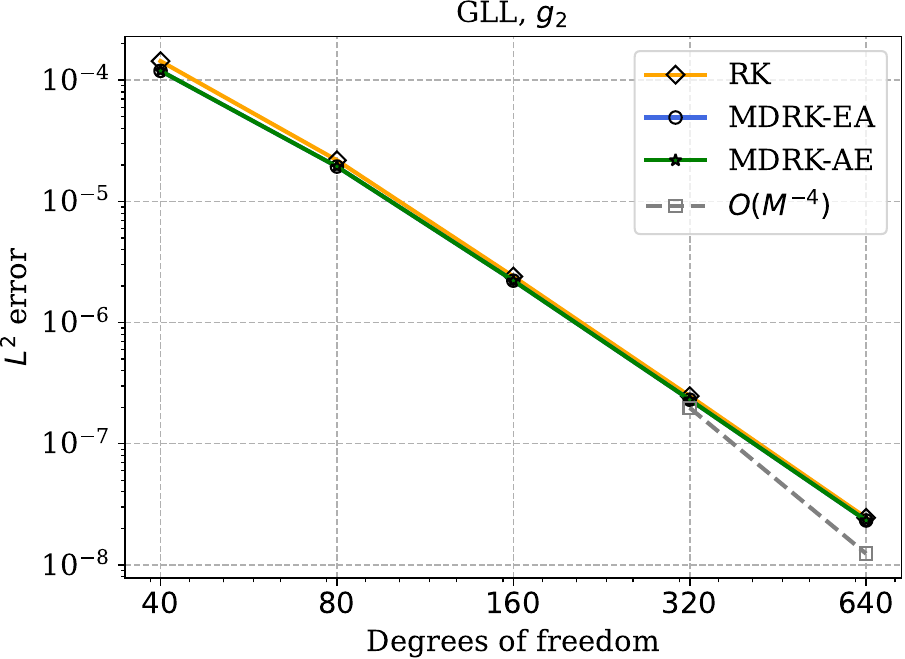}}\\
(a) & (b)
\end{tabular}
\caption{Comparing {\extrapolate} and {\evaluate} schemes using D2
dissipation for 1-D Burgers' equation at $t = 2$. (a) GL points with Radau
correction, (b) GLL points with $g_2$ correction \label{fig:mdrk.burg1}}
\end{figure}

\begin{figure}
\centering
{\noindent}\begin{tabular}{cc}
\resizebox{0.46\columnwidth}{!}{\includegraphics{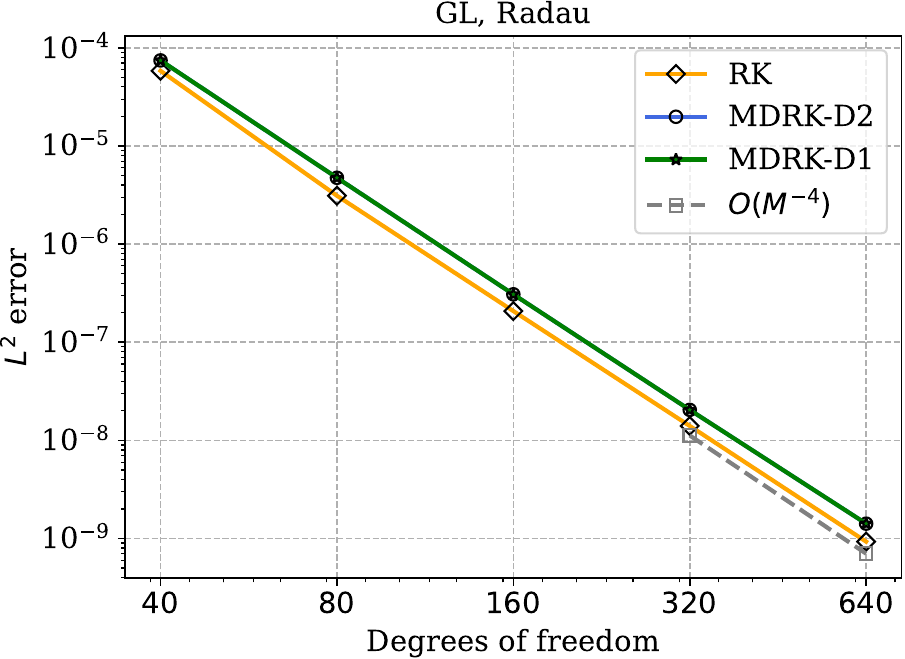}}
&
\resizebox{0.46\columnwidth}{!}{\includegraphics{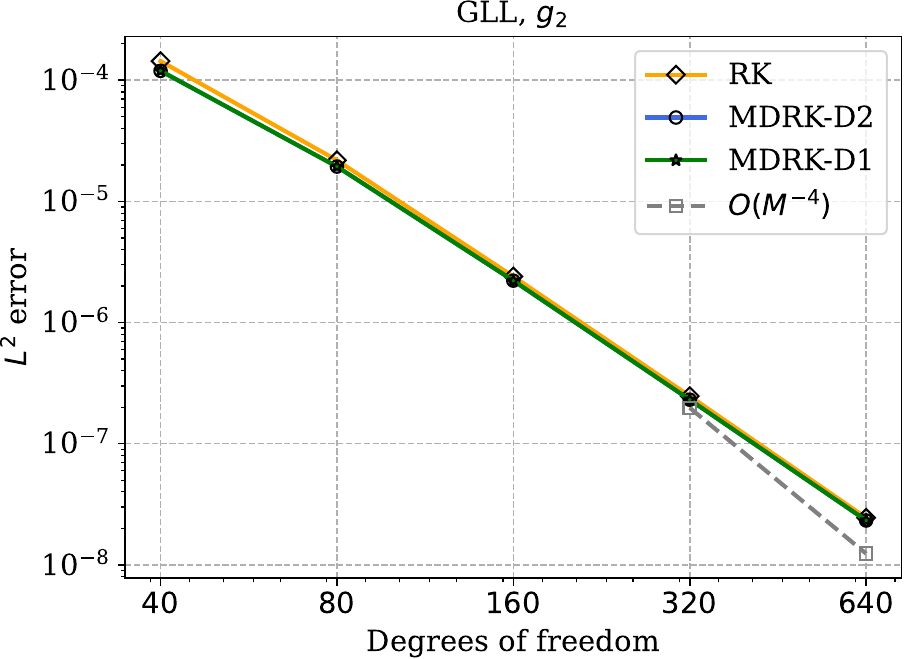}}\\
(a) & (b)
\end{tabular}
\caption{Comparing D1 and D2 dissipation for 1-D Burgers' equation at $t =
2$. (a) GL points with Radau correction, (b) GLL points with $g_2$
correction\label{fig:mdrk.burg2}}
\end{figure}

\subsection{1-D Euler equations}\label{sec:mdrk.res1dsys.chblend}

As an example of system of non-linear hyperbolic equations, we consider the
one-dimensional Euler equations of gas dynamics given by
\begin{equation}
\label{eq:mdrk.1deuler} \pd{}{t}  \left(\begin{array}{c}
\rho\\
\rho u\\
E
\end{array}\right) + \pd{}{x}  \left(\begin{array}{c}
\rho u\\
p + \rho u^2\\
(E + p) u
\end{array}\right) = \bzero
\end{equation}
where $\rho, u, p$ and $E$ denote the density, velocity, pressure and total
energy of the gas, respectively. For a polytropic gas, an equation of state $E
= E (\rho, u, p)$ which leads to a closed system is given by
\begin{equation}
\label{eq:mdrk.state} E = E (\rho, u, p) = \frac{p}{\gamma - 1} + \frac{1}{2}
\rho u^2
\end{equation}
where $\gamma > 1$ is the adiabatic constant, that will be taken as $1.4$
which is the value for air. In the following results, wherever it is not
mentioned, we use the Rusanov flux. We use Radau correction function with
Gauss-Legendre solution points in this section because of their superiority
seen in~{\cite{babbar2022,babbar2023admissibility}}. The time step size is
taken to be
\begin{equation}
\Delta t = C_s \min_e \left( \frac{\Delta x_e}{| \bar{v}_e | + \bar{c}_e}
\right)  \text{CFL} \label{eq:mdrk.dt.lw}
\end{equation}
where $e$ is the element index, $\bar{v}_e, \bar{c}_e$ are velocity and sound
speed of element mean in element $e$, CFL=0.107 (Table~\ref{tab:mdrk.cfl}). The
coefficient $C_s = 0.98$ is used in tests, unless specified otherwise.

\subsubsection{Blast wave}

The Euler equations~\eqref{eq:mdrk.1deuler} are solved with the initial condition
\[
(\rho, v, p) = \begin{cases}
(1, 0, 1000), & \text{if } x < 0.1\\
(1, 0, 0.01), & \text{if } 0.1 < x < 0.9\\
(1, 0, 100), & \text{if } x > 0.9
\end{cases}
\]
in the domain $[0, 1]$. This test was originally introduced
in~{\cite{Woodward1984}} to check the capability of the numerical scheme to
accurately capture the shock-shock interaction scenario. The boundaries are
set as solid walls by imposing the reflecting boundary conditions at $x = 0$
and $x = 1$. The solution consists of reflection of shocks and expansion waves
off the boundary wall and several wave interactions inside the domain. The
numerical solutions give negative pressure if the positivity correction is not
applied. With a grid of 400 cells using polynomial degree $N = 3$, we run the
simulation till the time $t = 0.038$ where a high-density peak profile is
produced. As tested in~{\cite{babbar2023admissibility}}, we compare first
order (FO) and MUSCL-Hancock (MH) blending schemes, and TVB limiter with
parameter $M = 300$~{\cite{Qiu2005b}} (TVB-300). We compare the performance of
limiters in Figure~\ref{fig:mdrk.blast} where the approximated density and
pressure profiles are compared with a reference solution computed using a very
fine mesh. Looking at the peak amplitude and contact discontinuity, it is
clear that the MUSCL-Hancock blending scheme gives the best resolution, especially
when compared with the TVB limiter.

\begin{figure}
\centering
{\noindent}\begin{tabular}{cc}
\resizebox{0.48\columnwidth}{!}{\includegraphics{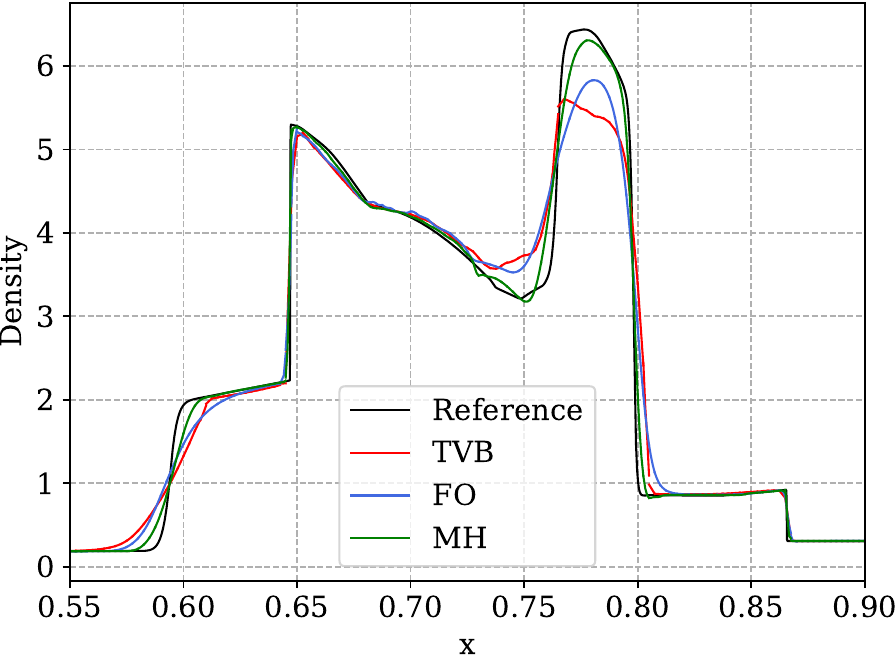}}
& \resizebox{0.48\columnwidth}{!}{\includegraphics{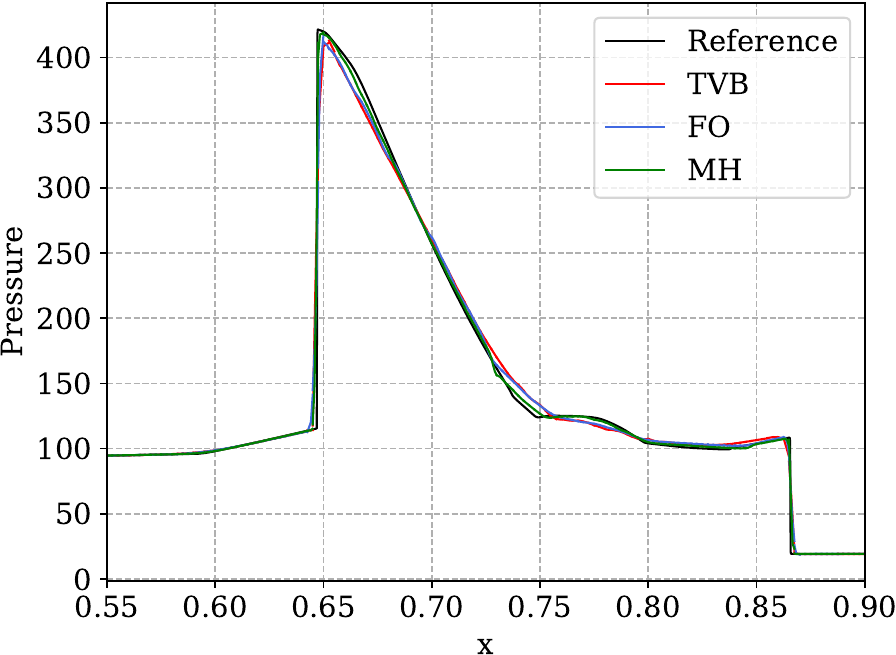}}\\
(a) & (b)
\end{tabular}
\caption{Blast wave problem using first order (FO) and MUSCL-Hancock
blending schemes, and TVB limited scheme (TVB-300) with parameter $M = 300$.
(a) Density, (b) Pressure profiles are shown at $t = 0.038$ on a mesh of 400
cells.\label{fig:mdrk.blast}}
\end{figure}

\subsubsection{Titarev Toro}

This is an extension of the Shu-Osher problem given by Titarev and
Toro~{\cite{Titarev2004}} and the initial data comprises of a severely
oscillatory wave interacting with a shock. This problem tests the ability of a
high-order numerical scheme to capture the extremely high frequency waves. The
smooth density profile passes through the shock and appears on the other side,
and its accurate computation is challenging due to numerical dissipation. Due
to presence of both spurious oscillations and smooth extrema, this becomes a
good test for testing robustness and accuracy of limiters. We discretize the
spatial domain with 800 cells using polynomial degree $N = 3$ and compare
blending schemes. As expected, the MUSCL-Hancock (MH) blending scheme is
superior to the First Order (FO) blending scheme and has nearly resolved in
the smooth extrema.

The initial condition is given by
\[ (\rho, v, p) = \begin{cases}
(1.515695, 0.523346, 1.805), \qquad & - 5 \leq x \leq - 4.5\\
(1 + 0.1 \sin (20 \pi x), 0, 1),  & - 4.5 < x \leq 5
\end{cases} \]
The physical domain is $[- 5, 5]$ and a grid of 800 elements is used. The
density profile at $t = 5$ is shown in Figure~\ref{fig:mdrk.titarev.toro}.

\begin{figure}
\centering
{\noindent}\begin{tabular}{cc}
\resizebox{0.485\columnwidth}{!}{\includegraphics{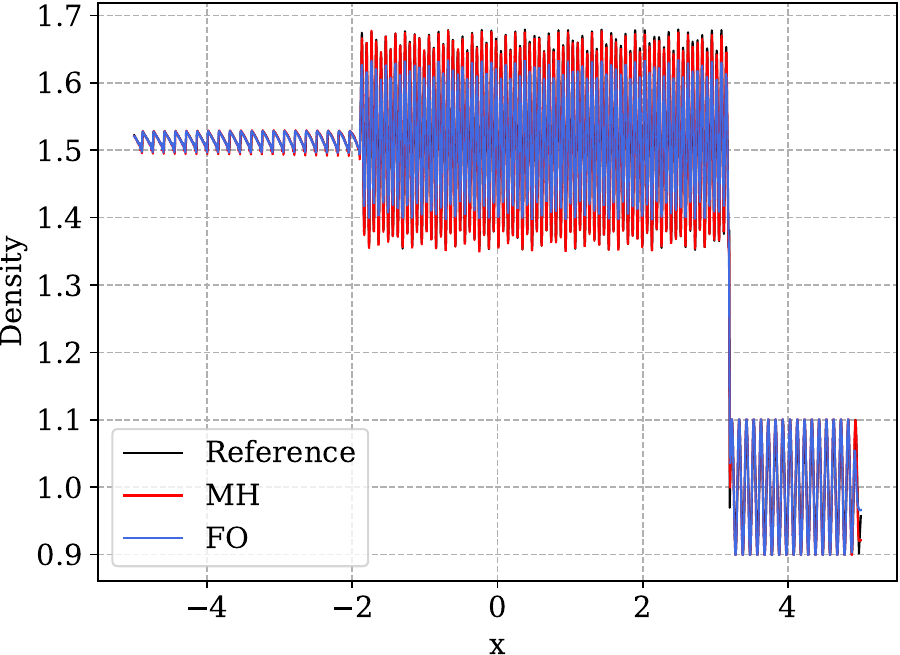}}
&
\resizebox{0.485\columnwidth}{!}{\includegraphics{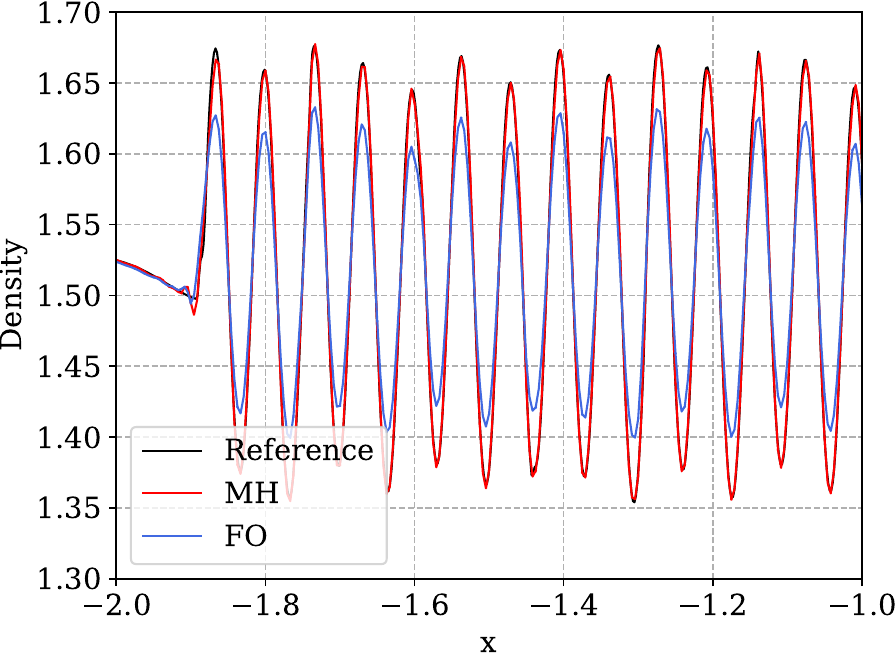}}\\
(a) & (b)
\end{tabular}
\caption{Titarev-Toro problem, comparing First Order (FO) and MUSCL-Hancock
(MH) blending (a) Complete plot, (b) Profile zoomed near smooth extrema on a
mesh of 800 cells.\label{fig:mdrk.titarev.toro}}
\end{figure}

\subsubsection{Large density ratio Riemann problem}

The second example is the large density ratio problem with a very strong
rarefaction wave~{\cite{Tang2006}}. The initial condition is given by
\[ (\rho, v, p) = \begin{cases}
(1000, 0, 1000), \quad & x < 0.3\\
(1, 0, 1),  & 0.3 < x
\end{cases} \]
The computational domain is $[0, 1]$ and transmissive boundary condition is
used at both ends. The density and pressure profile on a mesh of $500$
elements at $t = 0.15$ is shown in Figure~\ref{fig:mdrk.high.density}. The MH
blending scheme is giving better accuracy even in this tough problem.

\begin{figure}
\centering
{\noindent}\begin{tabular}{cc}
\resizebox{0.45\columnwidth}{!}{\includegraphics{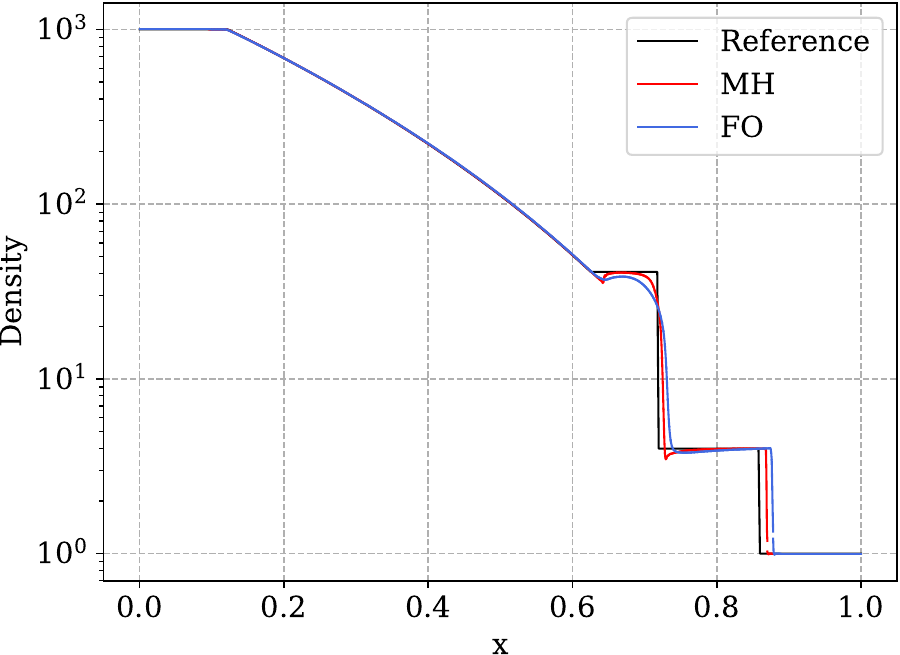}}
&
\resizebox{0.45\columnwidth}{!}{\includegraphics{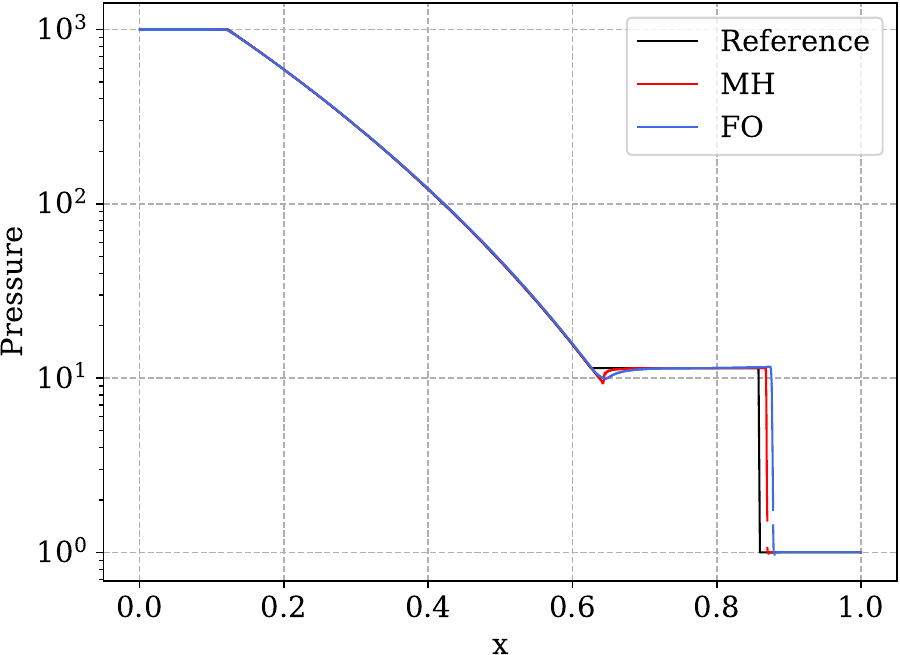}}\\
(a) & (b)
\end{tabular}
\caption{High density problem at $t = 0.15$ on a mesh of 500 elements (a)
Density plot, (b) Pressure plot\label{fig:mdrk.high.density}}
\end{figure}

\subsubsection{Sedov's blast}

To demonstrate the admissibility preserving property of our scheme, we
simulate Sedov's blast wave~{\cite{sedov1959}}; the test describes the
explosion of a point-like source of energy in a gas. The explosion generates a
spherical shock wave that propagates outwards, compressing the gas and
reaching extreme temperatures and pressures. The problem can be formulated in
one dimension as a special case, where the explosion occurs at $x = 0$ and the
gas is confined to the interval $[- 1, 1]$ by solid walls. For the simulation,
on a grid of 201 cells with solid wall boundary conditions, we use the
following initial data~{\cite{zhang2012}},
\[ \rho = 1, \qquad v = 0, \qquad E (x) = \left\{\begin{array}{ll}
\frac{3.2 \times 10^6}{\Delta x}, \qquad & |x| \le \frac{\Delta x}{2}\\
10^{- 12}, \qquad & \text{otherwise}
\end{array}\right. \]
where $\Delta x$ is the element width. This is a difficult test for positivity
preservation because of the high pressure ratios. Nonphysical solutions are
obtained if the proposed admissibility preservation corrections are not
applied. The density and pressure profiles at $t = 0.001$ are obtained using
blending schemes are shown in Figure~\ref{fig:mdrk.sedov.blast}.
In~{\cite{babbar2023admissibility}}, TVB limiter could not be used in this
test as the proof of admissibility preservation depended on the blending
scheme. Here, by using the generalized admissibility preserving scheme
of~{\cite{babbar2024generalized}} to be able to use the TVB limiter. However,
we TVB limiter being less accurate and unable to control the oscillations..

\begin{figure}
\centering
{\noindent}\begin{tabular}{cc}
\resizebox{0.485\columnwidth}{!}{\includegraphics{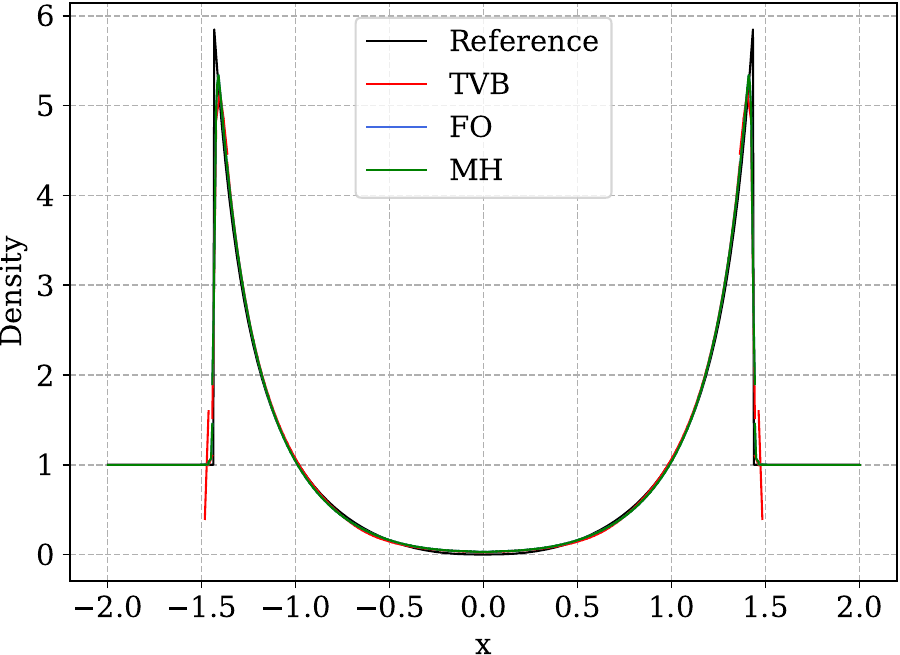}}
& \resizebox{0.485\columnwidth}{!}{\includegraphics{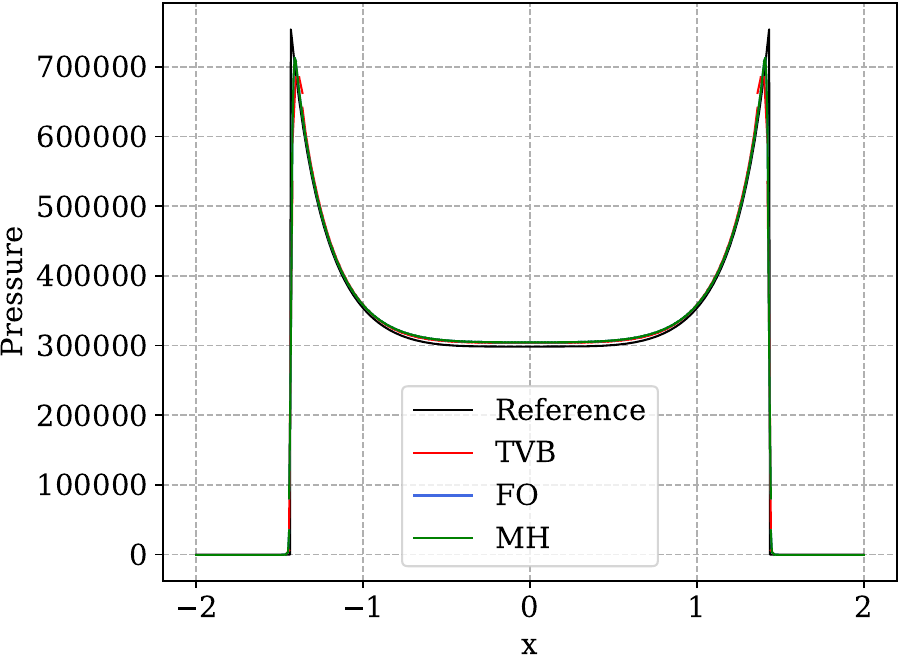}}\\
(a) & (b)
\end{tabular}
\caption{Sedov's blast wave problem, numerical solution using first order
(FO) and MUSCL-Hancock blending schemes, and TVD (a) Density, (b) Pressure
profiles are shown at $t = 0.001$ on a mesh of 201
cells\label{fig:mdrk.sedov.blast}}
\end{figure}

\subsection{2-D Euler's equations} \label{sec:2d.euler}

We consider the two-dimensional Euler equations of gas dynamics given by
\begin{equation}
\label{eq:mdrk.2deuler} \pd{}{t}  \left(\begin{array}{c}
\rho\\
\rho u\\
\rho v\\
E
\end{array}\right) + \pd{}{x}  \left( \begin{array}{c}
\rho u\\
p + \rho u^2\\
\rho uv\\
(E + p) u
\end{array} \right) + \pd{}{y}  \left( \begin{array}{c}
\rho v\\
\rho uv\\
p + \rho v^2\\
(E + p) v
\end{array} \right) = \bzero
\end{equation}
where $\rho, p$ and $E$ denote the density, pressure and total energy of the
gas, respectively and $(u, v)$ are Cartesian components of the fluid velocity.
For a polytropic gas, an equation of state $E = E (\rho, u, v, p)$ which leads
to a closed system is given by
\begin{equation}
\label{eq:mdrk.2dstate} E = E (\rho, u, v, p) = \frac{p}{\gamma - 1} +
\frac{1}{2} \rho (u^2 + v^2)
\end{equation}
where $\gamma > 1$ is the adiabatic constant, that will be taken as $1.4$ in
the numerical tests, which is the typical value for air. The time step size
for polynomial degree $N$ is computed as
\begin{equation}
\Delta t = C_s \min_e \left( \frac{| \bar{u}_e | + \bar{c}_e}{\Delta x_e} +
\frac{| \bar{v}_e | + \bar{c}_e}{\Delta y_e} \right)^{- 1} \text{CFL}
\label{eq:mdrk.time.step.2d}
\end{equation}
where $e$ is the element index, $(\bar{u}_e, \bar{v}_e), \bar{c}_e$ are
velocity and sound speed of element mean in element $e$, CFL=0.107
(Table~\ref{tab:mdrk.cfl}) and $C_s \leq 1$ is a safety factor.

For verification of some our numerical results and to demonstrate the accuracy
gain observed in~{\cite{babbar2023admissibility}} of using MUSCL-Hancock
reconstruction using Gauss-Legendre points, we will compare our results with
the first order blending scheme using Gauss-Legendre-Lobatto (GLL) points
of~{\cite{hennemann2021}} available in
\tmtexttt{Trixi.jl}~{\cite{Ranocha2022}}. The accuracy benefit is expected
since GL points and quadrature are more accurate than GLL points, and
MUSCL-Hancock is also more accurate than first order finite volume method.

\subsubsection{Double Mach reflection}

This test case was originally proposed by Woodward and
Colella~{\cite{Woodward1984}} and consists of a shock impinging on a
wedge/ramp which is inclined by 30 degrees. An equivalent problem is obtained
on the rectangular domain $\Omega = [0, 4] \times [0, 1]$ obtained by rotating
the wedge so that the initial condition now consists of a shock angled at 60
degrees. The solution consists of a self similar shock structure with two
triple points. Define $\uu_b = \uu_b (x, y, t)$ in primitive variables as
\[ (\rho, u, v, p) = \begin{cases}
(8, 8.25 \cos (\frac{\pi}{6}), - 8.25 \sin (\frac{\pi}{6}), 116.5), &
\text{if } x < \frac{1}{6} + \frac{y + 20 t}{\sqrt{3}}\\
(1.4, 0, 0, 1), & \text{if } x > \frac{1}{6} + \frac{y + 20 t}{\sqrt{3}}
\end{cases} \]
and take the initial condition to be $\uu_0 (x, y) = \uu_b (x, y, 0)$. With
$\uu_b$, we impose inflow boundary conditions on the left side $\{0\} \times
[0, 1]$, outflow boundary conditions both on $[0, 1 / 6] \times \{0\}$ and
$\{4\} \times [0, 1]$, reflecting boundary conditions on $[1 / 6, 4] \times
\{0\}$ and inflow boundary conditions on the upper side $[0, 4] \times \{1\}$.

The simulation is run on a mesh of $600 \times 150$ elements using degree $N = 3$ polynomials up to time $t = 0.2$. In Figure~\ref{fig:mdrk.dmr}, we compare the results of \tmtexttt{Trixi.jl} with the MUSCL-Hancock blended scheme zoomed near the primary triple point. As expected, the small scale structures are captured better by the MUSCL-Hancock blended scheme.

\begin{figure}
\centering
\resizebox{0.8\columnwidth}{!}{\includegraphics{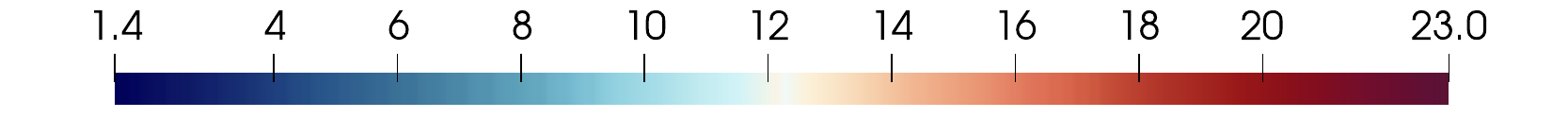}}\\
{\noindent}\begin{tabular}{cc}
\resizebox{0.45\columnwidth}{!}{\includegraphics{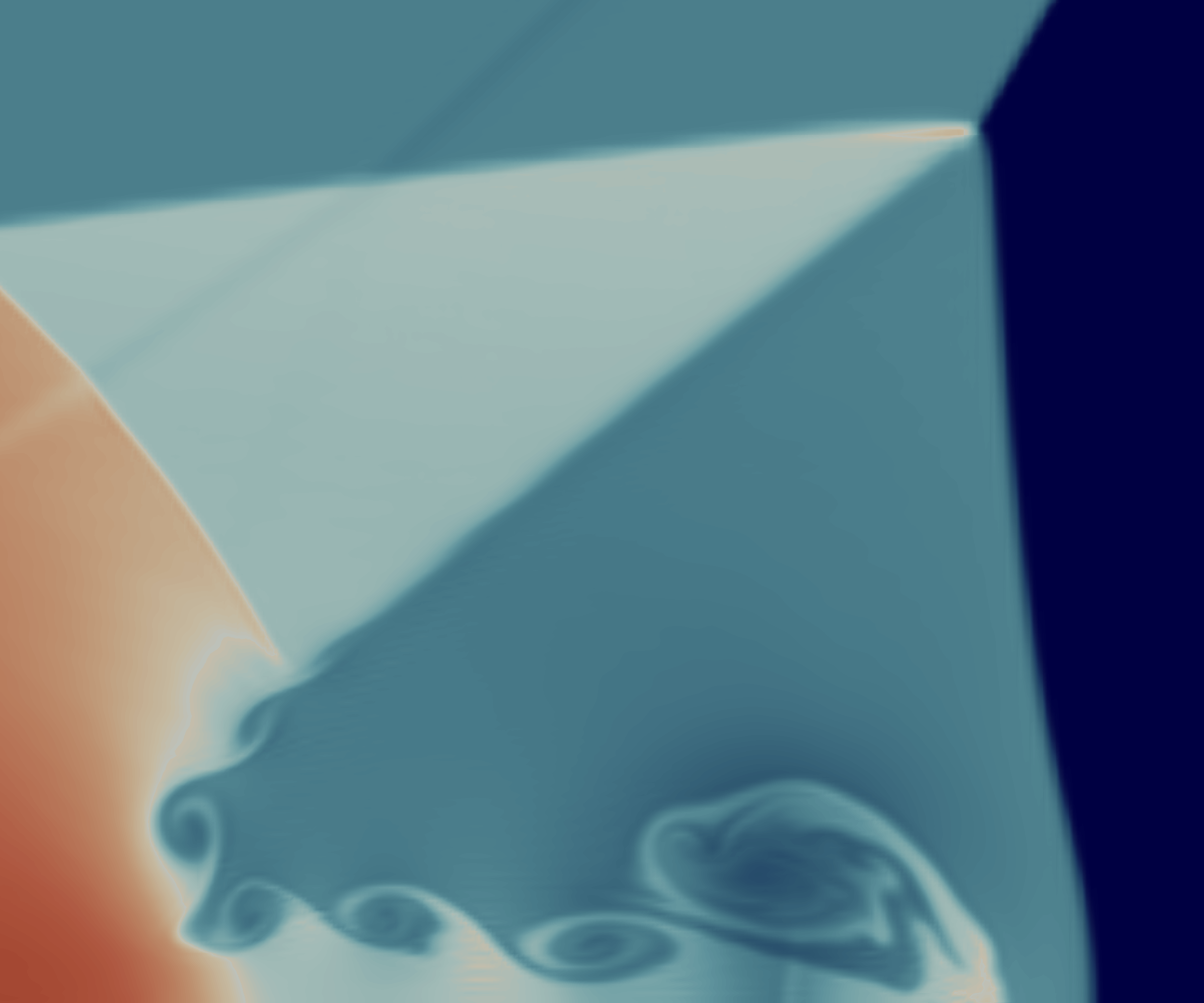}}
&
\resizebox{0.45\columnwidth}{!}{\includegraphics{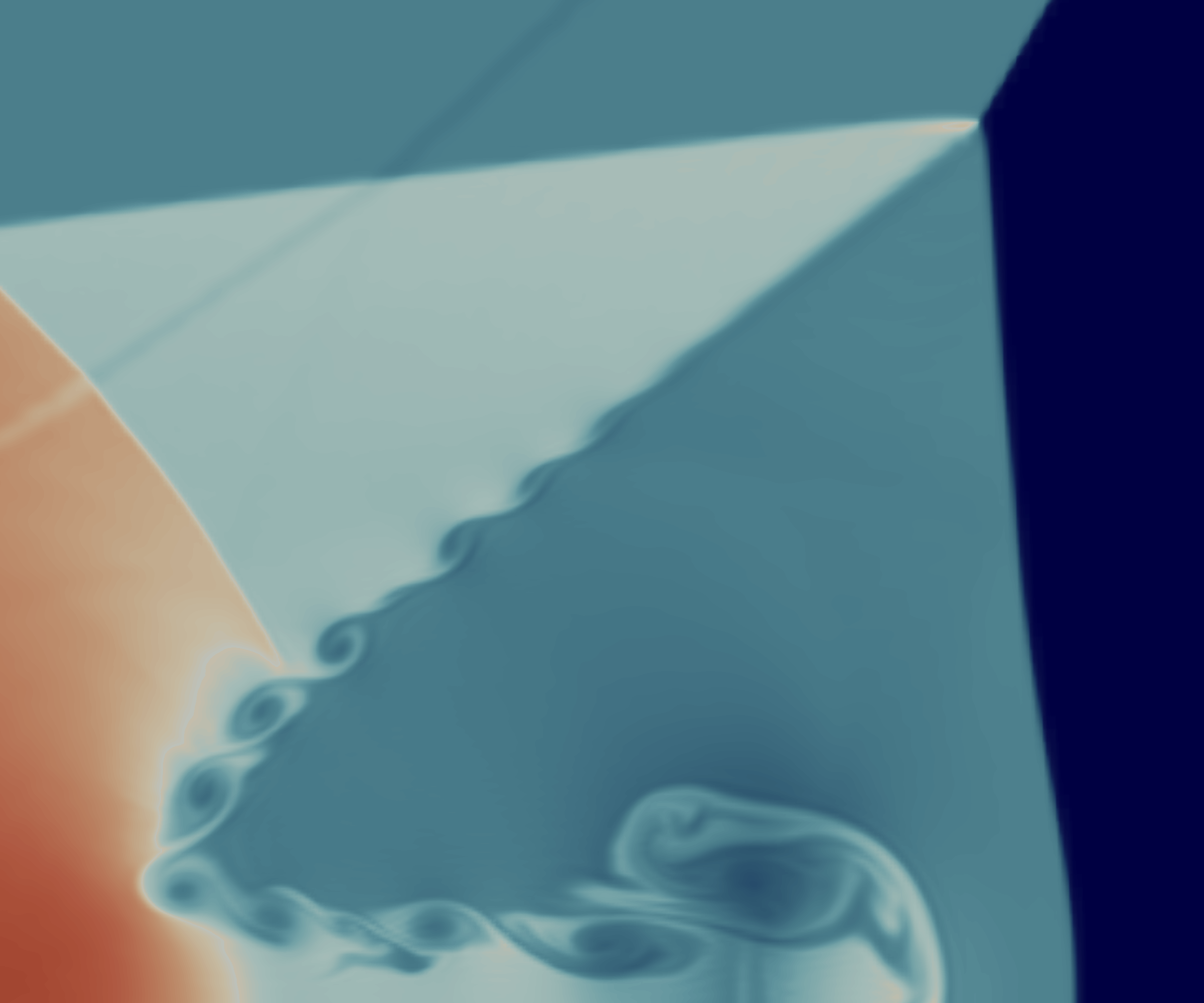}}\\
(a) \tmverbatim{Trixi.jl} & (b) MDRK
\end{tabular}
\caption{Double Mach reflection problem, density plot of numerical solution at $t = 0.2$ on a $600 \times 150$ mesh zoomed near the primary triple point.\label{fig:mdrk.dmr}}
\end{figure}

\subsubsection{Rotational low density problem}

This problems are taken from~{\cite{Pan2016}} where the solution consists of
hurricane-like flow evolution, where solution has one-point vacuum in the
center with rotational velocity field. The initial condition is given by
\[ (\rho, u, v, p) = (\rho_0, v_0 \sin \theta, - v_0 \cos \theta, A
\rho_0^{\gamma}) \]
where $\theta = \arctan (y / x)$, $A = 25$ is the initial entropy, $\rho = 1$
is the initial density, gas constant $\gamma = 2$. The initial velocity
distribution has a nontrivial transversal component, which makes the flow
rotational. The solutions are classified~{\cite{Zhang1997}} into three types
according to the initial Mach number $M_0 = | v_0 | / c_0$, where $c_0 = p'
(\rho_0) = A \gamma \rho_0^{\gamma - 1}$ is the sound speed.
\begin{enumerate}
\item \tmtextbf{ Critical rotation with $M_0 = \sqrt{2}$.} This test has an
exact solution with explicit formula. The solution consists of two parts: a
far field solution and a near-field solution. The former far field solution
is defined for $r \geq 2 t \sqrt{p' (\rho_0)}$, $r = \sqrt{x^2 + y^2}$,
\begin{equation}
\begin{split}
U (x, y, t) & = \frac{1}{r}  (2 tp_0' \cos \theta + \sqrt{2 p_0'}
\sqrt{r^2 - 2 t^2 p_0'} \sin \theta)\\
V (x, y, t) & = \frac{1}{r}  (2 tp_0' \sin \theta - \sqrt{2 p_0'}
\sqrt{r^2 - 2 t^2 p_0'} \cos \theta)\\
\rho (x, y, t) & = \rho_0\\
p (x, y, t) & = A \rho_0^{\gamma}
\end{split} \label{eq:mdrk.far.field.solution}
\end{equation}
and the near-field solution is defined for $r < 2 t \sqrt{p' (\rho_0)}$ as
\[ U (x, y, t) = \frac{x + y}{2 t}, \quad V (x, y, t) = \frac{- x + y}{2 t},
\quad \rho (x, y, t) = \frac{r^2}{8 At^2} \]
The curl of the velocity in the near-field is
\[ \text{curl} (U, V) = V_x - U_y = - \frac{1}{2 t} \neq 0 \]
and the solution has one-point vacuum at the origin $r = 0$. This is a typical hurricane-like solution that behaves highly singular, particularly near the
origin $r = 0$. There are two issues here challenging the numerical schemes:
one is the presence of the vacuum state which examines whether a high order
scheme can keep the positivity preserving property; the other is the
rotational velocity field for testing whether a numerical scheme can
preserve the symmetry. In this regime, we take $v_0 = 10$ on the
computational domain $[- 1, 1]^2$ with $\Delta x = \mathLaplace y = 1 /
100$. The boundary condition is given by the far field solution
in~\eqref{eq:mdrk.far.field.solution}.

\item \tmtextbf{High-speed rotation with $M_0 > \sqrt{2}$.} For this case,
$v_0 = 12.5$, so that the density goes faster to the vacuum and the fluid
rotates severely. The physical domain is $[- 2, 2]^2$ and the grid spacing
is $\mathLaplace x = \mathLaplace y = 1 / 100$. Outflow boundary conditions
are given on the boundaries.Because of the higher rotation speed, this case
is tougher than the first one, and can be used to validate the robustness of
the higher-order scheme.

\item \tmtextbf{Low-speed rotation with $M_0 < \sqrt{2}$.} In this test
case, we take $v_0 = 7.5$ making it a rotation with lower speed than the
previous tests. The outflow boundary conditions are given as in the previous
tests. The simulation is performed in the domain $[- 1, 1]^2$ till $t =
0.045$. The symmetry of flow structures is preserved.
\end{enumerate}
The density profile for the flow with critical speed are shown in
Figure~\ref{fig:mdrk.hurricane.critical} including a comparison with exact solution
at a line cut of $y = 0$ in Figure~\ref{fig:mdrk.hurricane.critical}b, showing near
overlap. In Figure~\ref{fig:mdrk.rotational.all.speed}a, we show the line cut of
density profile at $y = 0$ for the three rotation speeds. In
Figure~\ref{fig:mdrk.rotational.all.speed}b, we show streamlines for high
rotational speed, showing symmetry.

\begin{figure}
\centering
{\noindent}\begin{tabular}{cc}
\resizebox{0.38\columnwidth}{!}{\includegraphics{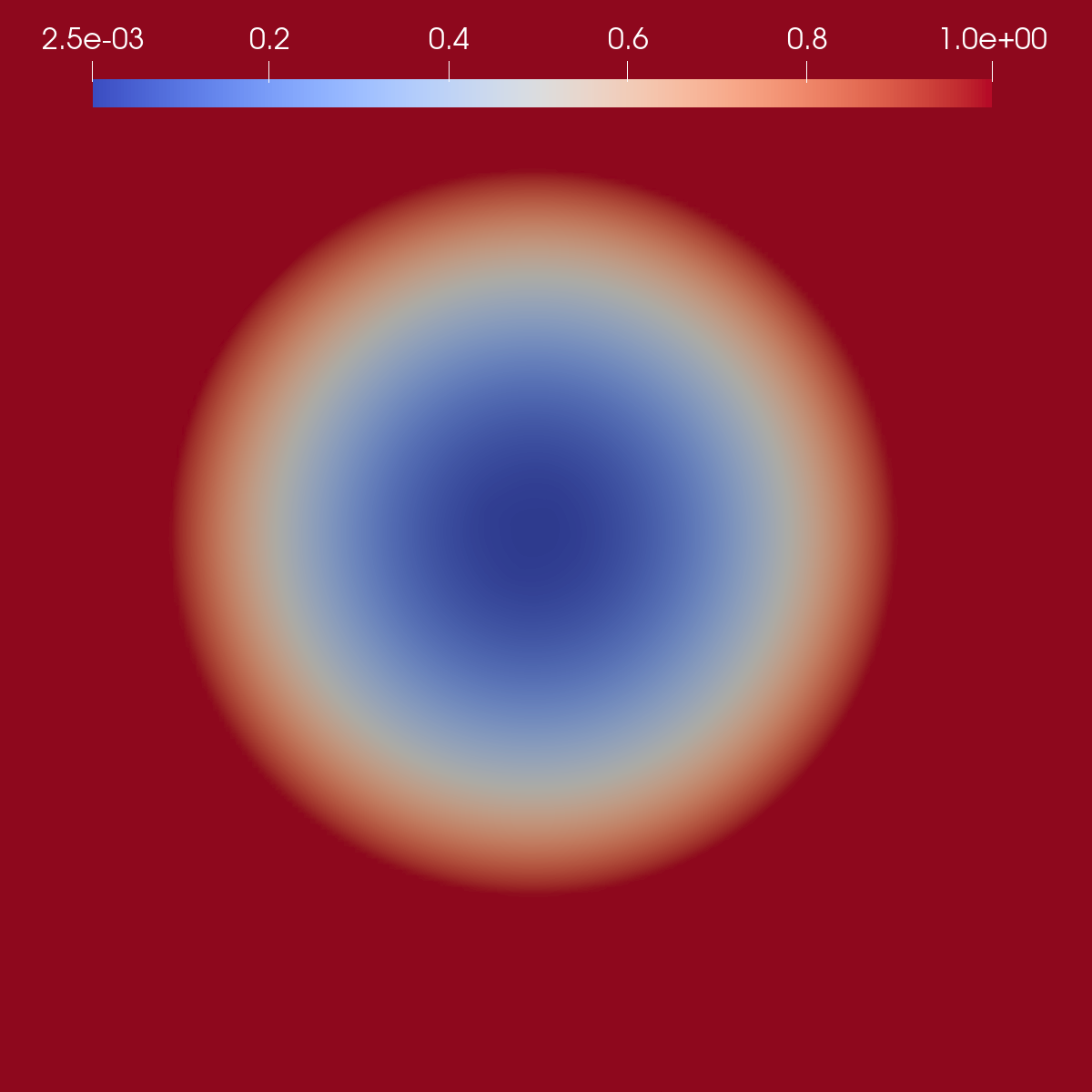}}
&
\resizebox{0.52\columnwidth}{!}{\includegraphics{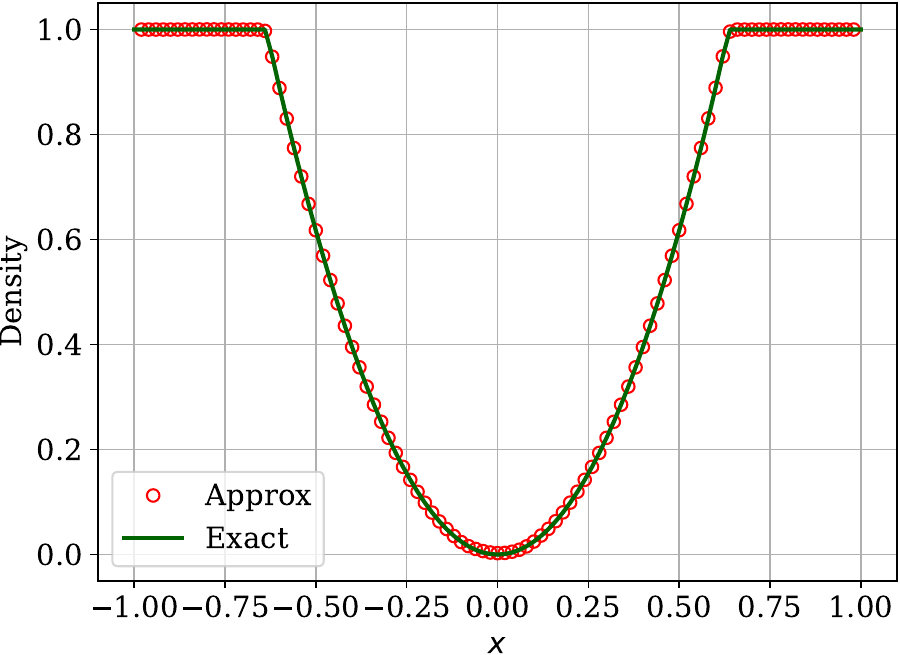}}\\
(a) & (b)
\end{tabular}
\caption{Density profile of rotational low density problem at critical speed
(a) Pseudocolor plot (b) Line cut at $y = 0$ on a mesh with $\mathLaplace x
= \mathLaplace y = 1 / 100$.\label{fig:mdrk.hurricane.critical}}
\end{figure}

\begin{figure}
\centering
{\noindent}\begin{tabular}{cc}
\resizebox{0.52\columnwidth}{!}{\includegraphics{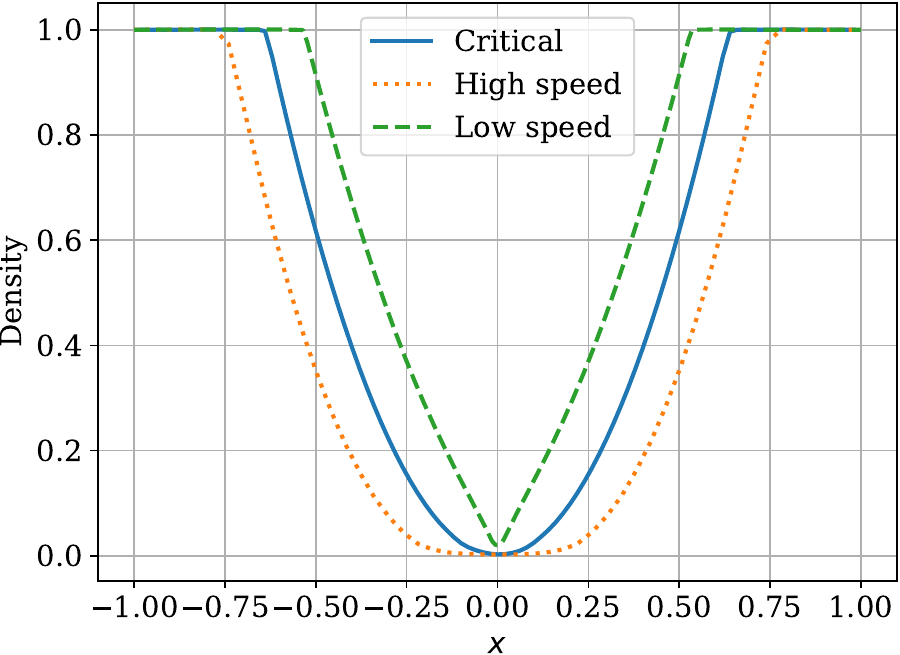}}
&
\resizebox{0.38\columnwidth}{!}{\includegraphics{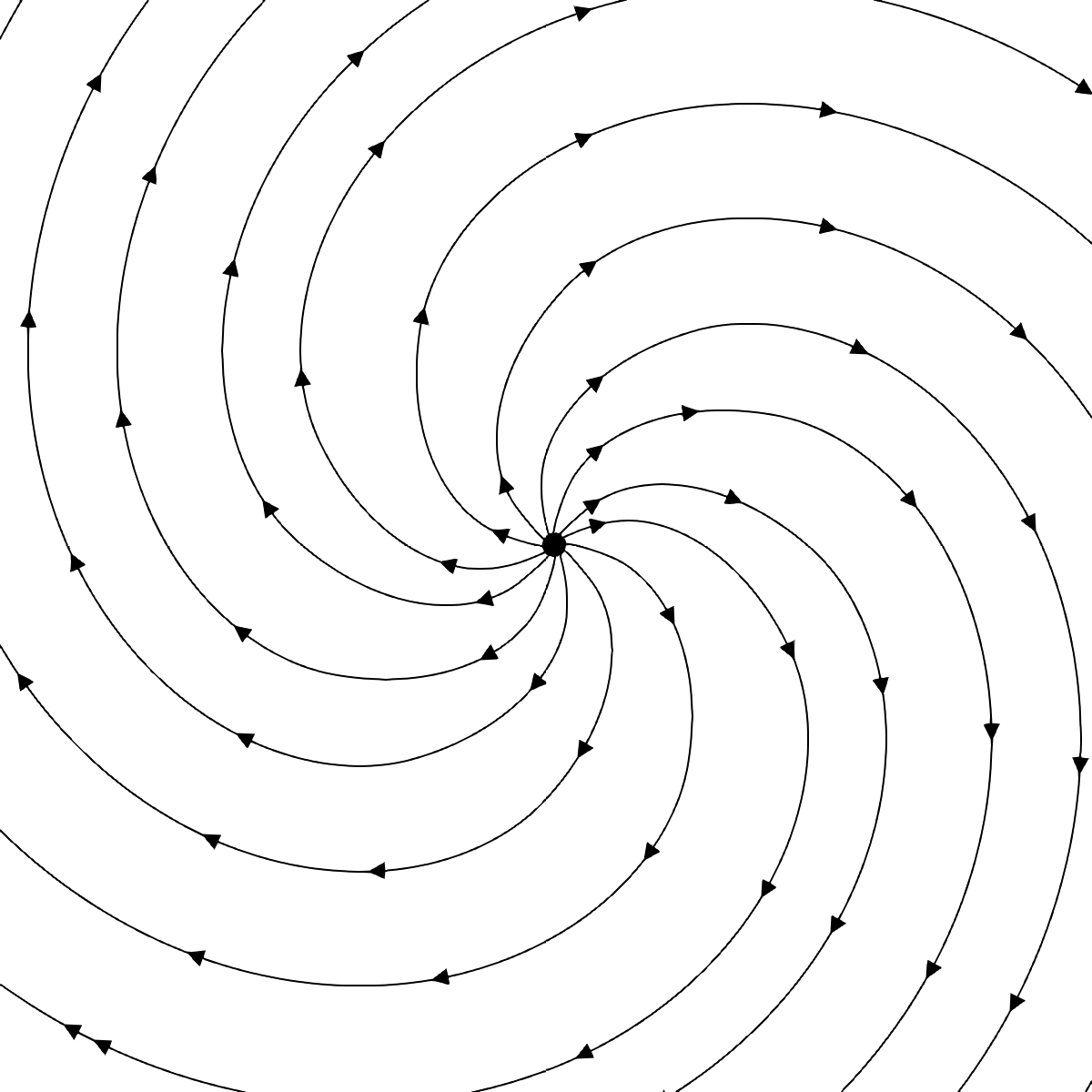}}\\
(a) & (b)
\end{tabular}
\caption{Rotational low density problem (a) Density profile line cut at $y =
0$ for different rotational speeds, (b) Stream lines for high rotational
speed\label{fig:mdrk.rotational.all.speed}}
\end{figure}

\subsubsection{Two Dimensional Riemann problem}

2-D Riemann problems consist of four constant states and have been studied
theoretically and numerically for gas dynamics in~{\cite{Glimm1985}}. We
consider this problem in the square domain $[0, 1]^2$ where each of the four
quadrants has one constant initial state and every jump in initial condition
leads to an elementary planar wave, i.e., a shock, rarefaction or contact
discontinuity. There are only 19 such genuinely different configurations
possible~{\cite{Zhang1990,Lax1998}}. As studied in~{\cite{Zhang1990}}, a
bounded region of subsonic flows is formed by interaction of different planar
waves leading to appearance of many complex structures depending on the
elementary planar flow. We consider configuration 12 of~{\cite{Lax1998}}
consisting of 2 positive slip lines and two forward shocks, with initial
condition
\[ (\rho, u, v, p) = \begin{cases}
(0.5313, 0, 0, 0.4) \qquad & \text{if } x \ge 0.5, y \ge 0.5\\
(1, 0.7276, 0, 1) & \text{if } x < 0.5, y \ge 0.5\\
(0.8, 0, 0, 1) & \text{if } x < 0.5, y < 0.5\\
(1, 0, 0.7276, 1) & \text{if } x \ge 0.5, y < 0.5
\end{cases} \]
The simulations are performed with transmissive boundary conditions on an enlarged domain \correction{up to} time $t = 0.25$. The density profiles obtained from the MUSCL-Hancock blending scheme and \tmtexttt{Trixi.jl} are shown in Figure~\ref{fig:mdrk.rp2d}. We see that both schemes give similar resolution in most regions. The MUSCL-Hancock blending scheme gives better resolution of the small scale structures arising across the slip lines.

\begin{figure}
\centering
\resizebox{0.7\columnwidth}{!}{\includegraphics{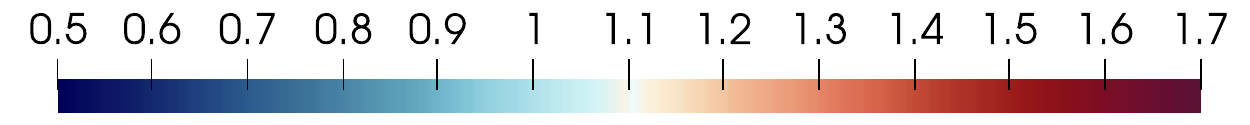}} \\
{\noindent}\begin{tabular}{cc}
\resizebox{0.46\columnwidth}{!}{\includegraphics{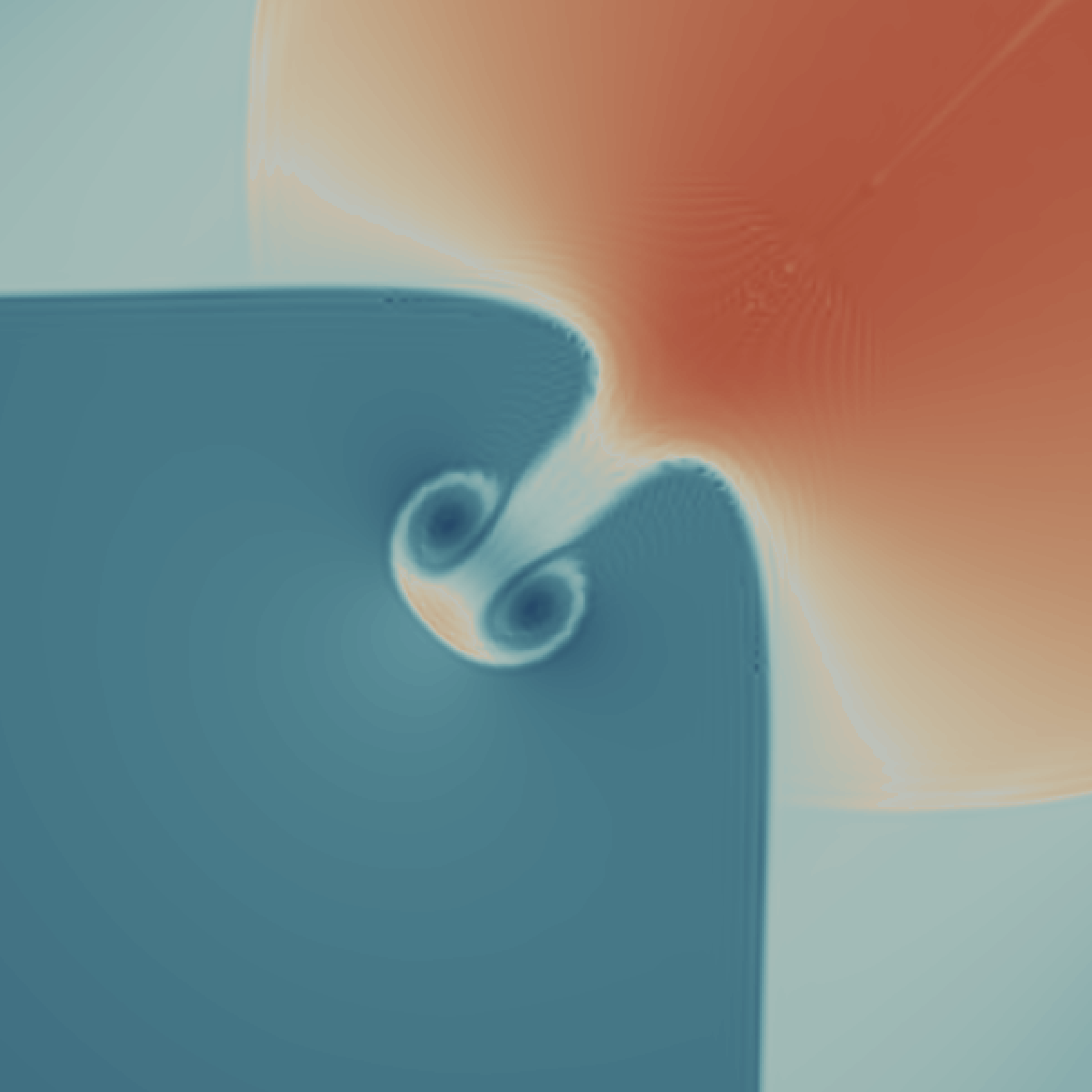}}
&
\resizebox{0.46\columnwidth}{!}{\includegraphics{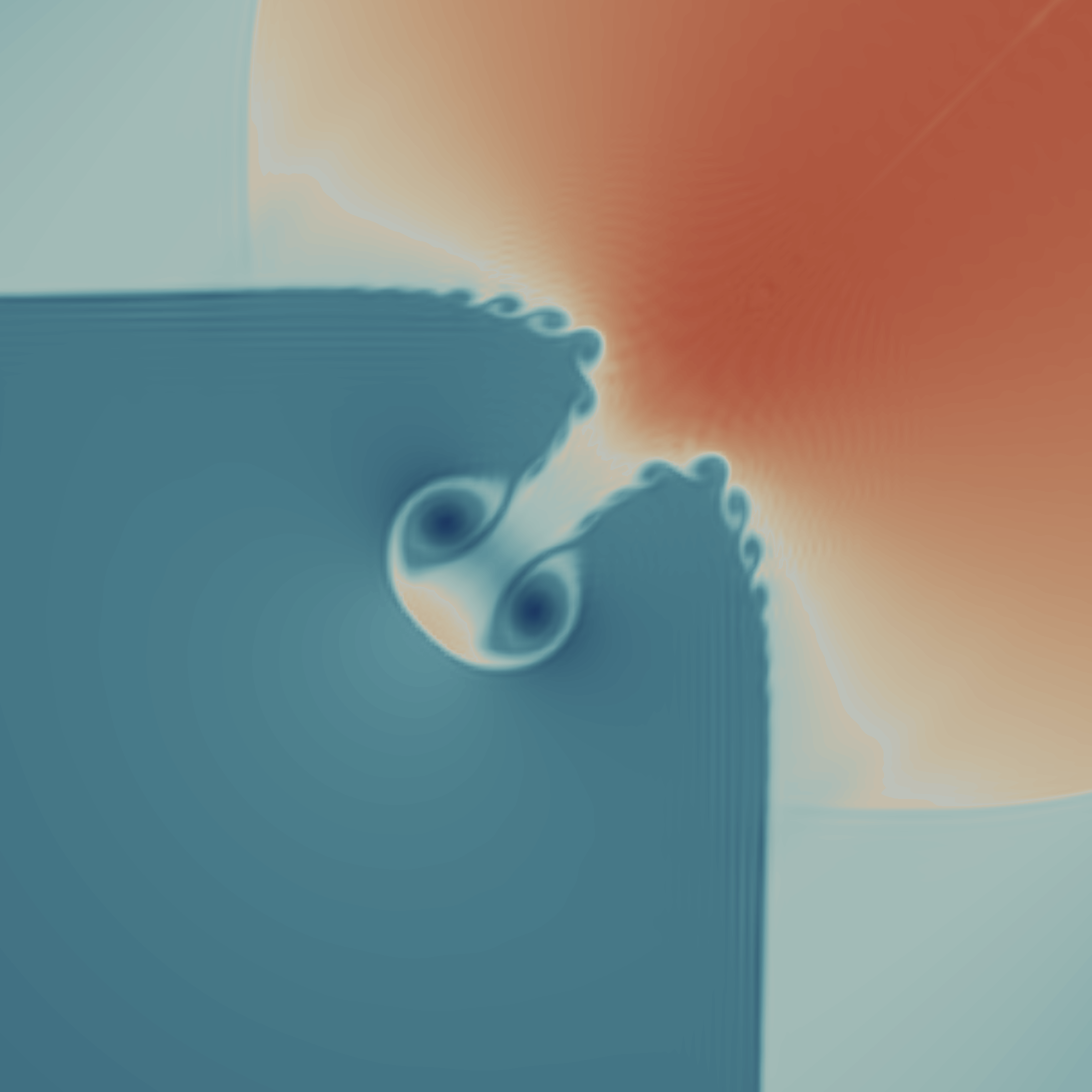}}\\
(a)
\tmverbatim{Trixi.jl} & (b) MDRK
\end{tabular}
\caption{2-D Riemann problem, density plots of numerical solution at $t =
0.25$ for degree $N = 3$ on a $256 \times 256$ mesh. \label{fig:mdrk.rp2d}}
\end{figure}

\subsubsection{Rayleigh-Taylor instability}

The last problem is the Rayleigh-Taylor instability to test the performance of
higher-order scheme for the conservation laws with source terms, and the
governing equations are written as
\[ \pd{}{t}  \left(\begin{array}{c}
\rho\\
\rho u\\
\rho v\\
E
\end{array}\right) + \pd{}{x}  \left( \begin{array}{c}
\rho u\\
p + \rho u^2\\
\rho uv\\
(E + p) u
\end{array} \right) + \pd{}{y}  \left( \begin{array}{c}
\rho v\\
\rho uv\\
p + \rho v^2\\
(E + p) v
\end{array} \right) = \left( \begin{array}{c}
0\\
0\\
\rho\\
\rho v
\end{array} \right) \]
It is also a test to check the suitability of higher-order schemes for the capturing of interface instabilities. The implementation of MDRK with source terms is explained in Appendix~\ref{sec:source.terms} where an approximate Lax-Wendroff procedure is also applied on the source term. The following description of this test is from~{\cite{Pan2016}}. The Rayleigh-Taylor instability happens on the interface between fluids with different densities when an acceleration is directed from the heavy fluid to the light one. The instability with fingering nature generates bubbles of light fluid rising into the ambient heavy fluid and spikes of heavy fluid falling into the light fluid. The initial condition of this problem~{\cite{Shi2003}} is given as follows
\[ (\rho, u, v, p) = \begin{cases}
(2, 0, - 0.025 a \cos (8 \pi x), 2 y + 1), \quad & y \leq 0.5\\
(1, 0, - 0.025 a \cos (8 \pi x), y + 1.5),  & y > 0.5
\end{cases} \]
where $a = \sqrt{\gamma p / \rho}$ is the sound speed, $\gamma = 5 / 3$ and  the computational domain is $[0, 0.25] \times [0, 1]$. The reflecting boundary
conditions are imposed for the left and right boundaries. At the top boundary,
the flow variables are set as $(\rho, u, v, p) = (1, 0, 0, 2.5)$. At the
bottom boundary, they are $(\rho, u, v, p) = (2, 0, 0, 1)$. The uniform mesh
with $64 \times 256$ elements is used in the simulation. The density
distributions at $t = 1.5, 1.75, 2, 2.25, 2.5$ are presented in
Figure~\ref{fig:mdrk.rayleigh.taylor} which follows qualitatively the physical picture described above.
\begin{figure}
\centering
\resizebox{0.7\columnwidth}{!}{\includegraphics{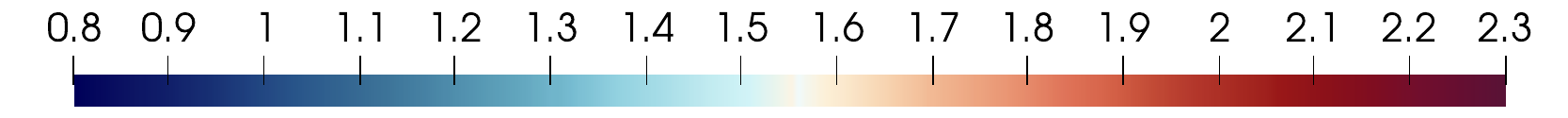}} \\
{\noindent}\begin{tabular}{ccccc}
\resizebox{0.18\columnwidth}{!}{\includegraphics{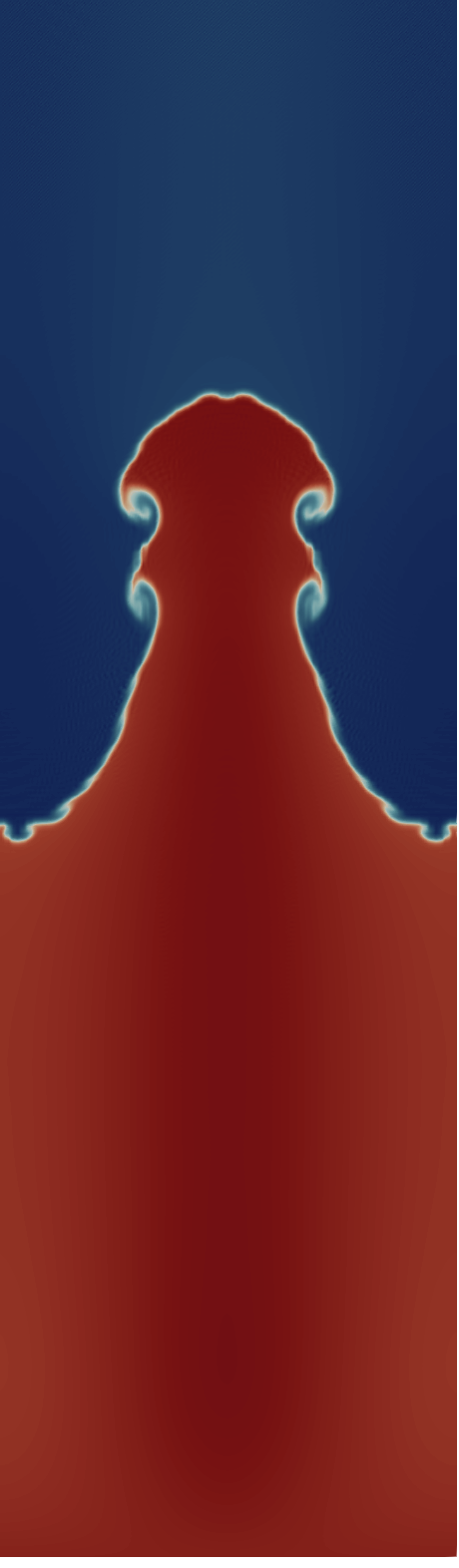}}
&
\resizebox{0.18\columnwidth}{!}{\includegraphics{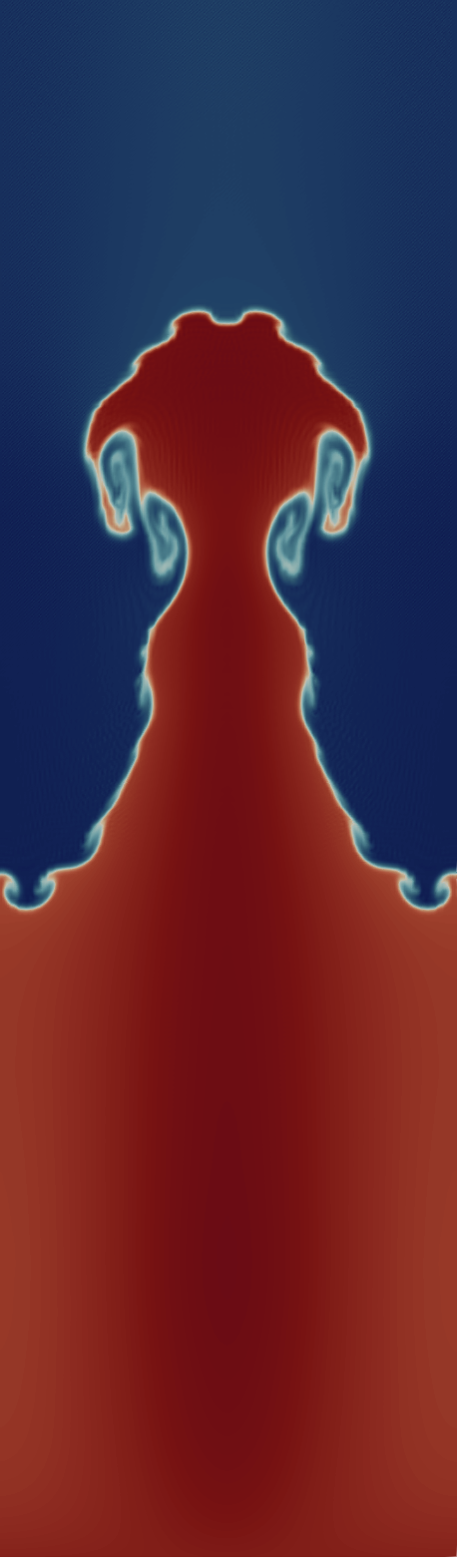}}
&
\resizebox{0.18\columnwidth}{!}{\includegraphics{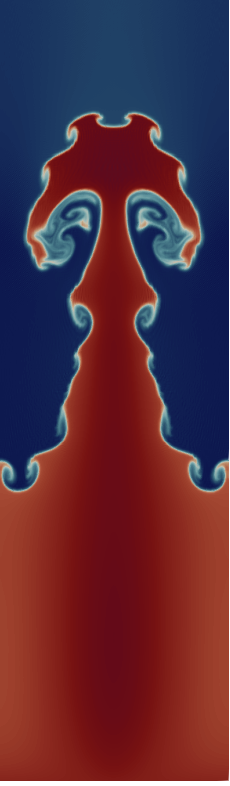}}
&
\resizebox{0.18\columnwidth}{!}{\includegraphics{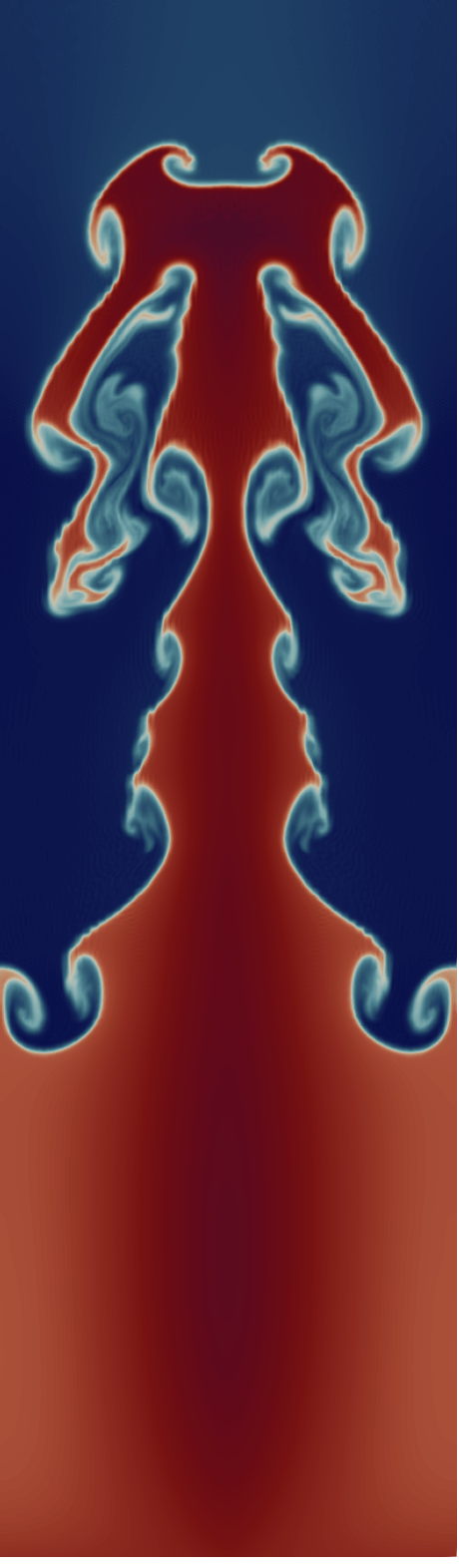}}
&
\resizebox{0.18\columnwidth}{!}{\includegraphics{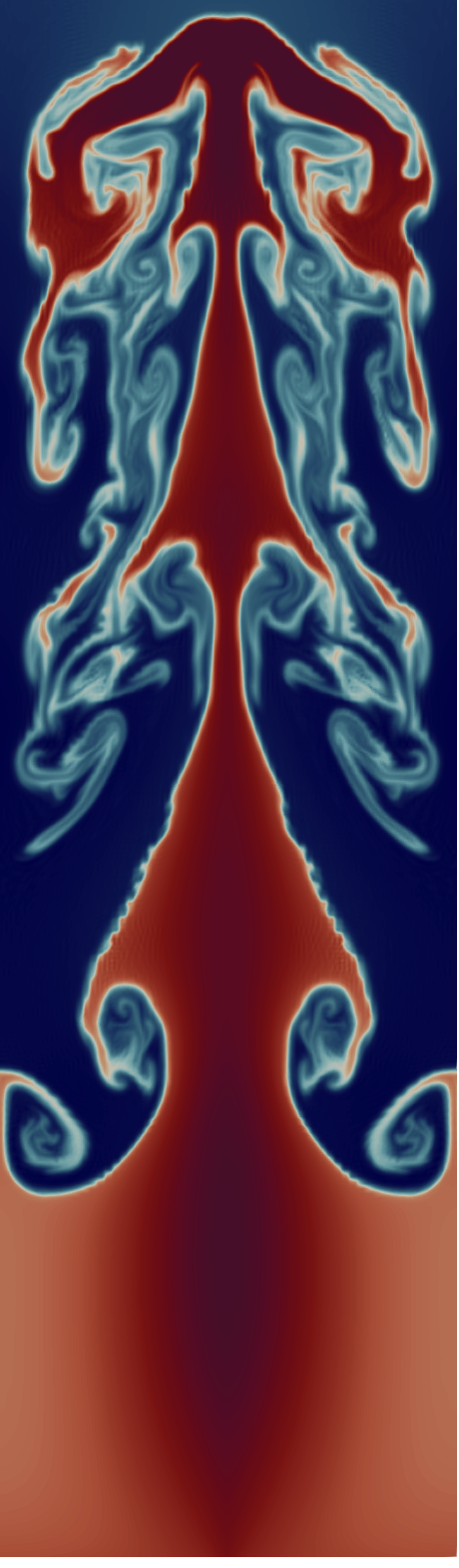}}\\
(a) $t = 1.5$ & (b) $t = 1.75$ & (c) $t = 2$ & (d) $t = 2.25$ & (e) $t =
2.5$
\end{tabular}
\caption{Density for the Rayleigh-Taylor instability on a $64 \times 256$
mesh\label{fig:mdrk.rayleigh.taylor}}
\end{figure}

\subsubsection{Timing studies of the MDRK method}
\correction{We will compare the proposed fourth order multiderivative Runge-Kutta method in Flux Reconstruction (MDRK-FR) framework with the standard Runge-Kutta Flux Reconstruction (RK-FR) using the five stage, fourth order Strong Stability Preserving Runge-Kutta (SSPRK) time discretization together with the blending limiter. The comparison is made using Gauss-Legendre solution points, with each scheme using its optimal CFL number in~\eqref{eq:mdrk.time.step.2d}. The CFL number of the MDRK scheme is taken from Table~\ref{tab:mdrk.cfl}. For the fourth order SSPRK-FR scheme, we performed a Fourier stability analysis to determine its optimal CFL number to be 0.215. The MDRK scheme uses the MUSCL-Hancock blending scheme, while the SSPRK scheme uses a blending limiter with second order MUSCL reconstruction on the subcells at each stage. The comparison has been made in Table~\ref{tab:performance} using our code \texttt{Tenkai.jl}~\cite{tenkai} for all the Cartesian tests in Section~\ref{sec:2d.euler}. Table~\ref{tab:performance} is not meant to be a comparison between different hardware vendors, especially since the processors compared are not of the same generation. These measurements are only to show the performance comparison of MDRK and SSPRK across different architectures. While there is a spread in the timings based on the hardware, we see that MDRK is about 20\%  to 40\% faster than Runge-Kutta scheme for most problems. A more realistic comparison should be done with MPI versions of these codes, but it is not part of the current work.
}

\begin{table}
\centering
\begin{tabular}{*{10}c}
\toprule
Test &  \multicolumn{3}{c}{Intel Xeon Gold 5320} & \multicolumn{3}{c}{AMD Threadripper} & \multicolumn{3}{c}{Apple M3 Pro}\\
\midrule
{} & SSPRK   & MDRK & Ratio   & SSPRK   & MDRK & Ratio & SSPRK & MDRK & Ratio\\
Double Mach &  14902 & 10857   & 1.37 & 9617 & 7690 & 1.25  & 6737 & 4907 & 1.37\\
Rotational flow &  1367 & 1102   & 1.24 & 971 & 793 & 1.22 & 705 & 533 & 1.32\\
Riemann problem &  12095  &  11319   & 1.06 & 8662 &  8384 & 1.03 & 6232 & 5307 & 1.17\\
Rayleigh Taylor &  9811  &  7989   & 1.22 & 7091 & 5271 & 1.34 & 4753 & 3368 &  1.41\\
\bottomrule
\end{tabular}
\caption{Wall Clock Time (WCT) in seconds of the numerical tests of 2-D Euler's equations for multidervative Runge-Kutta (MDRK) method and Strong Stability Preserving Runge-Kutta (SSPRK) methods in flux reconstruction framework.} \label{tab:performance}
\end{table}

\subsection{Extension to curvilinear grids}

\correction{The extension to curvilinear meshes to the MDRK scheme is made analogously to~\cite{babbar2024curved}. The time step size is computed using a CFL based formula~(88) from~\cite{babbar2024curved}. A more detailed description and an introduction of error based time stepping~\cite{babbar2024curved} for the MDRK scheme will be part of a future work.}

\subsubsection{Isentropic vortex test}

This is a test with exact solution taken from~{\cite{hennemann2021}} where the domain
is specified by the following transformation from $[0, 1]^2 \to \Omega$
\[ \bx (\xi, \eta) = \left(\begin{array}{c}
\xi L_x - A_x L_y \sin (2 \pi \eta)\\
\eta L_y + A_y L_x \sin (2 \pi \xi)
\end{array}\right) \]
which is a distortion of the square $[0, L_x] \times [0, L_y]$ with sine waves
of amplitudes $A_x, A_y$. Following~{\cite{hennemann2021}}, we choose length
$L_x = L_y = 0.1$ and amplitudes $A_x = A_y = 0.1$. The boundaries are set to
be periodic. A vortex with radius $R_v = 0.005$ is initialized in the curved
domain with center $(x_v, y_v) = (L_x / 2, L_y / 2)$. The gas constant is
taken to be $R_{\tmop{gas}} = 287.15$ and specific heat ratio $\gamma = 1.4$
as before. The free stream state is defined by the Mach number $M_0 = 0.5$,
temperature $T_0 = 300$, pressure $p_0 = 10^5$, velocity $u_0 = M_0
\sqrt{\gamma R_{\tmop{gas}} T_0}$ and density $\rho_0 =
\frac{p_0}{R_{\tmop{gas}} T_0}$. The initial condition $\uu_0$ is given by
\begin{equation*}
\begin{gathered}
(\rho, u, v, p) =  \left(
 \rho_0  \left( \frac{T}{T_0} \right)^{\frac{1}{\gamma - 1}},
 u_0  \left( 1 - \beta \frac{y - y_v}{R_v} e^{\frac{- r^2}{2}} \right),
u_0 \beta \frac{x - x_v}{R_v} e^{\frac{-r^2}{2}}, \rho (x, y) R_{\tmop{gas}} T
\right) \\
T(x, y) = T_0 - \frac{(u_0 \beta)^2}{2 C_p} e^{- r^2},\qquad r = \sqrt{(x - x_v)^2 + (y - y_v)^2} / R_v
\end{gathered}
\end{equation*}
where $C_p = R_{\tmop{gas}} \gamma / (\gamma - 1)$ is the heat capacity at constant pressure and $\beta = 0.2$ is the vortex strength. The vortex moves in the positive $x$ direction with speed $u_0$ so that the exact solution at time $t$ is $\uu(x,y,t) = \uu_0(x-u_0 t, y)$ where $\uu_0$ is extended outside $\Omega$ by periodicity. We simulate the propagation of the vortex for one time period $t_p = L_x / u_0$ and perform numerical convergence analysis for degree $N=3$ in Figure~\ref{fig:isentropic}b, showing optimal rates in grid size versus $L^2$ error norm for all the conserved variables.

\begin{figure}
\centering
\begin{tabular}{cc}
{\includegraphics[width=0.51\textwidth]{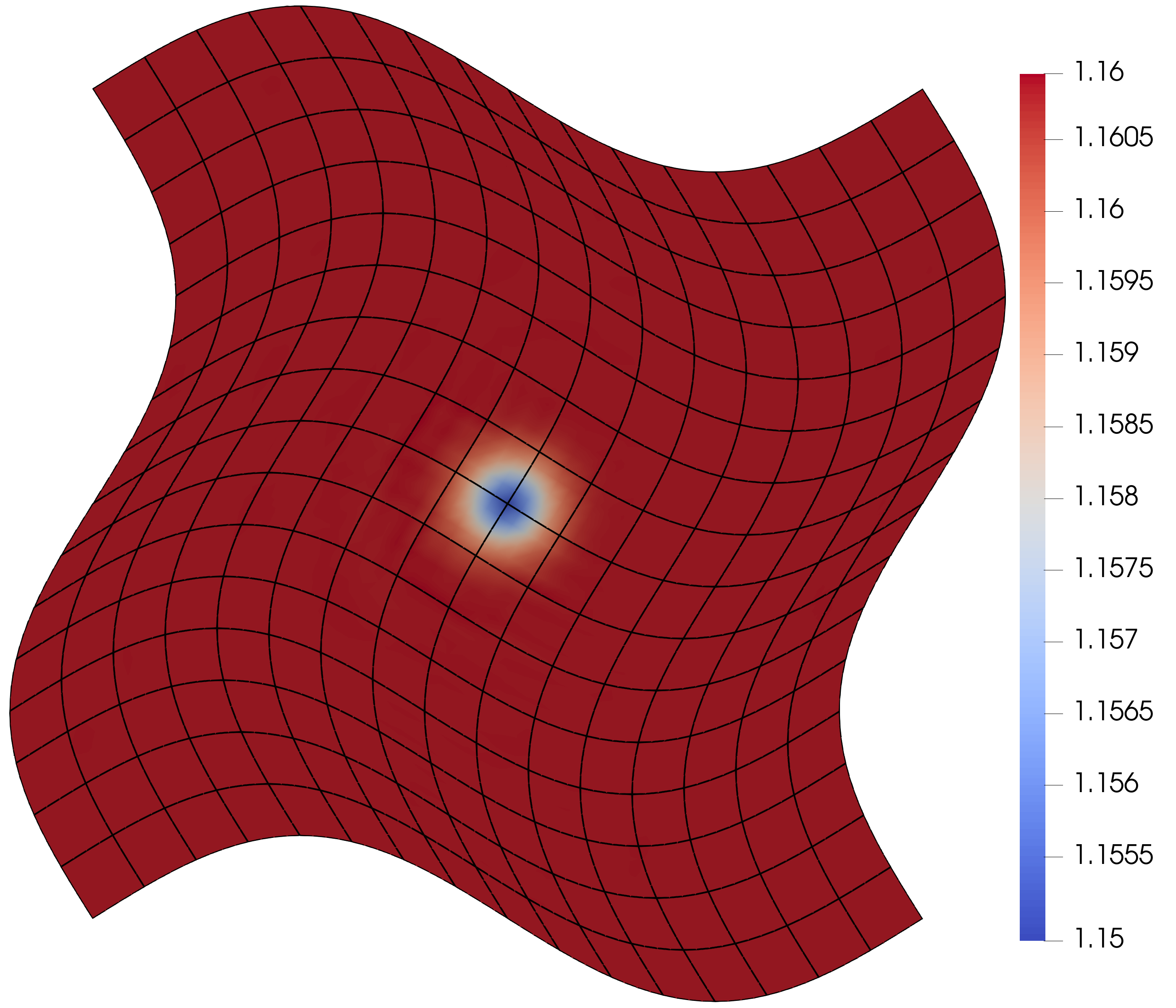}} & {\includegraphics[width=0.43\textwidth]{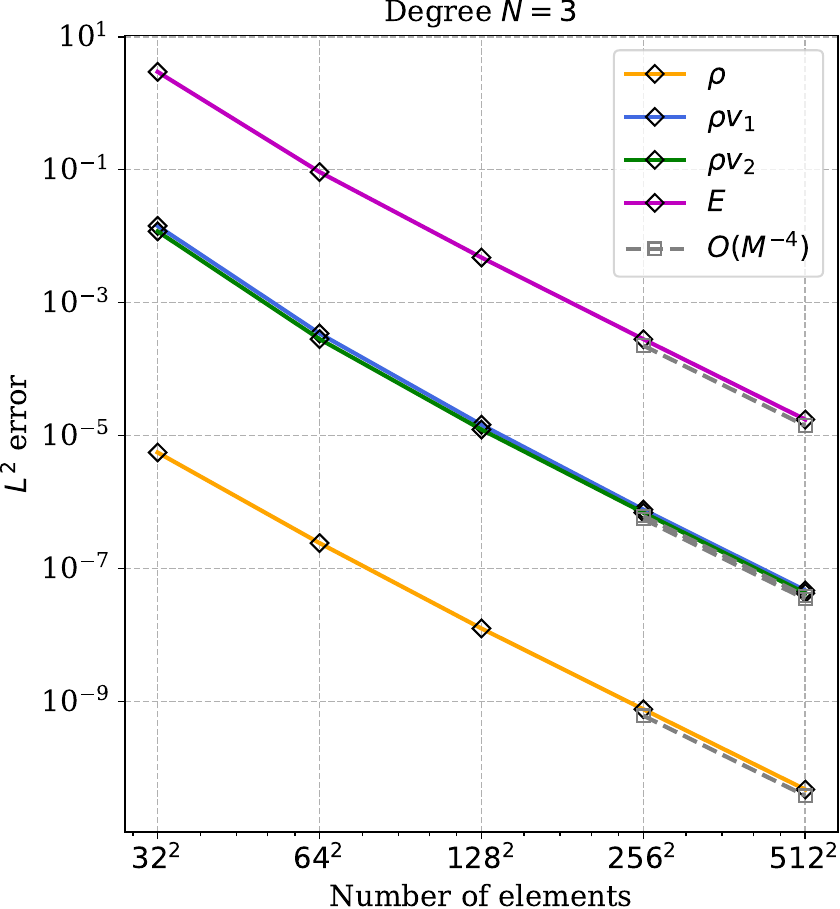}}\\
(a) & (b)
\end{tabular}
\caption{Convergence analysis for isentropic vortex problem with polynomial degree $N = 3$. (a) Density plot, (b) $L^2$ error norm of conserved variables\label{fig:isentropic}}
\end{figure}

\subsubsection{Supersonic flow over cylinder}

Supersonic flow over a cylinder is computed at a free stream Mach number of $3$ with the initial condition
\[
(\rho, u, v, p) = (1.4, 3, 0, 1)
\]
Solid wall boundary conditions are used on the cylinder and at the top and bottom boundaries. A bow shock forms ahead of the cylinder which reflects across the solid walls and interacts with the small vortices forming in the wake of the cylinder.  The setup of L{\"o}hner's smoothness indicator~\cite{lohner1987} is taken from an example of \tmverbatim{Trixi.jl}~{\cite{Ranocha2021}}
\[
(\tmverbatim{base{\_}level}, \tmverbatim{med{\_}level}, \tmverbatim{max{\_}level}) = (0, 3, 5), \qquad
(\tmverbatim{{med}{\_}{threshold}}, \tmverbatim{max{\_}threshold}) = (0.05, 0.1)
\]
where $\tmverbatim{base\_level} = 0$ refers to mesh in Figure~\ref{fig:supersonic.cylinder}a. The flow consists of a strong shock and thus the positivity limiter had to be used to enforce admissibility. The flow behind the cylinder is highly unsteady, with reflected shocks and vortices interacting continuously. The density profile of the numerical solution at $t = 5$ is shown in Figure~\ref{fig:supersonic.cylinder} with mesh and solution polynomial degree $N = 4$ using L{\"o}hner's indicator~\cite{lohner1987} for AMR. The AMR indicator is able to track the shocks and the vortex structures forming in the wake leading to mesh refinement in these areas. The initial mesh has 561 elements which steadily increases to peak at $\approx32000$ elements at $t=1$. The number of elements then slowly decrease to 26000 elements, then increasing again to reach $\approx 31000$ elements at $t=5$. The mesh is refined or coarsened once every 100 time steps. In order to capture the same effective refinement, a uniform mesh will require about 574464 elements.
\begin{figure}
\centering
\begin{tabular}{c}
{\includegraphics[width = 0.7\textwidth]{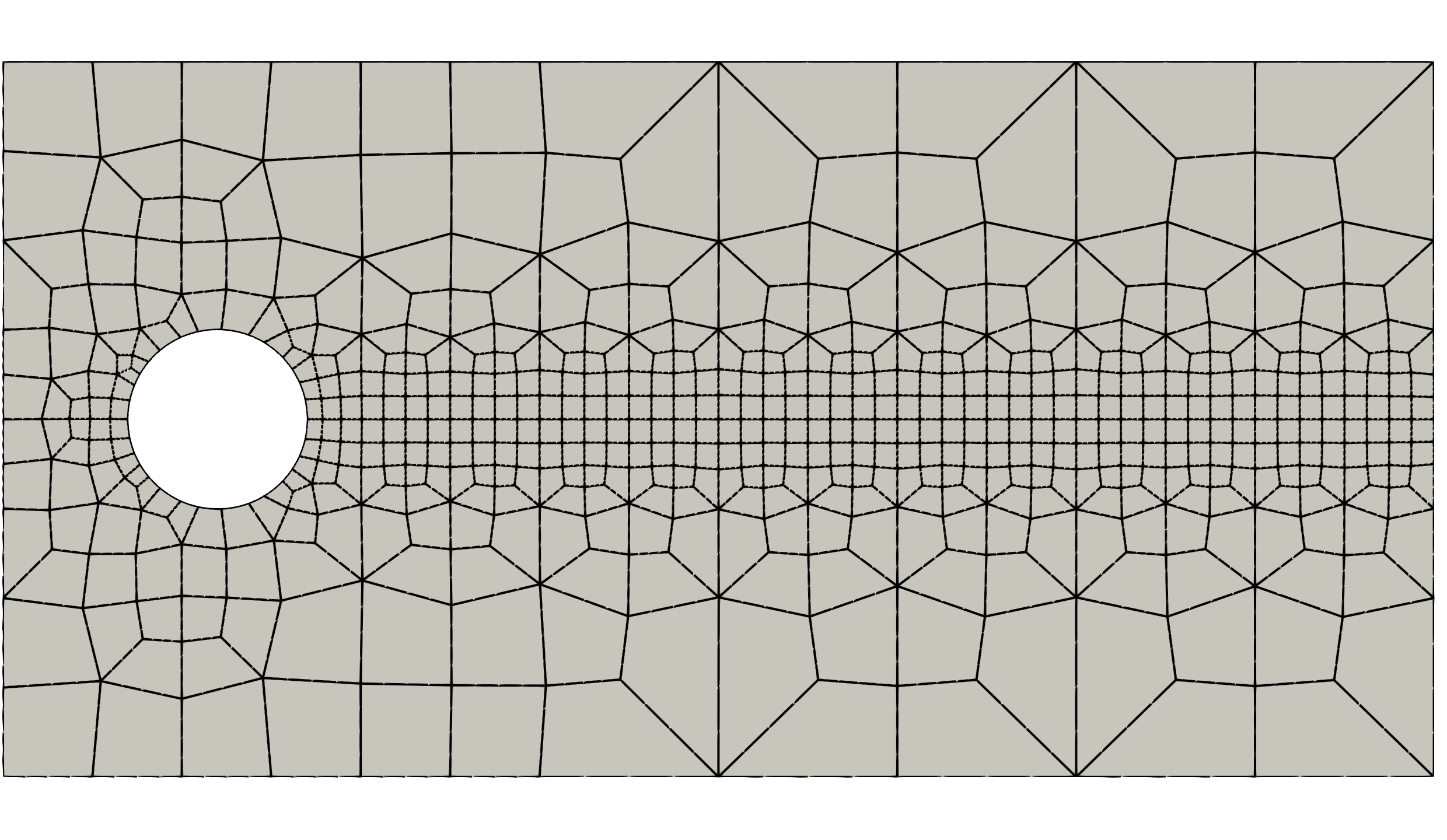}}\\
(a)\\
{\includegraphics[width = 0.7\textwidth]{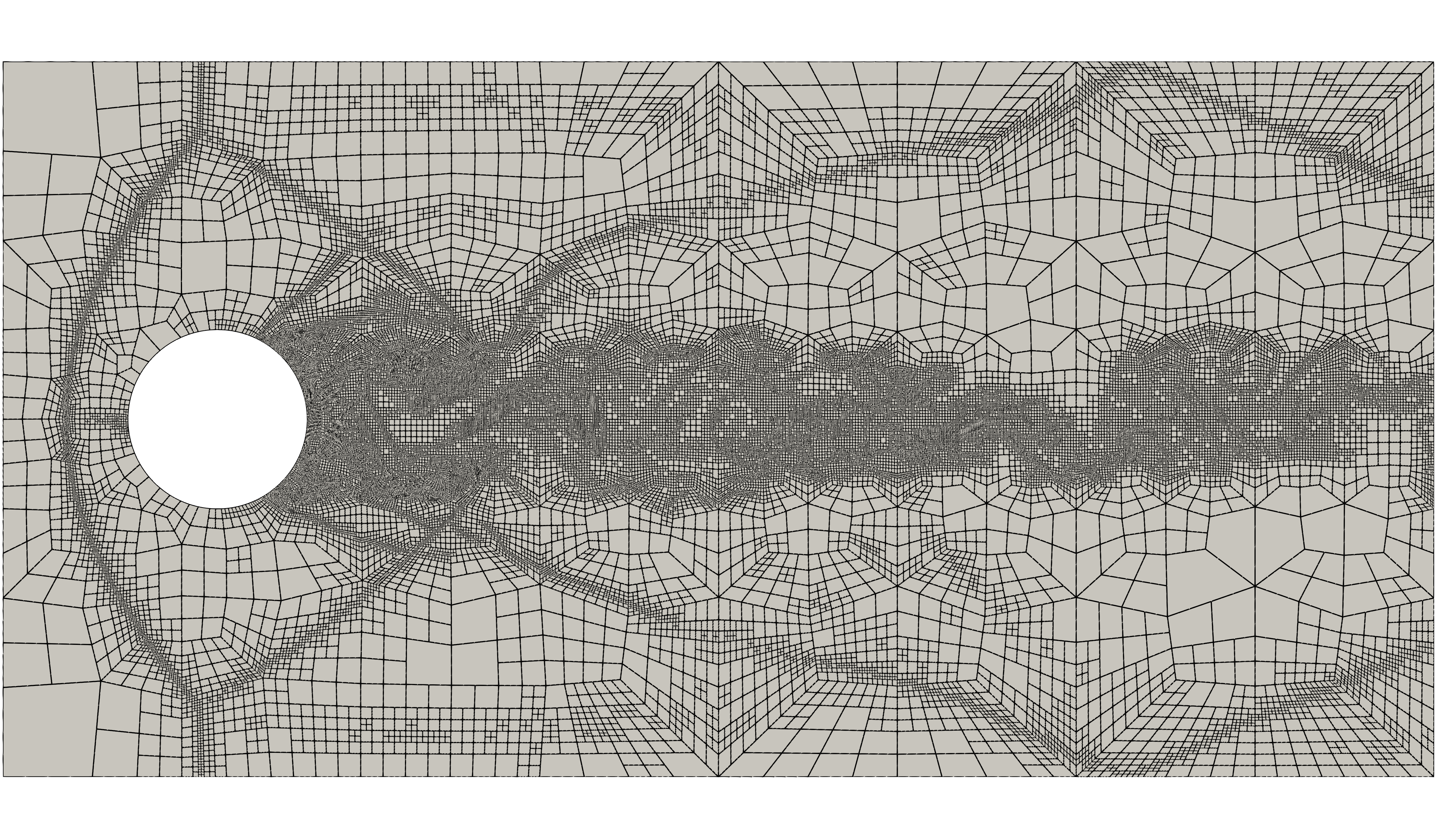}}\\
(b)\\
{\includegraphics[width = 0.7\textwidth]{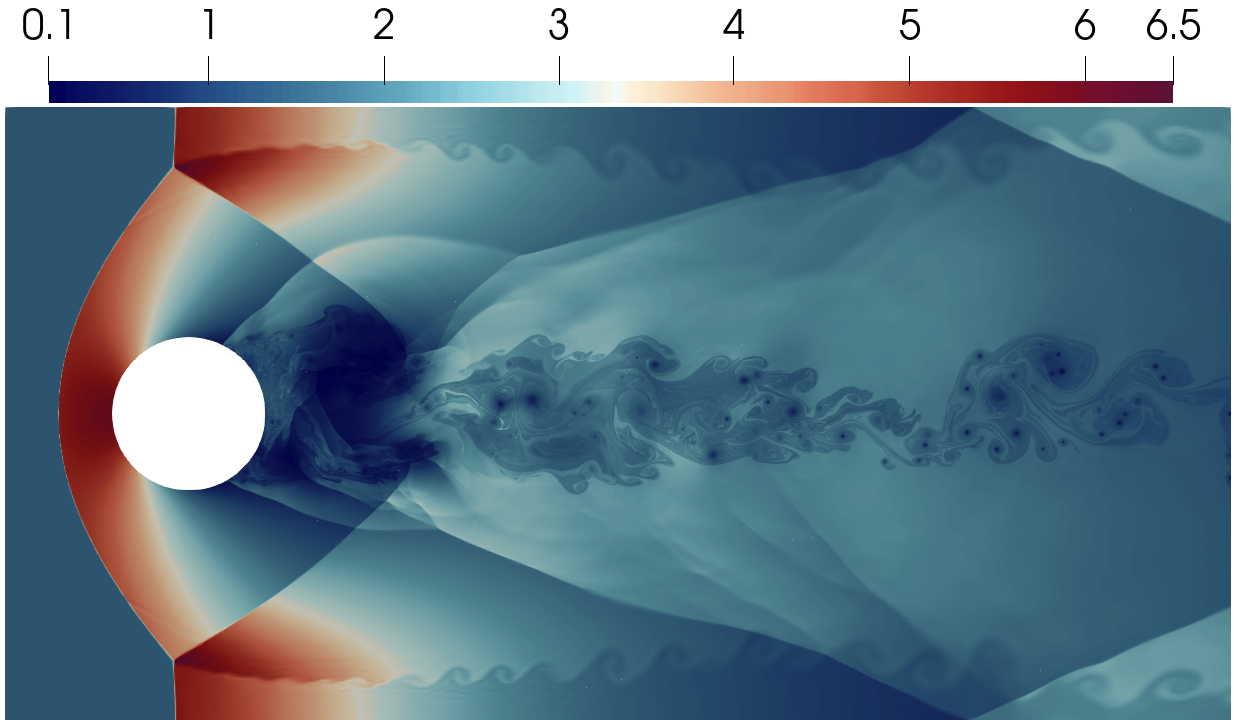}}\\
(c)
\end{tabular}
\caption{\label{fig:supersonic.cylinder}Mach 3 flow over cylinder using solution and mesh polynomial degree $N = 4$ at $t = 5$ (a) Initial mesh, (b) adaptively refined mesh at final time, (c) density plot at final time}
\end{figure}

\section{Summary and conclusions}\label{sec:mdrk.conclusion}

This work introduces the fourth order multiderivative Runge-Kutta (MDRK) scheme
of~{\cite{li2016}} in the conservative, quadrature free, \correction{Flux Reconstruction} framework to solve hyperbolic conservation laws. The idea is to cast each MDRK stage as an evolution involving a time average flux which is approximated by the Jacobian-free approximate Lax-Wendroff procedure. The D2 dissipation and {\evaluate} flux of~{\cite{babbar2022}} have been introduced for numerical flux computation, which enhance stable CFL numbers and accuracy for nonlinear problems, respectively. The stable CFL numbers are computed using Fourier stability analysis for two commonly used correction functions, $g_{\text{Radau}}$ and $g_2$, showing the improved CFL numbers. Numerical convergence studies for nonlinear problems were performed, which revealed that optimal convergence rates were only seen when using the {\evaluate} flux. The shock capturing blending scheme of~\cite{babbar2023admissibility} has also been introduced for the MDRK scheme applied at each stage. Following~\cite{babbar2023admissibility}, flux limiting is applied to each MDRK stage to obtain admissibility preservation in means. The scaling limiter of~\cite{Zhang2010b} is then applied to obtain admissibility preservation at the solution points. Along with being provably admissibility preserving, the scheme is more capable of capturing small scale structures through the use of Gauss-Legendre solution points and application of the MUSCL-Hancock scheme on the subcells. The claims are validated by numerical experiments for compressible Euler's equations with the modern test suite~\cite{Pan2016} for high order methods.

\appendix\section{Derivation of 2-stage, 4-th order scheme}\label{sec:mdrk.formal.accuracy}

We consider the system of time dependent equations
\[ \uu_t = \bL \left( \uu \right) \]
which relates to the hyperbolic conservation law~\eqref{eq:con.law} by
setting $\bL = - \pf \left( \uu \right)_x$. Now, we are analyzing the scheme
\begin{equation}
\begin{split}
\uus & = \uu^n + \mathLaplace ta_{2 \nocomma 1}  \bL \left( \uu^n
\right) + \mathLaplace t^2  \hat{a}_{21}  \bL_t \left( \uu^n \right)\\
\uu^{n + 1} & = \uu^n + \mathLaplace t \left( b_1  \bL \left( \uu^n
\right) + b_2  \bL \left( \uus \right) \right) + \mathLaplace t^2  \left(
\hat{b}_1 \partial_t  \tmmathbf{L} + \hat{b}_2 \partial_t  \tmmathbf{L}
\left( \uus \right) \right)
\end{split}
\end{equation}
where the coefficients $a_{ij}, b_i$ need to be determined to achieve fourth order accuracy. Further note that, we use the approximate Lax-Wendroff procedure (Section~\ref{sec:mdrk.alw})
to approximate $\partial_t  \bL \left( \uu^n \right), \partial_t  \bL \left(
\uus \right)$ to $O (\mathLaplace t^3)$ accuracy and thus we perform an error
analysis of an evolution performed as
\[ \uus = \uu^n + \mathLaplace ta_{2 \nocomma 1}  \bL \left( \uu^n \right) +
\mathLaplace t^2  \hat{a}_{21}  \bL_t \left( \uu^n \right) + O
(\mathLaplace t^5) \]
\begin{equation}
\uu^{n + 1} = \uu^n + \mathLaplace t \left( b_1  \bL \left( \uu^n \right) +
b_2  \bL \left( \uus \right) \right) + \mathLaplace t^2  \left( \hat{b}_1
\partial_t  \bL + \hat{b}_2 \partial_t  \bL \left( \uus \right) \right) + O
(\mathLaplace t^5) \label{eq:mdrk.unp.approximate}
\end{equation}
Now, note that
\begin{equation}
\begin{split}
\uu_{t \nocomma t} = \bL_{\uu}  \uu_t = \bL_{\uu}  \bL, \qquad \uu_{t
\nocomma t \nocomma t} = \bL_{\uu \nocomma \uu}  \uu_t^2 + \bL_{\uu}
\uu_{t \nocomma t} = \bL_{\uu \nocomma \uu}  \bL^2 + \bL_{\uu}^2  \bL \\
\uu_{t \nocomma t \nocomma t \nocomma t} = \bL_{\uu \nocomma \uu \nocomma
\uu}  \uu_t^3 + 3 \bL_{\uu \nocomma \uu}  \uu_t  \uu_{t \nocomma t} +
\bL_{\uu}  \uu_{t \nocomma t \nocomma t} = \bL_{\uu \nocomma \uu \nocomma
\uu}  \bL^3 + 4 \bL_{\uu \nocomma \uu}  \bL_{\uu}  \bL^2 + \bL_{\uu}^3
\bL
\end{split} \label{eq:mdrk.ut.ders}
\end{equation}
where $\bL_{\uu} = \partial \bL / \partial\uu$, etc. Starting from $\uu = \uu^n$, the exact solution satisfies
\begin{equation}
\uu^{n + 1} = \uu + \mathLaplace t \uu_t + \frac{\mathLaplace t^2}{2}
\uu_{t \nocomma t} + \frac{\mathLaplace t^3}{6}  \uu_{t \nocomma t \nocomma
t} + \frac{\mathLaplace t^4}{24}  \uu_{t \nocomma t \nocomma t \nocomma t} +
O (\mathLaplace t^5) \label{eq:mdrk.taylor.u.five}
\end{equation}
We note the following identities
\begin{gather*}
\correction{\bL_t}\left( \uu \right) = \bL_{\uu}  \uu_t,\\
\uus = \uu + \mathLaplace ta_{21}  \bL + \mathLaplace t^2  \hat{a}_{21}
\bL_{\uu}  \bL + O (\mathLaplace t^5)\\
\bL \left( \uus \right) = \bL + \bL_{\uu} \left( \uus - \uu \right) +
\frac{1}{2}  \bL_{\uu \nocomma \uu}  \left( \uus - \uu \right)^2 +
\frac{1}{6}  \bL_{\uu \nocomma \uu \nocomma \uu}  \left( \uus - \uu
\right)^3 + O (\mathLaplace t^4)\\
\bL_{\uu} \left( \uus \right) = \bL_{\uu} + \bL_{\uu \nocomma \uu}
\left( \uus - \uu \right) + \frac{1}{2}  \bL_{\uu \nocomma \uu \nocomma
\uu}  \left( \uus - \uu \right)^2 + O (\mathLaplace t^3)\\
\correction{\bL_t}\left( \uus \right) = \bL_{\uu} \left( \uus \right)  \bL
\left( \uus \right)
\end{gather*}
Now we will substitute these four equations into~\eqref{eq:mdrk.unp.approximate}
and use~\eqref{eq:mdrk.ut.ders} to obtain the update equation in terms of temporal
derivatives on $\uu$. Then, we compare with the Taylor's expansion of
$\uu$~\eqref{eq:mdrk.taylor.u.five} to get conditions for the respective orders of accuracy.

{\noindent}First order:
\begin{equation}
\label{eq:mdrk.c1} b_1 + b_2 = 1
\end{equation}
{\noindent}Second order:
\begin{equation}
\label{eq:mdrk.c2} b_2 a_{21} + \hat{b}_1 + \hat{b}_2 = \frac{1}{2}
\end{equation}
{\noindent}Third order:
\begin{eqnarray}
b_2 a_{21}^2 + 2 \hat{b}_2 a_{21} & = & \frac{1}{3}  \label{eq:mdrk.c3}\\
b_2  \hat{a}_{21} + \hat{b}_2 a_{21} & = & \frac{1}{6}  \label{eq:mdrk.c4}
\end{eqnarray}
{\noindent}Fourth order:
\begin{eqnarray}
b_2 a_{21}^3 + 3 \hat{b}_2 a_{21}^2 & = & \frac{1}{4}  \label{eq:mdrk.c5}\\
b_2 a_{21}  \hat{a}_{21} + \hat{b}_2 a_{21}^2 + \hat{b}_2  \hat{a}_{21} & =
& \frac{1}{8}  \label{eq:mdrk.c6}\\
\hat{b}_2 a_{21}^2 & = & \frac{1}{12}  \label{eq:mdrk.c7}\\
\hat{b}_2  \hat{a}_{21} & = & \frac{1}{24}  \label{eq:mdrk.c8}
\end{eqnarray}
From (\ref{eq:mdrk.c7}), (\ref{eq:mdrk.c8}) we get
\begin{equation}
\hat{a}_{21} = \frac{1}{2} a_{21}^2
\end{equation}
We then see that equations~(\ref{eq:mdrk.c3}), (\ref{eq:mdrk.c4}) become identical, and
equations~(\ref{eq:mdrk.c5}), (\ref{eq:mdrk.c6}) become identical. Simplifying the above
equations, we get five equations for the five unknown coefficients
\begin{eqnarray}
b_1 + b_2 & = & 1  \label{eq:mdrk.d1}\\
b_2 a_{21} + \hat{b}_1 + \hat{b}_2 & = & \frac{1}{2}  \label{eq:mdrk.d2}\\
b_2 a_{21}^2 + 2 \hat{b}_2 a_{21} & = & \frac{1}{3}  \label{eq:mdrk.d3}\\
b_2 a_{21}^3 + 3 \hat{b}_2 a_{21}^2 & = & \frac{1}{4}  \label{eq:mdrk.d4}\\
\hat{b}_2 a_{21}^2 & = & \frac{1}{12}  \label{eq:mdrk.d5}
\end{eqnarray}
Using \eqref{eq:mdrk.d5} in \eqref{eq:mdrk.d4} we get
\[ b_2 a_{21}^3 = 0 \]
The solution $a_{21} = 0$ does not satisfy \eqref{eq:mdrk.d3}, \eqref{eq:mdrk.d4}, hence let us choose
\[ b_2 = 0. \]
Then we get the unique solution for the coefficients
\[ b_1 = 1, \quad b_2 = 0, \quad \hat{b}_1 = \frac{1}{6}, \quad \hat{b}_2 =
\frac{1}{3}, \quad a_{21} = \frac{1}{2}, \quad \hat{a}_{21} = \frac{1}{8}
\]
These coefficients do give the scheme~\eqref{eq:mdrk.f2.defn} for which the two stage method is fourth order accurate.

\begin{remark} \label{rmk:why.accurate}
Equations~\eqref{eq:mdrk.unp.approximate} justify the choice of finite difference approximations made in Section~\ref{sec:mdrk.alw}. The time derivatives $\partial_t  \bL(\uu^n), \partial_t \bL(\uus)$ need to be approximated to atleast third order of accuracy for the one-step error to be $O(\Delta t^5)$ and the overall method to be fourth order accurate, which is why the third order accurate finite difference formulae in time are used in Section~\ref{sec:mdrk.alw}.
\end{remark}

\section{Treatment of source terms}\label{sec:source.terms}

In this section, we extend the multi-derivative Runge-Kutta (MDRK) Flux Reconstruction scheme to conservation laws with source terms
\begin{equation}
\uu_t + \pf \left( \uu \right)_x = \bss \label{eq:con.law.with.source}
\end{equation}
where $\bss = \bss(\uu, t, \bx)$. The idea is to apply the MDRK scheme of~\cite{li2016} to~\eqref{eq:con.law.with.source} as follows
\begin{equation*}
\begin{split}
\uus & = \uu^n - \frac{\mathLaplace t}{2} \partial_x \Fone + \frac{\Delta t}{2} \Sone
\\
\uu^{n + 1} & = \uu^n - \mathLaplace t \partial_x  \Ftwo + \Delta t \Stwo
\end{split}
\end{equation*}
where $\Fone, \Ftwo$ are the time averaged fluxes~(\ref{eq:Fone},~\ref{eq:Ftwo}) and $\Sone, \Stwo$ are the time averaged source terms satisfying
\begin{equation*}
\begin{split}
\Sone & \assign \bss \left( \bu^n \right) + \frac{1}{4}
\mathLaplace t \pdv{}{t}  \bss \left( \bu^n \right) \approx \frac{1}{\Delta t / 2} \int_{t^n}^{t^{\nph}} \bss \\
\Stwo & \assign \bss (\tmmathbf{u}^n) + \frac{1}{6}
\mathLaplace t \left( \pdv{}{t}  \bss (\tmmathbf{u}^n) + 2 \pdv{}{t}  \bss
(\tmmathbf{u}^{\ast}) \right) \approx \frac{1}{\Delta t}\int_{t^n}^{t^{n+1}} \bss
\end{split}
\end{equation*}
The approximate Lax-Wendroff procedure is used to approximate time average fluxes and source terms. It is an extension of the procedure in Section~\ref{sec:mdrk.alw} to solve~\eqref{eq:con.law.with.source} which we now describe.
\paragraph{First stage.}
The time averaged flux $\Fone$ and time averaged source term $\Sone$ are approximated as
\[
\vFone = \vf + \frac{1}{4}  \vf^{(1)}, \qquad \vSone = \vs + \frac{1}{4} \vs^{(1)}
\]
where
\begin{align*}
\vu^{(1)} & = - \frac{\Delta t}{\Delta x_e}  \vD \vf + \Delta t \bss \\
\vf^{(1)} & = \frac{1}{12}  \left[ - \pf ( \vu + 2 \vu^{(1)} )
+ 8 \pf ( \vu + \vu^{(1)} ) - 8 \pf ( \vu - \vu^{(1)}
) + \pf ( \vu - 2 \vu^{(1)} ) \right] \\
\vs^{(1)} & = \frac{1}{12}  \left[ - \bss ( \vu + 2 \vu^{(1)} )
+ 8 \bss ( \vu + \vu^{(1)} ) - 8 \bss ( \vu - \vu^{(1)}
) + \bss ( \vu - 2 \vu^{(1)} ) \right]
\end{align*}

\paragraph{Second stage.}
The time averaged flux $\Ftwo$ and the time averaged source term $\Stwo$ are approximated as
\[
\vFtwo = \vf + \frac{1}{6} (\vf^{(1)} + 2 \vfsone), \qquad \vStwo = \vs + \frac{1}{6} (\vs^{(1)} + 2 \vs^{*(1)})
\]
where
\begin{align*}
\vu^{*(1)} & = - \frac{\Delta t}{\Delta x_e}  \vD  \vfsone + \vs^{*(1)} \\
\vf^{*(1)} & = \frac{1}{12}  \left[ - \pf ( \vus + 2 \vusone
) + 8 \pf ( \vus + \vusone ) - 8 \pf ( \vus -
\vusone ) + \pf ( \vus - 2 \vusone ) \right] \\
\vs^{*(1)} & = \frac{1}{12}  \left[ - \bss ( \vus + 2 \vusone
) + 8 \bss ( \vus + \vusone ) - 8 \bss ( \vus -
\vusone ) + \bss ( \vus - 2 \vusone ) \right]
\end{align*}

\section*{Acknowledgments}
The work of Arpit Babbar and Praveen Chandrashekar is supported by the Department of Atomic Energy,  Government of India, under project no.~12-R\&D-   TFR-5.01-0520. \correction{We thank Jalil Khan for benchmarking our method on the AMD processor.}

\section*{Conflict of interest}
On behalf of all authors, the corresponding author states that there is no conflict of interest.

\bibliographystyle{siam}
\bibliography{references}

\end{document}